%% file: main_arxiv.tex
\documentclass{wiki}

\usepackage{color}
\usepackage[normalem]{ulem}

\begin{document}

\input{TitleAbstract}
\input{Introduction}

\input{Notation}
\input{BadTangentialSet}

\input{ChainRuleAlongCurve}
\input{StructureOfLevelSets}

\input{DisintegrationOfDiv}

\input{ChainRuleNI}
\input{RenormalizationNI}
\input{Acknowledgements}
\input{Appendix}

\bibliographystyle{unsrt}
\bibliography{bibliography}
\end{document}

%% file: TitleAbstract.tex
\title{Steady nearly incompressible vector fields in 2D: chain rule and renormalization}

\author{S. Bianchini, N.A. Gusev}

\maketitle

\begin{abstract}

Given bounded vector field $b\colon \RR^d \to \RR^d$, scalar field $u\colon \RR^d \to \RR$ and a smooth function $\beta \colon \RR \to \RR$ we study the characterization of the distribution $\div(\beta(u)b)$ in terms of $\div b$ and $\div(u b)$. In the case of $BV$ vector fields $b$ (and under some further assumptions) such characterization was obtained by L. Ambrosio, C. De Lellis and J. Mal\'y, up to an error term which is a measure concentrated on so-called \emph{tangential set} of $b$. We answer some questions posed in their paper concerning the properties of this term. In particular we construct a nearly incompressible $BV$ vector field $b$ and a bounded function $u$ for which this term is nonzero.

For steady nearly incompressible vector fields $b$ (and under some further assumptions) in case when $d=2$ we provide complete characterization of $\div(\beta(u) b)$ in terms of $\div b$ and $\div(u b)$. Our approach relies on the structure of level sets of Lipschitz functions on $\RR^2$ obtained by G. Alberti, S. Bianchini and G. Crippa. 

Extending our technique we obtain new sufficient conditions when any bounded weak solution $u$ of $\d_t u + b \cdot \nabla u=0$ is \emph{renormalized}, i.e. also solves $\d_t \beta(u) + b \cdot \nabla \beta(u)=0$ for any smooth function $\beta \colon\RR \to \RR$. As a consequence we obtain new uniqueness result for this equation.

\end{abstract}

\medskip

{\centerline{Preprint SISSA  43/2014/MATE}}

%% file: Introduction.tex
\section{Introduction}

\subsection{Transport equation and renormalization property}
The motivation of this paper comes from the problem of characterization of non-smooth vector fields $b\colon (0,T)\times \RR^n \to \RR^n$, $T>0$, for which the initial value problem for the transport equation
\begin{equation} \label{transport}
\d_t u + b \cdot \nabla u = 0
\end{equation}
has a unique bounded weak solution $u\colon (0,T)\times \RR^n \to \RR$ for any bounded initial data $u^\circ \colon \RR^n \to \RR$.

In \cite{dPL} this problem was studied in the class of vector fields $b$ which belong to Sobolev spaces. In particular it was proved that if $b \in L^1(0,T;W^{1,p})$, $p\ge 1$ and $\div b \in L^1(0,T;\Linf)$ then for any $u^\circ \in \Linf$ there exists unique weak solution $u\in L^\infty(0,T;\Linf)$ of \eqref{transport} such that $u|_{t=0}=u^\circ$.

We recall that $u\in L^\infty(0,T;\Linf)$ is called a \emph{weak solution} of \eqref{transport} if it satisfies $\d_t u + \div (ub) - u \div b = 0$ in sense of distributions $\ss D'((0,T)\times \RR^n)$. It is well-known (see e.g. \cite{dL_Notes}) that for any weak solution $u=u(t,x)$ of \eqref{transport} there exists unique function $\bar u \colon [0,T]\times \RR^n \to \RR$ such that
\begin{itemize}
\item $\bar u(t,\cdot) = u (t, \cdot)$ for a.e. $t\in (0,T)$ (consequently $\bar u$ also solves \eqref{transport})
\item the function $t\mapsto \bar u(t,\cdot)$ is weakly continuous on $[0,T]$ in the weak* topology of $\Linf$.
\end{itemize}
For brevity we will call such function $\bar u$ \emph{the weakly continuous version} of $u$.
In view of this definition the initial condition $u|_{t=0} = u^\circ$ for a weak solution of \eqref{transport} is understood in the following sense: we say that $u|_{t=0} = u^\circ$ if $\bar u(0, \cdot) = u^\circ(\cdot)$ a.e. in $\RR^n$.

Note that existence of weak solutions of \eqref{transport} is obtained by mollification of $b$, construction of approximate solutions using classical method of characteristics and passage to the limit using weak* compactness of $\Linf$. Uniqueness of weak solutions obtained as a consequence of the renormalization property:
\begin{definition}\label{def:renorm}
The vector field $b$ has the \emph{renormalization property} if for any bounded weak solution $u$ of \eqref{transport} and any smooth function $\beta\colon \RR \to \RR$ the function $\beta(u)$ is a weak solution of 
\begin{equation} \label{renorm}
\d_t \beta(u) + b \cdot \nabla \beta(u) = 0.
\end{equation}
\end{definition}
In the smooth setting the renormalization property simply follows from the classical chain rule.
However in the weak setting it is obtained in \cite{dPL} by a considerably more complicated argument based on the so-called \emph{commutator estimate}. We refer to \cite{CdL_ODE} for more details.

Later in \cite{AmbrosioBV} this theory was generalized for the class of vector fields $b$ with bounded variation. More precisely, existence and uniqueness of bounded weak solutions to the initial value problem for \eqref{transport} was proved when $b\in L^1(0,T;BV)$ and $\div b \in L^1(0,T; \Linf)$. The general strategy used in \cite{AmbrosioBV} is similar to the one used in \cite{dPL}, however the proof of the renormalization property is more difficult than in \cite{dPL} and involves convolutions with anisotropic kernels and Alberti's rank-one theorem \cite{RankOne}.

Note that both in \cite{dPL} and \cite{AmbrosioBV} the assumption $\div b \in L^1(0,T; \Linf)$ is used only in the proof of existence of weak solutions of \eqref{transport}, while the assumption $\div b \in L^1(0,T; L^1)$ is sufficient for uniqueness of weak solutions and for the renormalization property and \cite{AmbrosioBV}. Therefore one of the possible directions for developing further the theory of \eqref{transport} is to go beyond the assumption of absolute continuity of $\div b$ with respect to Lebesgue measure $\Le^d$.

The assumption $\div b \ll \Le^d$ is used in the definition of weak solution of \eqref{transport}. 
Therefore if $\div b$ is not absolutely continuous, we still have to impose some additional restriction on the divergence of $b$, which would allow us to give a meaning to \eqref{transport}. One of such restrictions is the assumption that $b$ is nearly incompressible.

\subsection{Nearly incompressible vector fields}
\begin{definition}\label{def:ni}
Let $I\subset \RR$ be an open interval and $\Omega \subset \RR^n$ be an open set.
A bounded vector field $b\colon I\times \Omega \to \RR^n$ is called \emph{nearly incompressible} with \emph{density} $\rho\colon I\times \Omega \to \RR$ if there exist real constants $C_2>C_1>0$ such that
\begin{equation*}
C_1 \le \rho \le C_2 \qquad \text{$\Le^{d+1}$-a.e. on $I\times \Omega$}
\end{equation*}
and
\begin{equation}\label{rho-continuity}
\d_t \rho + \div (\rho b) = 0 \qquad \text{in} \quad \ss D'(I\times \Omega).
\end{equation}
\end{definition}
Clearly divergence-free vector fields are nearly incompressible.
Moreover, vector fields with bounded divergence also belong to this class when $I$ is bounded.

For nearly incompressible vector fields we understand the transport equation \eqref{transport} in the following sense:
\begin{definition} \label{def:transport-ni}
Suppose that $b\colon I\times \Omega \to \RR^n$ is nearly incompressible with density $\rho\colon I\times \Omega \to \RR$. We say that a bounded function $u\colon I\times \Omega$ solves \eqref{transport} if it solves
\begin{equation} \label{transport-ni}
\d_t (\rho u) + \div (\rho u b) = 0 \qquad \text{in} \quad \ss D'(I\times \Omega).
\end{equation}
\end{definition}
We remark that in the smooth setting \eqref{transport-ni} follows from \eqref{transport} and \eqref{rho-continuity}.

Now the notion of \emph{renormalized solutions} of \eqref{transport} can be extended to the case of nearly incompressible vector fields: we only need to understand \eqref{renorm} in Definition~\ref{def:renorm} as in Definition~\ref{def:transport-ni}. Moreover, the initial condition $u|_{t=0}=u^\circ$ can be prescribed in the following sense: $\overline{\rho u}(0,\cdot) = \bar \rho(0,\cdot) \cdot u^\circ(\cdot)$.

Note that in general the notion of weak solution of \eqref{transport} for nearly incompressible vector field $b$ depends on its density $\rho$. However it was proved in \cite{dL_Notes} that if $b$ has renormalization property (with some fixed density $\rho$) then the notion of weak solution of \eqref{transport} is independent of the choice of $\rho$.

Nearly incompressible vector fields were introduced in connection with Keyfitz and Kranzer system \cite{dL_Notes}. They are also closely related to the compactness conjecture of Bressan \cite{BressanCauchyIllPosed}. In particular, it has been shown in \cite{ABdL} that this conjecture would follow from the following one:

\begin{conjecture}\label{conj:Bressan-renorm}
Any nearly incompressible vector field $b\in \Linf \cap BV_\rr{loc}(\RR \times \RR^n)$ with density $\rho \in \Linf (\RR \times \RR^n)$ has the renormalization property.
\end{conjecture}

As it was shown in \cite{AdLM}, the analysis of this conjecture can be reduced to the problem of characterization of so-called \emph{chain rule for the divergence operator}.

\subsection{Chain rule for the divergence operator}
Observe that if we let $d=n+1$ and introduce a vector field $B:=(1,b)$ then \eqref{rho-continuity} and \eqref{transport-ni} can be written respectively as
\begin{equation} \label{divrhob}
\div (\rho B) = 0
\end{equation}
and
\[
\div (\rho u B) = 0,
\]
where $\div$ denotes the divergence with respect to $(t,x)\in \RR^d$. If we introduce the function $h(\rho, w):=\beta(w/\rho) \cdot \rho$ then formally \eqref{renorm} (in sense of Definition~\ref{def:transport-ni}) can be written as
\[
\div(h(\rho, \rho u) B) = 0.
\]

The problem of computation of the distribution $\div(h(\rho, \rho u) B)$ can be reduced to the so-called \emph{chain rule problem for the divergence operator} \cite{AdLM}.
The latter can be formulated in an abstract way as follows: 
for given bounded vector field $B\colon \RR^d \to \RR^d$ and bounded scalar field $u \colon \RR^d \to \RR$ one has to characterize the relation between the distributions $\div B$, $\div (uB)$ and $\div(\beta(u) B)$ for any smooth function $\beta \colon \RR\to \RR$. In the smooth setting one can compute
\begin{equation} \label{chain-rule-smooth}
\div (\beta(u) B) = (\beta(u) - u \cdot \beta'(u) ) \div B + \beta'(u) \div (uB).
\end{equation}

In view of applications for nearly incompressible BV vector fields in this paper we consider the following
concrete setting of 
the chain rule problem:
given an open set $\Omega \subset \RR^d$ and Radon measures $\lambda$ and $\mu$ on $\Omega$ and bounded functions $B\colon \Omega \to \RR^d$ and $u \colon \Omega \to \RR$ such that
\begin{gather}
\div B = \lambda
\qquad
\text{in }\ss D'(\Omega), \label{divB} 
\\
\quad
\div (uB) = \mu
\qquad
\text{in }\ss D'(\Omega) \label{divuB}
\end{gather}
one has to characterize the distribution $\div (\beta(u) B)$.

For $B\in BV$ this problem was studied in \cite{AdLM}, where it was proved that there exists a Radon measure $\nu$ such that
\begin{equation}
\div(\beta(u) B) = \nu
\qquad
\text{in }\ss D'(\Omega). \label{divbeta}
\end{equation}
Moreover it was proved that $\nu \ll |\lambda| + |\mu|$. The measure $\nu$ (and the measures $\mu$ and $\lambda$) was decomposed into absolutely continuous part $\nu^a$, Cantor part $\nu^c$ and jump part $\nu^j$ (in a similar way as the derivatives of $BV$ functions, see \cite{AdLM}) and characterizations of these measures were studied. The following results were obtained in \cite{AdLM}:
\begin{itemize}
\item the chain rule for the absolutely continuous part is similar to \eqref{chain-rule-smooth}:
\begin{equation} \label{chain-rule-a}
\nu^a = (\beta(u) - u \cdot \beta'(u) ) \lambda^a + \beta'(u) \mu^a.
\end{equation}
\item the chain rule for the jump part is given by
\begin{equation} \label{chain-rule-j}
\nu^j = \sb{(\Tr^+ B) \beta\rb{\frac{\Tr^+ (uB)}{\Tr^+(B)}} - (\Tr^-B) \beta\rb{\frac{\Tr^- (uB)}{\Tr^-(B)}}} \Ha^{d-1} \rest \Sigma
\end{equation}
where $\Tr^\pm(U) = \Tr^\pm_\Sigma(U)$ denotes the normal trace operator \cite{ambcriman04} on $\Sigma$ for a bounded vector field $U$ whose divergence is represented by Radon measure, and $\Sigma$ denotes the countably rectifiable set on which $\lambda^j$ and $\mu^j$ are concentrated.
\item the chain rule for the Cantor part is given by
\begin{equation} \label{chain-rule-c}
\nu^c = (\beta(\wave u) - \wave u \cdot \beta'( \wave u) ) \lambda^c \rest (\Omega \setminus S_u) + \beta'(\wave u) \mu^c \rest (\Omega \setminus S_u) + \sigma.
\end{equation}
where $\wave u(x)$ denotes the $L^1$ approximate limit of $u$ at point $x$, $S_u$ denotes the set of points where $u$ does not have $L^1$ approximate limit and $\sigma$ is some Radon measure concentrated on $S_u$ (i.e. $|\sigma|(\Omega \setminus S_u)=0$) and absolutely continuous with respect to $\lambda^c$ and $\mu^c$.
\end{itemize}

Therefore in order to have a complete solution of the chain rule problem one has to characterize the ``error term'' $\sigma$. In \cite{AdLM} (see also \cite{dL_Notes}) it was shown that the inclusion 
\begin{equation} \label{SuTB}
S_u \subset T_B
\end{equation}
holds up to $|D^c B|$-negligible set, where $D^c B$ is the Cantor part of the derivative of $B$ and $T_B$ is so-called \emph{tangential set of $B$}:

\begin{definition}[see \cite{AdLM}]\label{def:ts}
Suppose that $b\colon \Omega \to \RR^d$ is a bounded vector field with bounded variation.
Consider the Borel set $E$ of all points $x\in \Omega$ such that
\begin{enumerate}
\item there exists finite
\begin{equation*}
M(x):=\lim_{r\to 0} \frac{Db(B_r(x))}{|Db|(B_r(x))}
\end{equation*}
(in our notation $Db$ is a matrix with the measure-valued components $(Db)_{ij}=D_j b_i$);
\item the approximate $L^1$-limit $\tilde b(x)$ of $b$ at $x$ exists.
\end{enumerate}
Then we call \emph{tangential set} of $b$ (in $\Omega$) the set
\begin{equation*}
T_B := \setof{x\in E}{M(x) \cdot \tilde b(x) = 0}.
\end{equation*}
\end{definition}
\begin{remark} \label{rem:ts}
From the definition of the tangential set one can see that
\begin{enumerate}
\item If $b$ is constant on $\RR^2$ then $T_B = \emptyset$.
\item If $b$ is smooth, then (up to a $|Db|$-negligible set)
\begin{equation*}
T_B = \setof{x\in \Omega}{\nabla \otimes b(x) \ne 0, \;\; \ab{b(x), \nabla} b(x) = 0},
\end{equation*}
i.e. the tangential set
is the set of all points where the derivative of $b$ does not vanish and the derivative of $b$ in the direction $b(x)$ is zero. (Here $\nabla \otimes b$ is the matrix with the real-valued components $(\nabla \otimes b)_{ij} = \d_j b_i$.)
\end{enumerate}
\end{remark}

Having in mind applications for nearly incompressible vector fields, one is particularly interested in the case when
$\mu =0$ (or $\mu$ is absolutely continuous).
In view of this and \eqref{SuTB} the following question was posed in \cite{AdLM}:
\begin{question} \label{q1}
Does the Cantor part $|\lambda^c|$ vanish on the tangential set $T_B$ for any $B\in \Linf \cap BV$?
\end{question}
In \cite{AdLM} the authors constructed a counterexample of a vector field $B\in (\Linf \cap BV)(\RR^2)$ for which $|\lambda^c|(T_B)>0$,
so the answer to Question~\ref{q1} is negative.
In connection with this the following questions were posed in \cite{AdLM}:
\begin{question} \label{q2}
Let $B\in (\Linf \cap BV)(\Omega)$. Under which conditions the Cantor part of the divergence $|\lambda^c|$ vanishes on the tangential set $T_B$?
\end{question}
\begin{question} \label{q3}
Let $B\in (\Linf \cap BV)(\Omega)$ and let $\rho \in \Linf(\Omega)$ be such that $\rho \ge C >0$ a.e. in $\Omega$ and
\eqref{divrhob} holds.
Is it true that $|\lambda^c|(T_B) = 0$?
\end{question}
In this paper we prove existence of a vector field $B$ which provides a negative answer to Question~\ref{q3}.
Moreover, for this $B$ we construct a scalar field $u$ such that the term $\sigma$ in \eqref{chain-rule-c} is nonzero.

In dimension two we provide a solution to the chain rule problem for the class of \emph{steady nearly incompressible} vector fields,
which is strictly bigger than the one considered in Question~\ref{q3}.
In particular our results are applicable to the vector field $B$ which answers Question~\ref{q3} and allow us to characterize the term $\sigma$.

\subsection{Steady nearly incompressible vector fields}
\begin{definition}\label{def:nis}
Suppose that $B\colon \Omega \to \RR^d$ is a bounded vector field and $\rho \colon \Omega \to \RR$ is a scalar field
such that $C_1 \le \rho \le C_2$ a.e. in $\Omega$ for some strictly positive constants $C_1$ and $C_2$.
The vector field $B$ is called \emph{steady nearly incompressible with steady density $\rho$}
if \eqref{divrhob} holds in $\ss D'(\Omega)$.
\end{definition}
Clearly steady nearly incompressible vector fields are a subclass of nearly incompressible ones.
However not every nearly incompressible vector field $b=b(t,x)$ which does not depend on time $t$ is \emph{steady nearly incompressible} in sense of Definition~\ref{def:nis} (for instance consider $b(x)=-x$, even in one-dimensional case). Nevertheless the following holds:
\begin{remark}
If $b\colon \Omega \to \RR^d$ is nearly incompressible with density $\rho\colon \RR\times\Omega \to \RR$, then $b$ is steady nearly incompressible.
\end{remark}
Indeed, in this case the functions $\rho_T(x):=\frac{1}{2T}\int_{-T}^{T} \rho(\tau,x) \, d\tau$ for a.e. $T>0$ solve $\div(\rho_T b) = -\frac{1}{2T}\rho(T,\cdot) + \frac{1}{2T}\rho(-T,\cdot)$ in $\ss D'(\Omega)$. Moreover, $C_1 \le \rho_T \le C_2$ a.e. for any $T>0$. Consider an appropriate sequence $\{T_n\}_{n\in \NN}$ converging to $+\infty$ as $n\to \infty$. Using sequential weak* compactness of $\Linf$ we extract a subsequence (which we do not relabel) such that $\rho_{T_n} \weakstarto r$ in $\Linf$. Passing to the limit as $n\to \infty$ we conclude that $C_1 \le r \le C_2$ a.e. and $\div({rb})=0$ in $\ss D'(\Omega)$, i.e. $r$ is the desired steady density.

In this paper we provide a solution to the chain problem stated above, assuming in addition that $d=2$, $\Omega$ is simply connected and $B$ is steady nearly incompressible and has bounded support. Namely we prove that there exists a Radon measure $\nu$ on $\Omega$
such that \eqref{divbeta} holds and $\nu \ll |\lambda| + |\mu|$.
More precisely we show that there exist bounded Borel functions $u^+$ and $u^-$ (depending on $u$) such that
\eqref{divbeta} holds with $\nu$ given by

\begin{equation} \label{nu-char}
\nu = f_0(u^+,u^-) \lambda + f_1(u^+,u^-) \mu
\end{equation}
where
\begin{subequations}
 \begin{equation} \label{f0}
 f_0(u^+,u^-) := 
  \begin{cases}
   \displaystyle
   \frac{u^+ \beta(u^-)  -  u^- \beta(u^+)}{u^+ - u^-} & \quad \text{if} \quad u^+ \ne u^-, \\
   \beta(\wave u) - \wave u \cdot \beta'( \wave u) & \quad \text{if} \quad u^+ = u^- = :\wave u,
  \end{cases}
 \end{equation}
 
 \begin{equation} \label{f1}
 f_1(u^+,u^-) := 
  \begin{cases}
   \displaystyle
   \frac{\beta(u^+) - \beta(u^-)}{u^+ - u^-} & \quad \text{if} \quad u^+ \ne u^-, \\
   \beta'(\wave u) & \quad \text{if} \quad u^+ = u^- = :\wave u,
  \end{cases}
 \end{equation}
\end{subequations}

We construct the functions $u^\pm$ as the traces of $u$ along the trajectories of $\rho B$. 
In order to formulate this statement more precisely let us recall the following standard definition:
a simple (possibly closed) Lipschitz curve $C \subset \Omega$ is called a \emph{trajectory} (or an \emph{integral curve}) of a bounded Borel vector field $B\colon \Omega \to \RR^d$ if there exists a connected subset $I\subset \RR$
and a Lipschitz parametrization $\gamma \colon I \to \Omega$ of $C$ such that
\begin{equation} \label{ode}
\d_t \gamma(t) = B(\gamma(t)) \qquad \text{for a.e.} \quad  t\in I.
\end{equation}
Redefining, if necessary, the functions $\rho$ and $B$ on a negligible subset, we can assume that these functions are Borel.
We prove that there exists a \emph{disjoint family $\gg F$ of trajectories} of $\rho B$ such that
\begin{itemize}
\item the set $F:= \cup_{C \in \gg F} C$ is Borel and coincides with $\Omega$,
 up to a $(|B|\Le^2+|\lambda|+|\mu|)$-negligible subset;
\item $u$ has bounded variation along the elements of $\gg F$,
i.e. $u\circ \gamma$ belongs to $BV(I)$ for any $C \in \gg F$ with Lipschitz parametrization $\gamma\colon I \to \Omega$.
Let $(u\circ \gamma)^+$ and $(u\circ \gamma)^-$ denote the right-continuous and left-continuous representations of $u\circ \gamma$. Then for any $x\in F$ we can define the 
\emph{traces $u^\pm(x)$ of $u$ along trajectories of $\rho B$}:
\begin{equation*}
u^\pm(\gamma(t)):= (u\circ \gamma)^\pm (t)
\end{equation*}
where $\gamma\colon I\to \Omega$ is the Lipschitz parametrization of the unique element of $C\in \gg F$ such $x\in C$,
and $t$ is the unique point of $I$ such that $\gamma(t)=x$.
\end{itemize}

Our methods are based on the observation that if $\Omega$ is a simply connected then in view of \eqref{divrhob} there exists a Lipschitz function $H\colon \Omega \to \RR$ such that 
\begin{equation}\label{rhob=nablaperpH}
\rho B = \nabla^\perp H
\quad\text{a.e. in } \Omega.
\end{equation}
Then we use results of \cite{ABC1} and \cite{ABC2} on the structure of the level sets of Lipschitz functions on $\RR^2$ to identify the trajectories of $B$ with the connected components of the level sets of $H$.

Finally, we prove that if $B$ is steady nearly incompressible and has bounded variation, then the function $H$ from \eqref{rhob=nablaperpH} has the \emph{weak Sard property}, introduced in \cite{ABC1}, i.e. $H_\# \Le^2 \rest \{\nabla H = 0\} \perp \Le^1$. As a consequence we prove that steady nearly incompressible vector fields have the renormalization property. This is a partial answer to Conjecture~\ref{conj:Bressan-renorm}.

%% file: Notation.tex
\section{Notation} \label{sec:Notation}
Throughout the paper functions and sets are tacitly assumed Borel measurable.
The measures are assumed to be defined on the appropriate Borel $\sigma$-algebras.
We will use the following notation:

\begin{tabular}[l]{p{50pt} p{300pt}}
   $\1_E$ & characteristic function of set $E$;
\\ $\Le^d$ & Lebesgue measure on $\RR^d$;
\\ $\Ha^k$ & $k$-dimensional Hausdorff measure;
\\ $|\mu|$ & nonnegative measure associated to a real- or vector-valued measure $\mu$ (total variation);
\\ $\mu \rest E$ & restriction of the measure $\mu$ on the set $E$;
\\ $a^\perp$ & the vector with the components $(-a_2, a_1)$, assuming $a=(a_1,a_2)$;
\\ $\ss{D}(M)$ & the space of test functions a smooth manifold $M$;
\\ $\ss{D}'(M)$ & the space of distributions on $M$;
\\ $\overline{A}$ & the closure of $A\subset \RR^d$;
\\ $\inte A$ & interior of $A \subset \RR^d$;
\\ $\ff{Conn}(E)$ & the set of all connected components of $E$;
\\ $\ff{Conn^*}(E)$ & the set of all connected components of $E$ which are not single points (hence they have positive $\Ha^1$ measure, see e.g. \cite{Falconer})
\\ $E^*$ & the union of all elements of $\ff{Conn^*}(E)$;
\\ $d_{\cc H}$ & Hausdorff metric;
\end{tabular}

We recall that a set $E\subset \RR^d$ is called \emph{connected} if it cannot be written as a union of two disjoint sets which are closed in the induced topology of $E$.
If $E\subset \RR^d$ and $x\in E$ then by \emph{connected component of $x$ in $E$} we mean the union of all connected subsets of $E$ which contain $x$.
Connected components of $E$ are connected and closed in $E$ (i.e. closed in the induced topology of $E$).

A measure $\mu$ on $X$ is said to be \emph{concentrated} on a set $E\subset X$ if $\mu(X \setminus E) =0$. For measures $\mu$ and $\nu$ we write $\mu \ll \nu$ if $\mu$ is absolutely continuous with respect to $\nu$ and we write $\mu \perp \nu$ if $\mu$ and $\nu$ are mutually singular.

Given metric spaces $X$, $Y$, a function $f\colon X\to Y$ and a measure $\mu$ on $X$ we denote by $f_\# \mu$ the pushforward (or image) of $\mu$ under the map $f$, i.e. $f_\#\mu$ is a measure on $Y$ such that $(f_\# \mu)(E) = \mu(f^{-1}(E))$ for any Borel set $E\subset Y$.
By $\int f \,d\mu$ or $\int f(x) \,d\mu(x)$ we denote the integral of a function $f\colon X \to Y$ with respect to measure $\mu$ on $X$.

Let $\Omega\subset \RR^d$ be an open set.
We say that $C\subset \Omega$ is a \emph{simple curve} if $C$ is an image of a nonempty connected subset $I\subset \RR$ under a continuous map $\gamma \colon \overline{I}\to \overline{\Omega}$ which is injective on $I$. Any such map $\gamma$ we will call a \emph{parametrization of $C$}.
If for some parametrization we have $\overline{I}=[0,\ell]$ and $\gamma(0)=\gamma(\ell) \in C$ then we say that $C$ is \emph{closed}. 

We say that a simple curve is \emph{Lipschitz} if it has finite $\Ha^1$ measure. 
By \emph{length} of a Lipschitz curve $C$ we mean $\Ha^1(C)$.
It is known that for any simple Lipschitz curve there exists a parametrization $\gamma$ which is a Lipschitz function (see e.g. \cite{Falconer}).

We will say that a curve $C$ is \emph{$\Omega$-closed} if either $C$ is closed or $\gamma(0)\in \d\Omega$ and $\gamma(\ell)\in \d \Omega$. For $\Omega$-closed curves we introduce the \emph{domain of $\gamma$} by
\begin{equation} \label{param-dom}
I_\gamma := 
\begin{cases}
(0,\ell) & \text{if } \gamma(0),\gamma(\ell)\in \d \Omega, \\
\RR/(\ell \ZZ) & \text{if } \gamma(0)=\gamma(\ell)\in C.
\end{cases}
\end{equation}
Clearly $I_\gamma$ is a metric space with distance given by $d(t,t') = |t-t'| \; \rr{mod\,} \ell$.
A Lipschitz parametrization $\gamma\colon [0,\ell]\to \overline{\Omega}$ of an $\Omega$-closed simple curve can always be viewed as an injective Lipschitz function from $I_\gamma$ to $\Omega$.
By $\Le$ we will always denote the Lebesgue measure on $I_\gamma$.

Unless we specify the measure explicitly, by ``almost everywhere'' we mean a.e. with respect to Lebesgue measure.

Given a function $f\colon \Omega \to \RR$ we denote by $1/f$ the function $1/(\1_{\{f=0\}} + f)$.

By \emph{Radon measure} we will for brevity mean \emph{finite} Radon measure.

Suppose that $\Omega \subset \RR^d$ is an open set, $\mu$ is a Radon measure and a vector field $V\in [\Linf(\Omega)]^d$ satisfies $\div V = \mu$ in $\ss D'(\Omega)$ (i.e. in sense of distributions). Then $\mu$ vanishes on $\Ha^{d-1}$-negligible sets (see Proposition~6 from \cite{AdLM}) and therefore can be decomposed as 
\begin{equation*}
\mu = \mu^a + \mu^j + \mu^c
\end{equation*}
where $\mu^a \ll \Le^d$, $\mu^j + \mu^c \perp \Le^d$, $\mu^c(E)=0$ for any set $E$ with $\Ha^{d-1}(E)<\infty$
and $\mu^j \ll \Ha^{d-1}$ is concentrated on some $\sigma$-finite (with respect to $\Ha^{d-1}$) set.
Such decomposition exists and is unique (see Proposition~5 from \cite{AdLM}). 
The measures $\mu^a$, $\mu^j$ and $\mu^c$ are called respectively
the \emph{absolutely continuous}, \emph{jump} and \emph{Cantor} parts of $\mu$.
Note that this decomposition is similar to the decomposition of the derivatives of BV functions \cite{AFP}.

%% file: BadTangentialSet.tex
\section{New jump-type discontinuities along trajectories}
In this section we prove existence of a vector field $B$ which provides a (negative) answer to Question~\ref{q3}.
For this vector field $B$ we also construct a scalar field $u$ such that the term $\sigma$ in \eqref{chain-rule-c} is a nontrivial measure concentrated on some set $S$. We study the discontinuities of $u$ on the set $S$ and it turns out that $u$ has jump-type discontinuities along the trajectories of the vector field $B$.

Our construction of the vector field $B$ is inspired by the counterexample from \cite{AdLM} (see Proposition~2), which is the negative answer to Question~\ref{q1}.
While the original presentation of that counterexample is completely analytic, we will construct $B$ in a substantially more geometrical way.

\subsection{Compressible and incompressible cells}
Before constructing the desired vector field $B$, let us introduce some notation. First we fix an orthonormal basis $\{e_1, e_2\}$ in $\RR^2$ and the origin $O$.

\begin{definition} \label{def:cell}
A trapezium $ABCD$ (see Fig.~\ref{fig:cells}) is called a \emph{cell} if 
\begin{itemize}
\item the bases $AB \parallel CD$ and are parallel to $e_1$;
\item $AD$ and $CB$ are the sides, and $(AD, e_2) > 0$;
\item $|CD|\le |AB|$.
\end{itemize}
\end{definition}
(The points $A$, $B$, $C$ and $D$ are always assumed to be different.)

Let us denote $\conv (ABCD):= \conv\{A,B,C,D\}$ and $\inte(ABCD) := \inte \conv (ABCD)$, where $\conv$ is the convex hull and $\inte$ is the topological interior.

\begin{figure}[h]
\centering
\includegraphics{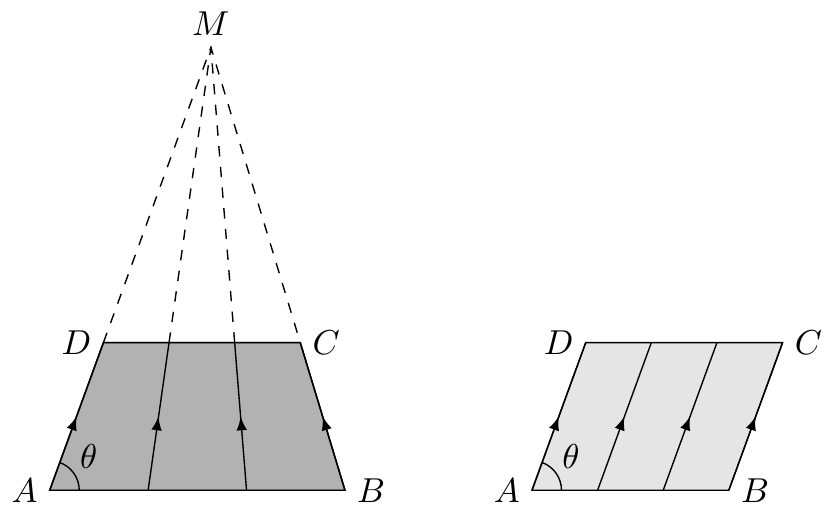}
\caption{Compressible (on the left) and incompressible (on the right) cells.}\label{fig:cells}
\end{figure}

From the Definition~\ref{def:cell} one can see that there are only the following two types of cells:
\begin{definition} \label{def:inc-cell}
A cell $ABCD$ is called
\begin{enumerate}
\item \emph{incompressible} if $|CD|=|AB|$;
\item \emph{compressible} if $|CD|<|AB|$.
\end{enumerate}
\end{definition}

\subsection{Patches and iterative construction}
In this section 
we iteratively construct the approximations $B_n$ of the desired vector field $B$, starting from a fixed compressible cell $\cc C_0$ and associated with it vector field $B_0$ (see Definition~\ref{def:cell2vec}).
We introduce a special partition of $\cc C_0$ into a union of smaller compressible and incompressible cells (which we call \emph{patches}). Next we replace each of these cells with the corresponding associated vector field. Every following step we refine the partition in a self-similar way.

In order to define the partition of a compressible cell $\mathcal C = ABCD$ we need to introduce some auxiliary points.

Let $E,F$ and $G$ denote the points on $CD$ such that $|DE|=|EF|=|FG|=|GC|$.
Similarly, let $L,M$ and $N$ denote the points on $AB$ such that $|AL|=|LM|=|MN|=|NB|$.
Let $U$ and $V$ denote the midpoints of $AD$ and $CB$ respectively.
Finally, we introduce the points $X,Y$ and $Z$ on $UV$ such that
$|UX|=|AL|$, $|XY|=|EF|$, $|YZ|=|MN|$ and $|ZV|=|GC|$.
This way we have divided the compressible cell $ABCD$ into:
\begin{itemize}
\item 4 incompressible cells: $ALXU$, $XYFE$, $MNZY$ and $ZVCG$;
\item 4 compressible cells: $UXED$, $LMYX$, $YZGF$ and $NBVZ$;
\end{itemize}
\begin{figure}[h]
\centering
\includegraphics{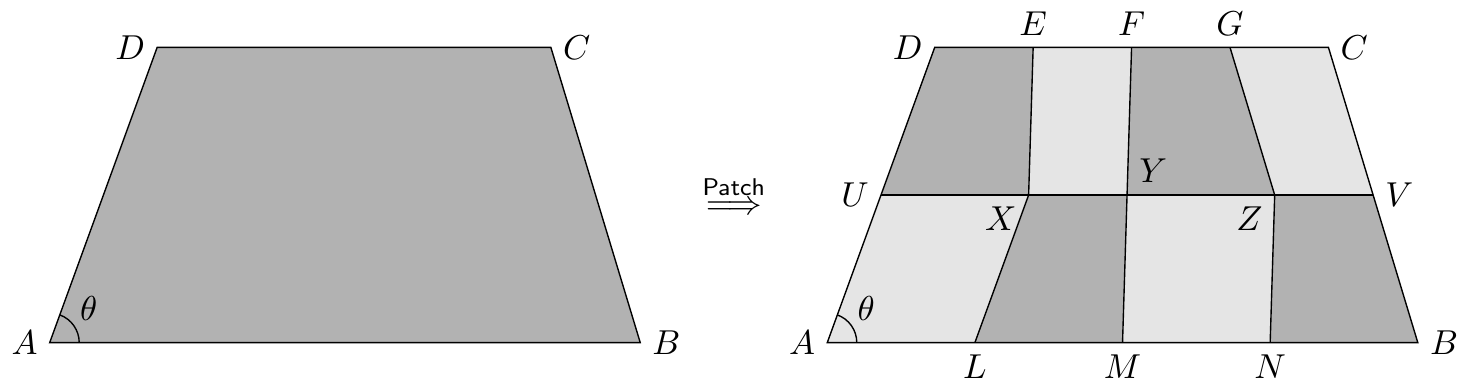}
\caption{A compressible cell before (on the left) and after (on the right) the $\mathsf{Patch}$.}\label{fig:patch}
\end{figure}

\begin{definition}\label{def:patches}
Let $\cc C=ABCD$ be a cell. We define the map $\ff{Patch}$ as follows:
\begin{itemize}
\item if $\cc C$ is incompressible, then $\ff{Patch}(\cc C):= \cc C$;
\item if $\cc C$ is compressible, then
\begin{equation*}
\ff{Patch}(\cc C):= \{ALXU, \; XYFE, \; MNZY , \; ZVCG , \; UXED, \; LMYX, \; YZGF , \; NBVZ \};
\end{equation*}
\item if $\{\cc C_i \}_{i=1}^N$ is a finite union of cells then
\begin{equation*}
\ff{Patch}\rb{\cup_{i=1}^N \cc C_i}:= \cup_{i=1}^N \ff{Patch}(\cc C_i).
\end{equation*}
\end{itemize}
\end{definition}

For any compressible cell $\cc C$ we define
\begin{equation}
\label{P_n}
\cc P_n(\cc C) := \underbrace{\ff{Patch} \circ \dots \circ \ff{Patch}}_{\text{$n$ times}} (\cc C).
\end{equation}
One can compute that $\cc P_n(\cc C)$ is a union of $4^{n}$ compressible cells and $\frac{4}{3}(4^{n}-1)$ incompressible cells. ($\cc P_0(\cc C)$ is consists only of the initial cell $\cc C$.)

Now we identify cells and ``patched'' cells with vector fields:

\begin{definition} \label{def:cell2vec}
Let $\mathcal C = ABCD$ be a cell. We define the vector field $V_{\cc C}\colon \RR^2 \to \RR^2$ \emph{associated with the cell} $\mathcal C$ as follows (see Figure~\ref{fig:cells}):
\begin{enumerate}
\item if $\mathcal C$ is incompressible, we set 
$V_{\cc C}(x):=\1_{\conv \cc C} \cdot \displaystyle\frac{AD}{\langle AD ,\, e_2 \rangle}$;
\item if $\mathcal C$ is compressible, we set 
$V_{\cc C}(x):=\1_{\conv \cc C} \cdot \displaystyle\frac{x-M}{\langle x-M ,\, e_2 \rangle}$, where $M$ is the intersection of lines AD and CB.
\end{enumerate}
Moreover if $\cup_{i=1}^N \cc C_i$ is a finite union of cells (with disjoint interiors) then we define the \emph{associated vector field} as
\begin{equation*}
V_{\cup_{i=1}^N \cc C_i} := \sum_{i=1}^N \hat \1_{\cc C_i} \cdot V_{\cc C_i}
\end{equation*}
where $\hat \1_{\cc C_i}:=\1_{\conv \cc C}/\max\rb{1,\sum_{i=1}^N \1_{\conv \cc C_i}}$.
\end{definition}

Observe that inside any cell $\cc C$ the associated vector field $V_{\cc C}$ is constant along the straight lines parallel to $V_{\cc C}$, therefore
\begin{equation} \label{cell-ts}
\ab{V_{\cc C}, \nabla} V_{\cc C} = 0
\end{equation}
inside $ABCD$. In view of Remark~\ref{rem:ts} this means that if a vector field $B\in BV(\RR^2)$ coincides with $V_{\cc C}$ inside the cell $\cc C$ then 
\begin{equation*}
\inte \cc C\subset T_B
\end{equation*}
where $T_B$ is the tangential set of $B$ (see Definition~\ref{def:ts}).
In other words, the interior of any compressible cell is the tangential set of the associated vector field.
A vector field associated with an incompressible cell has empty tangential set.

For brevity we introduce the \emph{range} of a cell $\cc C$ by
\begin{equation*}
\ff{Ran}(\cc C) := \osc_{\conv \cc C} V_{\cc C}
\end{equation*}
where $\osc_E f:=\sup_{x,y\in E} |f(x) - f(y)|$ is the standard oscillation of a function $f\colon E\to \RR^n$.

It is easy to compute that
\begin{equation}\label{cell-range}
\ff{Ran}(\cc C) = |V_{\cc C}(D) - V_{\cc C}(C)| = \frac{|AD-BC|}{\dist(AB,CD)} = 
\frac{|AB|-|CD|}{\dist(AB,CD)}.
\end{equation}

Let us now consider the cell after the $\ff{Patch}$ operation.
Let $V:=V_{\cc C}$ and $W:=V_{\ff{Patch}(\cc C)}$ denote the vector fields associated with $\cc C$ and $\ff{Patch}(\cc C)$ respectively.
For any fixed $x_2$ the maps
\begin{equation*}
x_1 \mapsto V_1(x_1,x_2)
\quad\text{and}\quad
x_1 \mapsto W_1(x_1,x_2)
\end{equation*}
are (non-strictly) decreasing functions of~$x_1$ (inside $\cc C$). They also take the same values on the sides $AD$ and $BC$.
Therefore in view of \eqref{cell-range} we obtain
\begin{equation}\label{changes}
\sup_{x\in \conv \cc C}|V_{\cc C}(x) - V_{\ff{Patch}(\cc C)}(x)| \le \ff{Ran}(\cc C).
\end{equation}

For any $a\in \RR^2$ with $a_2 \ne 0$ let us denote 
\begin{equation*}
\arg a:= \pi/2 - \arctan(a_1/a_2).
\end{equation*}
By Definition~\ref{def:cell} we have 
\begin{equation*}
\angle DAB \le \arg V_{\cc C}(x) \le \pi - \angle ABC
\end{equation*}
for any $x\in \conv \cc C$.
From Definition~\ref{def:patches} it is easy to see that for any fixed $x_2$ the map
\begin{equation*}
x_1 \mapsto \arg V_{\ff{Patch}(\cc C)}(x_1,x_2)
\end{equation*}
is (nonstrictly) increasing function of $x_1$  which takes values $\angle DAB$ and $\pi - \angle ABC$ on the edges $AD$ and $BC$ respectively. Therefore
\begin{equation} \label{angle-changes}
\angle DAB \le \arg V_{\ff{Patch}(\cc C)}(x) \le \pi - \angle ABC
\end{equation}
for any $x\in \conv \cc C$.

Let us fix some initial compressible cell $\cc C=ABCD$.
Let $\Omega:= \inte \cc C$ denote its interior and let $I$ denote the projection of $\Omega$ on the axis $Oe_2$.
We define the $n$-th approximation of the desired vector field as follows:
\begin{equation} \label{B_n}
B_n := V_{\cc P_n(\cc C)}
\end{equation}
where $\cc P_n$ and $V$ were introduced in \eqref{P_n} and Definition~\ref{def:cell2vec} respectively.

\subsection{Passage to the limit}
In this section we consider the approximate vector fields $B_n$ defined by \eqref{B_n} and pass to the limit as $n\to \infty$. Our proofs will be direct and elementary, nevertheless we will present all the details for the sake of completeness.

\begin{lemma}\label{lem:B_n-ests}
There exist a constant $C$ and a bounded vector field $B\colon \Omega \to \RR^2$ such that
\begin{enumerate}
\labitem{(i)}{uniform-bounds}
$\sup_{\Omega}|B_n| \le C$ for all $n\in \NN$;
\labitem{(ii)}{uniform-convergence}
$B_n \to B$ uniformly on $\Omega$ as $n\to \infty$;
\labitem{(iii)}{BV-upper}
$\|B_n\|_{BV(\Omega)}\le C$ for all $n\in \NN$.
\end{enumerate}
(The constant $C$ depends only on the geometric properties of the initial cell $\cc C$, namely $|AB|$, $\ff{Ran}(ABCD)$ and $\dist(AB,CD)$.)
\end{lemma}
\begin{proof}
Let $\cc C':=A'B'C'D'$ be any compressible cell in $\cc P_n(\cc C)$.

Let us compare $\cc C'$ with the initial cell $\cc C$. By construction (see Definition~\ref{def:patches}) we have
\begin{itemize}
  \item $|A'B'| = 4^{-n} |AB|$, 
  \item $|C'D'| = 4^{-n} |CD|$,
  \item $\dist(A'B',C'D') = 2^{-n} \dist(AB,CD)$
\end{itemize}
and by \eqref{cell-range}
\begin{equation*}
\ff{Ran}(\cc C') = 2^{-n} \ff{Ran}(\cc C).
\end{equation*}
Together with \eqref{changes} this implies that
\begin{equation}\label{deltaB}
|B_{n+1}-B_{n}| \le 2^{-n} \ff{Ran}(\cc C)
\end{equation}
everywhere on $\Omega$. Hence the series 
\begin{equation*}
\sum_{n=1}^\infty \sup_{\Omega} |B_{n+1}-B_{n}|
\end{equation*}
converge and the sequence $\{B_n\}_{n\in \RR}$ is uniformly bounded.
Therefore \ref{uniform-convergence} and \ref{uniform-bounds} are proved.

It remains to prove the estimate of total variation \ref{BV-upper}.

First we are going to estimate the absolutely continuous part of $D B_n$.
Let $M$ denote the intersection point of straight lines passing through $A'D'$ and $B'C'$.
Consider for a moment the local coordinates $(x_1,x_2)$ with origin in $M$.
Let $v = V_{\cc C'}$ be the associated with $\cc C'$ vector field.
Then by Definition~\ref{def:cell2vec} we have $v(x_1,x_2) = (\frac{x_1}{x_2}, 1)$ when $(x_1,x_2)\in \conv \cc C'$.
Hence
\begin{equation*}
D^a v = \begin{pmatrix} \frac{1}{x_2} & -\frac{x_1}{x_2^2} \\ 0 & 0\end{pmatrix} \Le^2 \rest \inte \cc C'
\end{equation*}

In view of \eqref{angle-changes} we also have
\begin{equation*}
\angle DAB \le \arg B_n(x) \le \pi - \angle ABC
\end{equation*}
for any $x\in \Omega$ and any $n\in \NN$.
Hence $\angle DAB \le \angle D'A'B' \le \pi - \angle ABC$, so the cell $\cc C'$ is contained in the half-cone
\begin{equation*}
\setof{(x_1,x_2)}{|x_1|\le c |x_2|, \; x_2 \le 0},
\end{equation*}
where $c:= \max(\cot \angle DAB, \cot \angle ABC)$ is independent of $n$.
Hence 
\begin{equation*} 
\abs{\partial_2 v_1} \le c \abs{\partial_1 v_1}
\end{equation*}
inside $\cc C'$ and consequently
\begin{equation}
\label{ac-cell-est-above}
|D^a v|(\Omega) \le (1+c) \int_{\inte \mathcal C'} |\partial_1 v| \mathcal L^2 = (1+c) \ff{Ran(\cc C')} \cdot \dist(A'B',C'D').
\end{equation}

Since $\cc P_n(\cc C)$ consists of $4^n$ compressible cells with ranges $2^{-n} \ff{Ran}(\cc C)$, by \eqref{ac-cell-est-above} we have
\begin{align}
\label{ac-est-upper}
|D^a B_n|(\Omega) =&~ \sum_{\cc C' \in \cc P_n} |D^a V_{\cc C'}|(\Omega) \le  4^n \cdot 2^{-n} (1+c)\ff{Ran}(\cc C) \cdot 2^{-n} \dist(AB,CD) \\
\le&~ (1+c)\ff{Ran}(\cc C) \cdot \dist(AB,CD)
\end{align}
is bounded for all $n\in \NN$.

Finally, let us estimate the jump part of $B_n$. The jumps of $B_{n+1} - B_{n}$ are concentrated on the union of upper and lower bases and midlines of the compressible cells in $\cc P_n(\cc C)$. The $\Ha^1$-measure of this union is not greater than $4^n \cdot 3 \cdot 4^{-n} |AB|$. But from \eqref{deltaB} it follows that the values of jumps are not greater than $2\cdot 2^{-n} \ff{Ran}(\cc C)$. Therefore
\begin{equation}\label{delta-jump-est}
|D^j (B_{n+1} - B_{n})|(\inte \cc C) \le 6 \cdot 2^{-n} |AB| \cdot \ff{Ran}(\cc C)
\end{equation}
and (since $D^j B_0(\Omega)=0$)
\begin{equation}\label{jp-est}
|D^j B_{n}|(\inte \cc C) \le \sum_{k=1}^n |D^j (B_{k}-B_{k-1})|(\Omega) \le \sum_{k=1}^n 12 \cdot 2^{-k} |AB| \cdot \ff{Ran}(\cc C) \le 12|AB| \cdot \ff{Ran}(\cc C)
\end{equation}
is bounded for all $n\in \NN$. 

The estimate \ref{BV-upper} is a consequence of \eqref{ac-est-upper} and \eqref{jp-est}.
\end{proof}

In order to study the properties of the limit vector field $B$ we need to introduce some auxiliary sets.

Assuming that the origin coincides with point $A$ of the initial cell, let
\begin{equation} \label{def:L}
L:=\setof{\rb{x_1, \frac{m}{2^n}\dist(AB,CD)}}{x_1 \in \RR, \; n\in \NN, \; m\in \ZZ \cap[0, 2^n]}
\end{equation}
denote the union of the horizontal lines containing bases and midlines of all the cells obtained by iterations of $\ff{Patch}$.

Let $\cc P^\rr{comp}_n(\cc C)$ denote the union of all compressible cells in $\cc P_n(\cc C)$. Let
\begin{equation} \label{def:S}
S_n:=\bigcup_{\cc C' \in \cc P^\rr{comp}_n(\cc C)} \conv C',
\qquad
S:=\bigcap_{n\in \NN} S_n.
\end{equation}

The following lemma characterizes some fine properties of the limit vector field $B$ and the sequence $B_n$:

\begin{lemma}\label{lem:B}
Let $B$ denote the limit of $\{B_n\}_{n\in\NN}$ given by Lemma~\ref{lem:B_n-ests}. Then $B\in BV(\Omega)$ and
\begin{itemize}
\item the absolutely continuous part $D^a B$ is zero;
\item the jump part $D^j B$ is concentrated on $L$ and
\begin{equation*}
D^j B_n \weakstarto D^j B
\end{equation*}
locally in $\Omega$ as $n\to \infty$;
\item $\div^a B = \div^j B = 0$;
\item the Cantor part $D^c B$ is concentrated on the set $S$ (defined in \eqref{def:S}) and
\begin{equation*}
D^a B_n \weakstarto D^c B
\end{equation*}
locally in $\Omega$ as $n\to \infty$;
\item there exists a constant $\alpha>0$ such that $\div^c B = -\alpha \Ha^{3/2} \rest S$.
\end{itemize}
\end{lemma}
\begin{proof}
For convenience of the reader we divide the proof in several steps.

\textit{Step 1.}
First we apply Lemma~\ref{lem:B_n-ests}: \ref{uniform-convergence}, \ref{BV-upper} and lower semicontinuity of total variation imply that $B\in BV(\Omega)$.

Let us consider the sets introduced in \eqref{def:L} and \eqref{def:S}.
By our construction the area of all compressible cells in $\cc P_n(\cc C)$ equals to $\Le^2(S_n) = 2^{-n} \Le^2(\Omega)$, so the intersection $S$ of all $\{S_n\}$ is $\Le^2$-negligible.
Clearly the union of the boundaries of all cells in $\cc P_n(\cc C)$ is $\Le^2$-negligible (for all $n\in \NN$).
Hence a.e. point $x\in \Omega$ 
belongs to $\inte \conv \cc I$ for some incompressible cell $\cc I \in \cc P_n(\cc C)$ for some $n\in \NN$. Hence inside this incompressible cell $B$ coincides with the vector field $v$ associated with $\cc I$. But $v$ is constant inside $\cc I$ by Definition~\ref{def:cell2vec}. Hence the approximate differential of $B$ is zero a.e. in $\Omega$ and so $D^a B=0$
(e.g. by Theorem 3.83 from \cite{AFP}).

\textit{Step 2.}
Let us prove that $D^j B$ is concentrated on $L$. 

Observe that any $x\in \Omega \setminus L$ is a Lebesgue point of $B$. Indeed, any such $x$ is a Lebesgue point of $B_n$ for all $n\in \NN$. Since 
\begin{multline*}
\int_{B_r(x)} |B(y)-B(x)| \, dy \le {} \\ {} \le
\int_{B_r(x)} |B(y)-B_n(y)| \, dy +
\int_{B_r(x)} |B_n(y)-B_n(x)| \, dy + \int_{B_r(x)} |B_n(x)-B(x)| \, dy
\end{multline*}
using uniform convergence \ref{uniform-convergence} one can directly show that $B(x)$ is the Lebesgue value of $B$ at $x$.

Since all the points in $\Omega \setminus L$ are Lebesgue points of $B$, the jump part $D^j B$ is concentrated on $L$ (by definition of the jump part, see e.g. \cite{AFP}).

\textit{Step 3.}
By \eqref{delta-jump-est} the jump part $D^j B_n$ converges to some measure $\mu$ concentrated on $L$:
\begin{equation*}
D^j B_n \weakstarto \mu
\end{equation*}
locally in $\Omega$ as $n\to \infty$.
Since $B_n \to B$ uniformly in $\Omega$, we have $D B_n \weakstarto DB$ locally in $\Omega$ as $n\to \infty$.
Therefore 
\begin{equation*}
D^a B_n \weakstarto \nu
\end{equation*} 
locally in $\Omega$ as $n\to \infty$, where $\nu := DB- \mu$.

Since $\mu+\nu = D^c B + D^j B$, and both $D^j B$ and $\mu$ are concentrated on $L$, we have $D^j B = \mu + \nu\rest L$. Therefore to prove that $D^j B = \mu$ it is sufficient to show that $\nu \rest L =0$.

Using the same computations as in derivation of \eqref{ac-est-upper} one can derive the estimate
\begin{equation}
\label{ac-est-stripe-upper}
|D^a B_n|(\Omega\cap \RR\times(a,b)) \le
2 \ff{Ran}(\cc C) \cdot (b-a)
\end{equation}
for any $a,b\in I$ and $n\in \NN$. Covering $L$ by horizontal stripes with arbitrary small total projection on $Oe_2$ it is possible to show that $|\nu|(L)<\eps$ for any $\eps>0$. Hence indeed $\nu \rest L =0$.

Now we can conclude that $\mu = D^j B$ and $\nu = D^c B$.

By construction of $B_n$ for any $n\in \NN$ we have $\div^j B_n = 0$.
But $\div^j B_n \weakstarto \div^j B$ as $n\to \infty$, since $D^j B_n \weakstarto D^j B$ as $n\to \infty$.
Therefore $\div^j B = 0$.

\textit{Step 4.}
The set $S$ is closed, so for any compact $K\subset \Omega \setminus S$ we have $\dist(K,S)>0$. Hence $|D^a B_n|(K)=0$ for sufficiently large $n\in \NN$. Therefore $D^c B$ is concentrated on $S$ (by inner regularity).

Observe that $S_n$ consists of $4^n$ trapeziums whit height $2^{-n} \dist(AB,CD)$ and bases $4^{-n}|AB|$ and $4^{-n}|CD|$. Therefore $S_n$ can be covered by $4^n \cdot 2^n$ balls of radius $r_n:= 4^{-n}|AB|$
(assuming for simplicity that $\dist(AB,CD)=1$).
Hence the Hausdorff dimension of $S$ is $d=3/2$.

From the previous step we know that $\div^c B$ is the *-weak limit of the negative measures $\div^a B_n = D^a_1 (B_n)_1$. For any fixed $n\in \NN$ for any compressible cell $\cc C'=A'B'C'D' \in \cc P_n(\cc C)$
using Lemma~\ref{lem:B_n-ests} we obtain
\begin{equation*}
(\div^a B)(\inte \cc C') = -\ff{Ran} (\cc C') \cdot \dist(A'B',C'D') = -4^{-n} \ff{Ran}(\cc C)
\end{equation*}
Hence there exists a constant $\alpha>0$ (which depends only on the range of the initial cell $\cc C$) such that
$\div^c B = - \alpha \Ha^{3/2} \rest S$.
\end{proof}

\begin{remark}
In fact since $D^j B_n \weakstarto D^j B$ from \eqref{delta-jump-est} it follows that $|D^j B_n - D^j B|(\Omega) \to 0$ as $n \to \infty$.
\end{remark}

Now we prove that the vector fields $B_n$ are \emph{uniformly} nearly incompressible in the following sense:
\begin{lemma}\label{lem:uniform-NI}
There exists a constant $C>0$ (depending only on the ratio between the bases of the initial cell $\cc C$) such that for any $n\in \NN$ there exists $\rho_n \in \Linf(\Omega)$ which satisfies
\begin{equation*}
\div (\rho_n B_n) = 0 \quad\text{in} \; \;\ss D'(\Omega)
\end{equation*}
and $\frac{1}{C} \le \rho_n \le C$ a.e. in $\Omega$.
\end{lemma}
\begin{proof}
Fix $n\in \NN$.
Let $\cc C':=A'B'C'D'$ be any compressible cell in $\cc P_n(\cc C)$ and let $\Omega':=\inte \cc C'$. Consider for a moment the local coordinates $(x_1,x_2)$ with origin in $M=A'D'\cap B'C'$. Let $\bar y$ denote the second coordinate of the lower base $A'B'$. For any $x\in \Omega'$ we define
\begin{equation}\label{rho-comp}
\rho_n(x):=\frac{\bar y}{x_2} \qquad \text{if} \quad x\in \Omega'.
\end{equation}
Since $B_n(x)=(\frac{x_1}{x_2},1)$ for $x\in \Omega'$, we have
\begin{equation*}
\div(\rho_n B_n) = \bar y \rb{ \d_1 \frac{x_1}{x_2^2} + \d_2 \frac{1}{x_2} } = 0,
\end{equation*}
i.e. $\rho_n$ solves $\div(\rho_n B_n)$ in $\Omega'$.

Observe that for any $x\in \Omega'$ we have $1 \ge \frac{x_2}{\bar y} \ge \frac{|C'D'|}{|A'B'|}$, and by construction $\frac{|C'D'|}{|A'B'|}=\frac{|CD|}{|AB|}$, hence
\begin{equation}\label{rho_n-bounds}
1 \le \rho_n(x) \le \frac{|AB|}{|CD|}
\end{equation}
for all $x\in \Omega'$.

Now we will define $\rho_n$ inside of the incompressible cells. In order to do so we divide the set of all incompressible cells in $\cc P_n (\cc C)$ into two classes: \emph{lower} and \emph{upper}.
Namely, if $n=1$ then the cells $ALXU$ and $MNZY$ are called \emph{lower} and the cells $XYFE$ and $ZVCG$ are called \emph{upper}. If $n>1$ then an incompressible cell $\cc I \in \cc P_n(\cc C)$ is called \emph{lower} (\emph{upper}) if it is one of the \emph{lower} (\emph{upper}) cells of $\cc P_{n-1}(\cc C)$ or if $\cc I$ is one of the \emph{lower} (\emph{upper}) cells in $\ff{Patch}(\Sigma)$ for some compressible cell $\Sigma \in \cc P_{n-1}(\cc C)$.
Then we define (see Figure~\ref{fig:singularSet} for the case when $|AB|=2|CD|$)
\begin{equation}\label{rho-inc-lower}
\rho_n(x):=1 \qquad \text{if $x$ is inside some \emph{lower} incompressible cell $\cc I\in \cc P_n(\cc C)$}
\end{equation}
and
\begin{equation}\label{rho-inc-upper}
\rho_n(x):=\frac{|AB|}{|CD|} \qquad \text{if $x$ is inside some \emph{upper} incompressible cell $\cc I\in \cc P_n(\cc C)$.}
\end{equation}
By formulas \eqref{rho-comp}, \eqref{rho-inc-lower} and \eqref{rho-inc-upper} the function $\rho_n$ is defined almost everywhere in $\Omega$. Moreover, by construction it satisfies \eqref{rho_n-bounds}.
Finally, the normal traces of $\rho_n B_n$ on the interfaces between adjoint cells cancel, therefore one can directly check that $\div(\rho_n B_n) = 0$ in $\ss D'(\Omega)$.
\end{proof}

\begin{theorem} \label{th:counterexample}
There exist a bounded vector field $B\in BV(\Omega;\RR^2)$ and a bounded function $\rho\colon \Omega \to \RR$ such that
\begin{itemize}
\labitem{(i)}{c-i} $\div^c B$ is a nontrivial nonpositive measure concentrated on the set $S$ (see \eqref{def:S})
\labitem{(ii)}{c-ii} $|\div^c B|$-a.e. $x\in S$ belongs to the tangential set $T_B$ of $B$
\labitem{(iii)}{c-iii} $\rho(\Omega \setminus S)=\{1\}\cup\left\{r\right\}$ where $r:=\frac{|AB|}{|CD|}$
\labitem{(iv)}{c-iv} $\div(\rho B)=0$
\labitem{(v)}{c-v} $|\div^c B|$-a.e. $x\in S$ is not $L^1$ approximate continuity point of $\rho$.
\end{itemize}
\end{theorem}

\begin{proof}
Consider the sequence of vector fields \eqref{B_n}. By Lemma~\ref{lem:B} there exists a bounded vector field $B\colon \Omega \to \RR^2$ which belongs to $BV(\Omega)$ such that $B_n \to B$ uniformly in $\Omega$ as $n\to \infty$. Then \ref{c-i} is a direct consequence of Lemma~\ref{lem:B}.

Let $\pfi \colon\Omega \to \RR$ be a bounded Borel function with compact support in $\Omega$.

\begin{equation*}
\int \pfi \ab{B, D^c B} = \lim_{n \to \infty} \int \pfi \ab{B, D^a B_n}
\end{equation*}
since $D^a B_n \weakstarto D^c B$ locally in $\Omega$ by Lemma~\ref{lem:B}.
On the other hand
\begin{equation*}
\int \pfi \ab{B, D^a B_n} = \int \pfi \ab{B-B_n, D^a B_n} + \int \pfi \ab{B_n, D^a B_n}.
\end{equation*}
Since $B_n$ converges to $B$ uniformly in $\Omega$ and total variation of $B_n$ is uniformly bounded (see \ref{BV-upper} of Lemma~\ref{lem:B}), we have
\begin{equation*}
\lim_{n\to \infty} \int \pfi \ab{B-B_n, D^a B_n} = 0.
\end{equation*}
By construction of $B_n$ we have $\rb{B_n, D^a B_n} = 0$ (see \eqref{cell-ts}). Therefore we conclude that
\begin{equation}\label{testfun}
\int \pfi \ab{B, D^c B} = \int \pfi(x) M(x)\cdot B(x) \, d|D^c B|(x) = 0,
\end{equation}
where $D^c B = M |D^c B|$ is the polar decomposition of $D^c B$.
Since \eqref{testfun} holds for arbitrary $\pfi$, we deduce that 
\begin{itemize}
\item $M(x)\cdot B(x) = 0$ for $|D^c B|$-a.e. $x\in \Omega$;
\item $|D^c B|$ is concentrated on the set $\setof{x\in \Omega}{M(x) \cdot B(x) = 0}$.
\end{itemize}

But $|D^c B|$-a.e. point is a Lebesgue point of $B$ (since the Cantor part $D^c B$ is concentrated on the set of Lebesgue points of $B$, see e.g. \cite{AFP}, p. 184). Thus $|D^c B|$-a.e. point belongs to $T_B$.
Since $|\div^c B| \ll |D^c B|$ this implies \ref{c-ii}.

Using the terminology we introduced in the proof of Lemma~\ref{lem:uniform-NI} let us define  $\rho=\rho(x)$ by
\begin{equation}\label{lim-rho}
\rho(x):=
\begin{cases}
1 & \text{if $x$ is inside some \emph{lower} incompressible cell $\cc I\in \cc P_n(\cc C)$ for some $n\in \NN$}; \\
r & \text{if $x$ is inside some \emph{upper} incompressible cell $\cc I\in \cc P_n(\cc C)$ for some $n \in \NN$.}
\end{cases}
\end{equation}
As we noted in the proof of Lemma~\ref{lem:B},
a.e. point $x\in \Omega$ belongs to the interior of some incompressible cell $\cc I \in \cc P_n(\cc C)$ starting from some $n=n_x\in \NN$. Therefore the function $\rho$ is well-defined a.e. in $\Omega$ and clearly \ref{c-iii} is satisfied. (See Figure~\ref{fig:singularSet}.)

Let $\rho_n$ denote the sequence constructed in the proof of Lemma~\ref{lem:uniform-NI}.
Since a.e. $x\in \Omega$ is inside some incompressible cell in $\cc P_n(\cc C)$ when $n$ is large enough, we immediately obtain that $\rho_n \to \rho $ a.e. in $\Omega$ as $n\to \infty$.
Since $B_n\to B$ uniformly in $\Omega$ as $n\to \infty$, now we can pass to the limit as $n\to \infty$ in the distributional formulation of $\div(\rho_n B_n)=0$ and deduce \ref{c-iv}.

\begin{figure}[!h]
\centering
\includegraphics{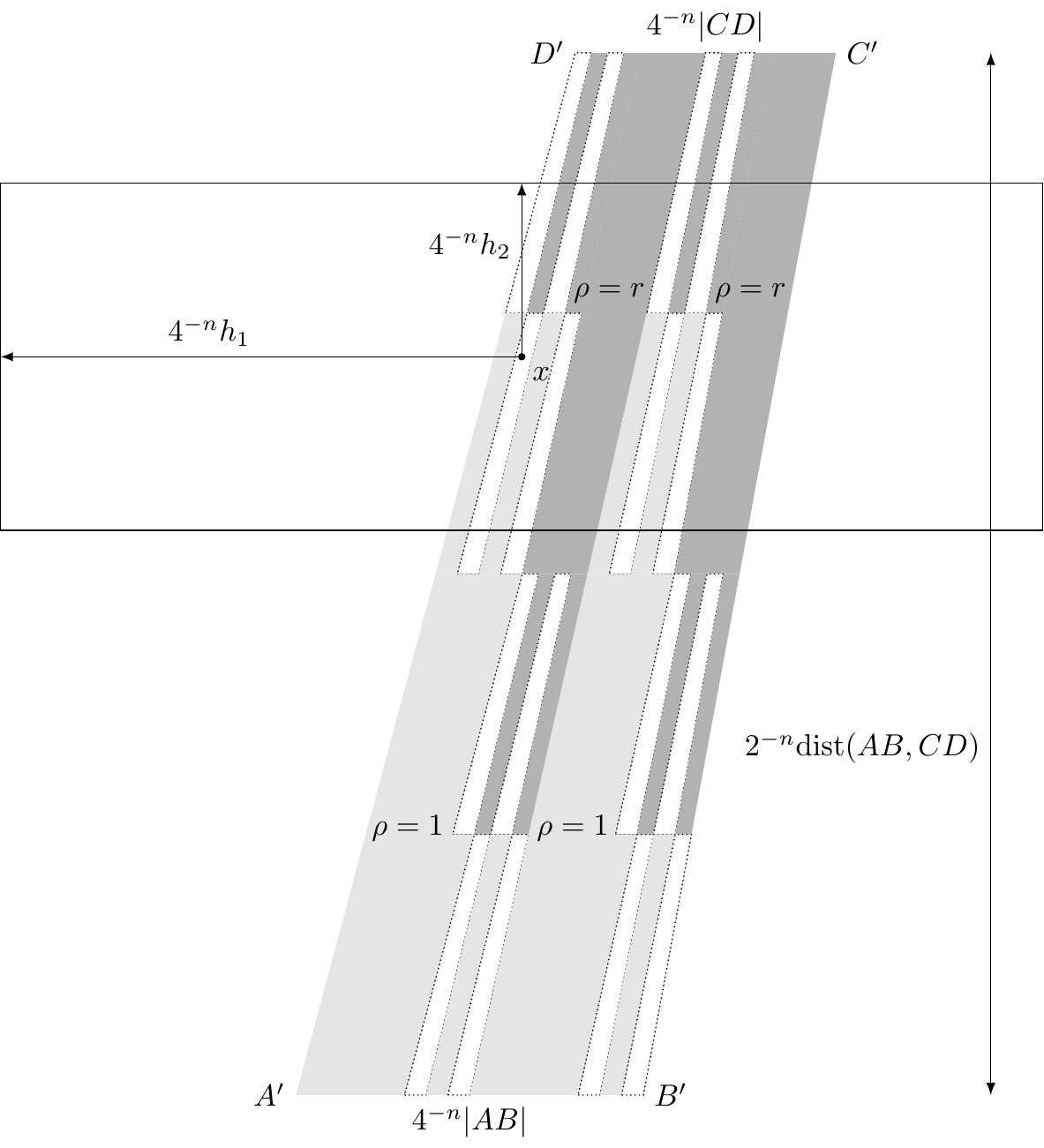}
\caption{Blow-up of the tangential set.}\label{fig:singularSet}
\end{figure}

Before proving the last claim \ref{c-v} let us introduce some notation.
For any $L^1$-approximate continuity point of $\rho$ let $\tilde \rho(x)$ denote the corresponding approximate $L^1$ limit of $\rho$ at $x$, i.e. 
\begin{equation*}
\lim_{\eps\to 0}\frac{1}{|B_\eps(x)|} \int_{B_\eps(x)} |\rho(y) - \tilde \rho(x)| \, dy = 0.
\end{equation*}
For simplicity let us assume that $A$ has coordinates $(0,0)$.
Let $c:= \max(\cot \angle DAB, \cot \angle ABC)$.
For any $x\in \Omega$ and $n\in \NN$ let $Q_n(x)$ denote the rectangle
with center at $x$ and sides $4^{-n} \cdot h_1$ and $4^{-n} \cdot h_2$ (see Figure~\ref{fig:singularSet}), where $h_1=(1+c)|AB|$ and $h_2=|AB|$.
We also denote
\begin{equation*}
Q^+_n(x):=Q_n(x)\cap\rho^{-1}(r),
\quad
Q^-_n(x):=Q_n(x)\cap \rho^{-1}(1)
\end{equation*}
and
\begin{equation*}
R_n(x):=\frac{\min(|Q^+_n(x)|, |Q^-_n(x)|)}{|Q_n(x)|}.
\end{equation*}

One can prove that the following properties hold for the function $\rho$ which we consider:
\begin{itemize}\label{app-cont-points}
\item if $x$ is $L^1$-approximate limit of $\rho$ then either $\tilde \rho(x)=1$ or $\tilde \rho(x)=r$ (this follows directly from definition, since $\rho$ takes only values $1$ and $r$ a.e.);
\item consequently, if $x$ is a point of $L^1$-approximate continuity of $\rho$ then
$\lim_{n\to \infty} R_n(x) = 0$.
\end{itemize}
Suppose that $x\in S$ has coordinates $(x_1,x_2)$. We are going to estimate from below the area $|Q^-_n(x)|$ of the part of $\rho^{-1}(1)$ which lies inside the rectangle $Q_n(x)$.

By construction of $S$ for any $n\in \NN$ we have $x\in \conv A'B'C'D'$ for some compressible cell $A'B'C'D'\in \cc P_n(ABCD)$ (see Figure~\ref{fig:singularSet}).
When $n$ is sufficiently large we also have $Q_n(x) \subset \Omega$.
Since $|\cot \angle D'A'B'| \le c$ (see proof of Lemma~\ref{lem:B_n-ests})
in view of our choice of $h_1$ we have
\begin{equation*}
4^{-n}|AB| + 4^{-n} h_2 |\cot\angle B'A'D'| \le 4^{-n} \cdot (1+c) \cdot |AB| = 4^{-n} h_1
\end{equation*}
hence $Q_n(x)$ contains \emph{the whole} intersection of $\inte A'B'C'D'$ with the horizontal stripe
\begin{equation*}
H_n(x):=\setof{(\xi_1,\xi_2)}{x_2-4^{-n}h_2 \le \xi_2 \le x_2+4^{-n}h_2}.
\end{equation*}
Therefore it is sufficient to estimate from below the area of $H_n(x)\cap \inte A'B'C'D' \cap \rho^{-1}(1)$.

In order to compute $|H_n(x)\cap \inte A'B'C'D' \cap \rho^{-1}(1)|$ let us find out how many incompressible cells in $\cc P_m(A'B'C'D')$ with $\rho=1$ have nonempty intersection with $H_n(x)$, where $m \in \NN$.
Clearly the number of such cells depends on the position of $x$ in $A'B'C'D'$ and on the value of $m$.
However if $m=2$ and $\dist(x,C'D')\ge \frac{1}{4} \dist(A'B',C'D')$ then $H_n(x)\cap \inte A'B'C'D'$ intersects at least $4$ incompressible cells in $\cc P_m(A'B'C'D')$ with $\rho = 1$
(see Figure~\ref{fig:singularSet}) and moreover we can estimate
\begin{align}
|Q^-_n(x)| \ge&~ \Bigl|H_n(x)\cap \inte A'B'C'D' \cap \rho^{-1}(1)\Bigr|\notag\\
\ge&~ 4 \cdot \frac{4^{-n} |CD|}{16} \cdot 4^{-n} h_2 = \frac{|CD|}{16 h_1} |Q_n(x)|
=\frac{|CD|}{32|AB|} |Q_n(x)|.
\label{Q-}
\end{align}

The same way one can prove that
\begin{equation}\label{Q+}
|Q^+_n(x)| \ge \frac{|CD|}{32 |AB|} |Q_n(x)|
\end{equation}
if $\dist(x,A'B')\ge \frac{1}{4} \dist(A'B',C'D')$.

Equations \eqref{Q-} and \eqref{Q+} imply that if
\begin{equation*}
R_n(x) < \frac{|CD|}{32 |AB|}
\end{equation*}
then either $\dist(x, A'B')< \frac{1}{4}\dist(A'B',C'D')$ or $\dist(x, C'D')< \frac{1}{4}\dist(A'B',C'D')$, i.e.  $x_2 \in 2^{-n} (\ZZ+(-1/4,1/4))$.
Therefore if $R_n(x)\to 0$ as $n\to \infty$ then $x_2 \in I:=\cap_{n=N}^\infty 2^{-n} (\ZZ+(-1/4,1/4))$ for some $N\in \NN$. Observe that 
\begin{align*}
I =&~ \cap_{n=N}^\infty 2^{-n} (\ZZ+(-1/4,1/4)) \cap 2^{-n-1}(\ZZ+(-1/4,1/4)) =
\cap_{n=N}^\infty 2^{-n}(\ZZ + (-1/8,1/8)) \\
=&~ \dots = \cap_{n=N}^\infty 2^{-n} (\ZZ + (-2^{-k},2^{k})) \subset 2^{-N}(\ZZ+(-2^{-k},2^{k}))
\end{align*}
for any natural $k\ge 4$,
hence $I\subset 2^{-N}\ZZ \subset \QQ$.
Since $|\div^c B|((\RR\times \QQ) \cap \Omega) = 0$ (this follows from \eqref{ac-est-stripe-upper}) the proof is complete.
\end{proof}

\begin{remark}
The cells in $\cc P_n(\cc C)$ have bases with lengths proportional to $4^{-n}$, while the distance between the bases is proportional to $2^{-n}$.
In other words the cells become ``stretched'' along the vertical axis as $n\to \infty$ (in comparison with the appropriately scaled initial cell $\cc C$).
Due to this, even though the rectangle $Q_n(x)$ always intersects some compressible cells in $\cc P_n(\cc C)$, it is not possible to control a priori the value of $\rho$
inside of this intersection (since the position of $x$ is not known precisely and the aspect ratio of $Q_n(x)$ is fixed).
Because of this we had to consider the structure of $\rho$ inside of the compressible cell $A'B'C'D'\in \cc P_n(\cc C)$.
In order to use this structure we involved $\cc P_m(A'B'C'D')$ with $m=2$.
\end{remark}

\begin{remark}
As we already mentioned,
our construction of $B$ is inspired by the vector field presented in \cite{AdLM}, which answers Question~\ref{q1} in negative.
Though the original construction of that vector field is completely analytic, 
it can also be presented in our geometric terms. Namely, the only difference between these two constructions is that in \cite{AdLM} the $\ff{Patch}$ operation is given by
\begin{equation*}
\ff{Patch}(ABCD):= \{ UE'ED, \; F'VCF , \; AGG'U , \; HBVH'
, \; E'F'FE, \; GHH'G' \},
\end{equation*}
where $U$ and $V$ are midpoints of $AD$ and $BC$ respectively, $|DE|=|EF|=|FC|$, $|AG|=|GH|=|HB|$, $E',F',G',H' \in UV$ and $|E'F'|=|EF|$, $|G'H'|=|GH|$, $|UG'|=|H'V|$, $|UE'|=|F'V|$
(see Figure~\ref{fig:AdLMpatch}).
\begin{figure}[!h]
\centering
\includegraphics{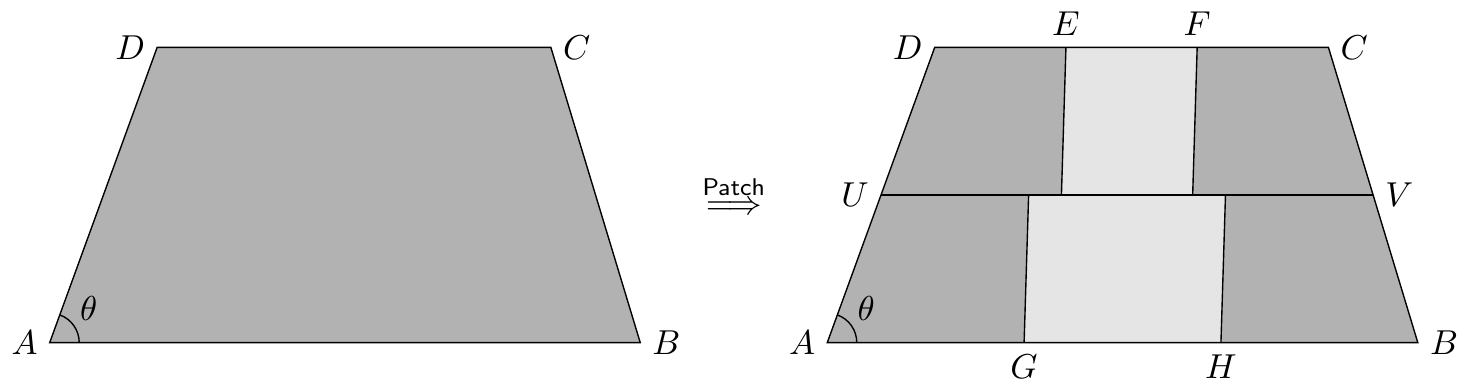}
\caption{The $\ff{Patch}$ operation used in \cite{AdLM}.}\label{fig:AdLMpatch}
\end{figure}

The main property of $\ff{Patch}$ used in the present paper (see Figure~\ref{fig:patch}) is that the resulting approximate vector fields are uniformly nearly incompressible (see Lemma~\ref{lem:uniform-NI}).
\end{remark}

\subsection{Analysis of the discontinuity set} \label{sec:Analysis-of-S}
In this section we would like to study in more details the properties of the set $S$.
Our analysis will be based on the following observation concerning the trajectories of $B$:
\begin{lemma}[Flow of $B$]\label{lem:flow}
Let $B$ denote the limit of $\{B_n\}_{n\in\NN}$ given by Lemma~\ref{lem:B_n-ests}. Then there exists a disjoint family $\gg F$ of trajectories of $B$ such that $\cup_{C\in \gg F} C = \Omega$. Moreover, for $(|\div^c B| + \Le^2)$-a.e. $x\in \Omega$ the corresponding trajectory $C_x\in \gg F$ contains exactly one point $s_x\in S$ of the set $S$ defined in \eqref{def:S}.
\end{lemma}

Let $\gg F$ denote the disjoint family of the trajectories of $B$ given by Lemma~\ref{lem:flow}.
Let $\gamma_x\colon (0,1)\to \Omega$ denote the parametrization of the trajectory $C_x\in \gg F$ passing through $x\in \Omega$.
For $(|\div^c B|+\Le^2)$-a.e. $x\in \Omega$ let $t_x\in (0,1)$ be such that $\gamma_x(t_x)=x$ and let $t^*_x\in(0,1)$ be such that $\gamma_x(t^*_x)\in S$. Since $B_2=1$ the value $t_x$ is unique; in view of Lemma~\ref{lem:flow} the value $t^*_x$ is also unique.

Then the function $\rho$ defined in \eqref{lim-rho} satisfies
\[
\rho(x):=\begin{cases}
1, & t_x<t^*_x;\\
\frac{|AB|}{|CD|}, & t_x>t^*_x.
\end{cases}
\]
In other words, $\rho$ is equal to $1$ before the trajectory of $B$ intersects $S$ and $\rho$ is equal to $\frac{|AB|}{|CD|}$ after this intersection. (See Figure~\ref{fig:flowOfB} for the approximate flow.)

\begin{figure}[!h]
\centering
 \includegraphics{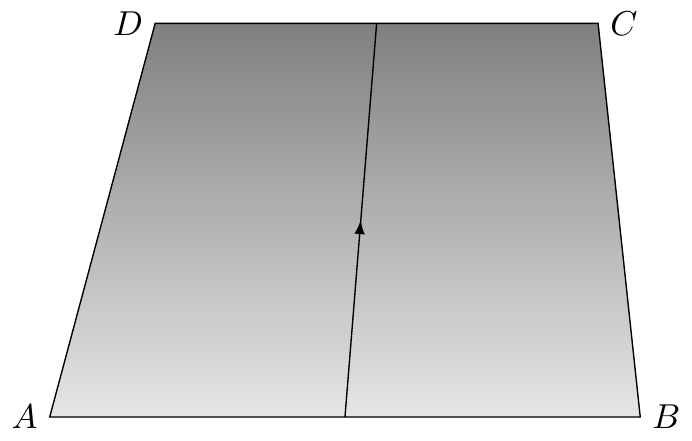}
 \quad
 \includegraphics{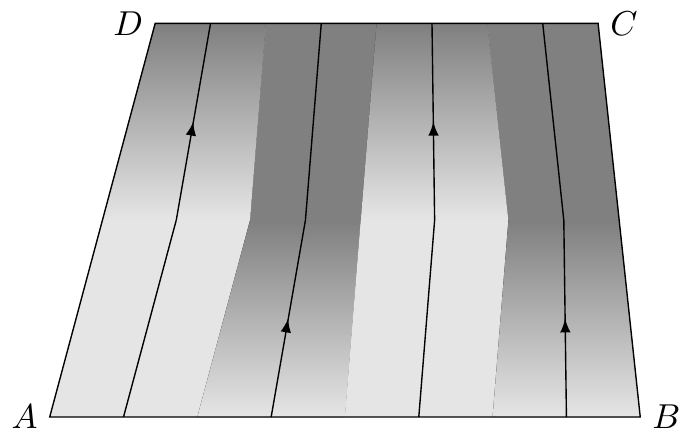}
 \\
 \includegraphics{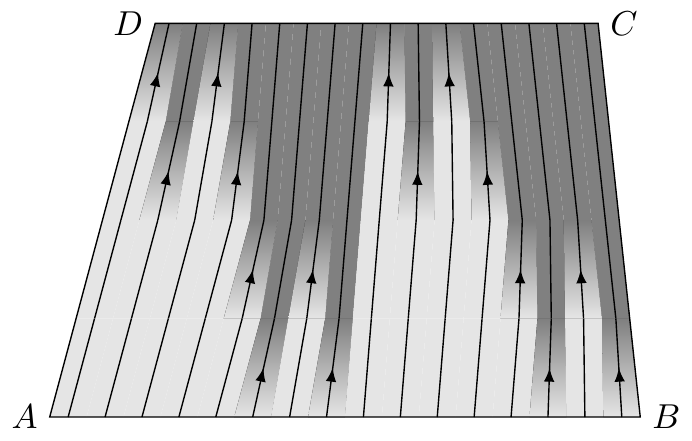}
 \quad
 \includegraphics{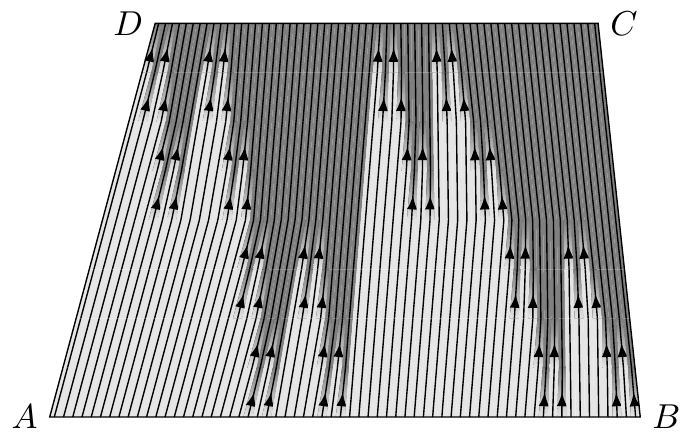}
\caption{Approximation of the flow of $B$.}\label{fig:flowOfB}
\end{figure}

Though later we will prove more general results concerning the trajectories of nearly incompressible vector fields,
we would like to conclude this section by presenting a direct proof of Lemma~\ref{lem:flow}.

\begin{proof}[Proof of Lemma~\ref{lem:flow}]
For any $n\in \NN$ the disjoint family $\gg F_n$ of the trajectories of the approximate vector field $B_n$ can be constructed directly by ``gluing'' the classical trajectories of the vector fields associated with all compressible and incompressible cells in $\cc P_n(\cc C)$. For any $x\in \Omega$ let $C_x \in \gg F$ denote the trajectory of $B_n$ such that $x\in C_x$. Let $\gamma^n_x$ denote the Lipschitz parametrization of $C_x$ which solves the ODE \eqref{ode} for $B_n$. For simplicity we can assume that $\dist(AB,CD)=1$, then for any $x\in \Omega$ the function $\gamma_x^n$ is defined on $(0,1)$.

By induction one can prove that each trajectory $C^n_x$ intersects at most two compressible cells. Let $I^n_x\subset (0,1)$ denote the preimage of this intersection under $\gamma^n_x$. Since the $\ff{Patch}$ operation does not change the incompressible cells, for any $m>n$ the value $\dot\gamma^m_x(t)$ can be different from the value $\dot\gamma^n_x(t)$ only for $t\in I^n_x$. Moreover, by construction of the patches we have $|\dot\gamma^m_x(t)-\dot\gamma^n_x(t)|\le 2^{-n} C$ for some $C>0$. Hence $\dot \gamma^m_x$ converges a.e. on $(0,1)$
as $m\to \infty$. 
Note that for any $x\in \Omega$ there exists unique $t=t_x$ such that $\gamma^n_x(t_x)=x$ for all $n\in \NN$ (in fact since $(B_n)_2 = 1$ for all $n\in \NN$ the value of $t_x$ depends only on $x_2$).
Therefore $\gamma^m_x \to \gamma_x$ uniformly on $(0,1)$ for some injective $\gamma_x \in \Lip((0,1))$ as $m\to \infty$, and $x\in \gamma_x((0,1))$.

Again using invariance of incompressible cells under $\ff{Patch}$ one can check that if $t\in (0,1) \setminus I^n_x$ then
\begin{equation*}
B_n(\gamma^n_x(t)) = B(\gamma^n_x(t)) = B(\gamma_x(t)).
\end{equation*}
The last equality is due to the following fact: if $\gamma^n_x((0,1))$ intersects some incompressible cell for some $n$ then $\gamma^m_x((0,1))$ intersects this cell for all $m\ge n$.
Finally, since $\Le^1(I^n_x) \to 0$ one can pass directly to the limit in $\gamma^n(t)-\gamma^n(t_x) = \int_{t_x}^t B_n(\gamma_n(\tau)) \, d \tau$ and conclude that $\gamma_x$ satisfies the desired ODE \eqref{ode}.

Observe that $C_x:=\gamma_x((0,1))$ can intersect two compressible cells only if $C_x$ contains some of the side edges (which are not parallel to $AB$) of some cell in $\cc P_n(\cc C)$ for some $n\in \NN$. Union of all points of such edges is countably rectifiable hence it is $|\div^c B|$-null set. Let us exclude all such points.

For the remaining points $x$ the corresponding trajectory $C_x$ intersects exactly one compressible cell in $\cc P_n(\cc C)$ for each $n\in \NN$. Hence the projection of $C_x \cap S$ on the vertical axis can be written as an intersection of a sequence of closed segments $\overline{I^n_x}$. Since these segments are nested and $\overline{I^n_x} = [a_n,b_n]$, where $b_n -a_n =2^{-n}$, the intersection of all $\overline{I^n_x}$ is exactly one point.
Hence the intersection of $C_x\cap S$ is a contained in some line parallel to $Ox_1$.
Since by construction $(B)_2=1$, the curve $C_x$ intersects such line in only one point. Hence $C_x\cap S$ is exactly one point.
\end{proof}

%% file: ChainRuleAlongCurve.tex
\section{Chain rule along curves} \label{sec:param}
In this section we study equation \eqref{divuB} in a particular case when $B$ is not a vector field but a special vector measure concentrated on a simple $\Omega$-closed Lipschitz curve. Our main result here is Theorems~\ref{th:eq-along-curve}.

Would like to start by we recalling some standard properties of parameterizations of simple Lipschitz curves in an open set $\Omega \subset \RR^2$ (see also Section~\ref{sec:Notation}).

\begin{definition}\label{def:param}
A Lipschitz parametrization $\gamma\colon [0,\ell] \to \overline{\Omega}$ of a simple curve $C\subset \Omega$ is called \emph{admissible} if $\gamma'\ne 0$ a.e. on $(0,\ell)$.
If, moreover, $|\gamma'|=1$ a.e. on $(0,\ell)$ then $\gamma$ is called \emph{natural}.
\end{definition}

Without loss of generality we can assume that any Lipschitz parametrization $\gamma$ is admissible (if~it is not then we can always reparametrize $C$ by $\tilde \gamma(s): = \gamma(\sigma^{-1}(s))$ where $\sigma(t):=\int_0^t 1_{\gamma' \ne 0}(\tau) \, d\tau$). Moreover, natural parametrization always exists for any simple Lipschitz curve.

\begin{definition}\label{def:agrees-with}
Suppose that $B\colon \Omega \to \RR^2$ is a bounded Borel vector field. We say that a Lipschitz parametrization $\gamma\colon [0,\ell] \to \overline{\Omega}$ of a simple curve $C\subset \Omega$ \emph{agrees with $B$} if
\begin{equation*}
\gamma'(t) \cdot B(\gamma(t)) > 0 \quad \text{for a.e.} \;\; t\in (0,\ell).
\end{equation*}
\end{definition}

\begin{definition}\label{def:traces}
Suppose that $B\colon \Omega \to \RR^2$ is a bounded vector field and $u \colon \Omega \to \RR$. Suppose that $C\subset \Omega$ is an $\Omega$-closed simple Lipschitz curve.
Let $C^+_B:=\gamma([0,\ell))$, where $\gamma$ is any\footnote{clearly $C^+_B$ does not depend on the choice of $\gamma$} Lipschitz parametrization of $C$ which agrees with $B$.
We say that $u$ \emph{has trace $u^+\colon C^+_B \to \RR$ along $(C,B)$} if
\begin{itemize}
\item $u^+(x) = u(x)$ for $\Ha^1$-a.e. $x\in C$;
\item for any Lipschitz parametrization $\gamma\colon [0,\ell]\to \overline{\Omega}$ of $C$ which agrees with $B$ the function $u^+\circ \gamma$ is right-continuous on $[0,\ell)$.
\end{itemize}
The trace $u^-$ is defined in a same way: the only difference is that $u^-$ is defined on $C^-_B:= \gamma((0,\ell])$ and the function $u^-\circ \gamma$ has to be left-continuous.
\end{definition}

Note that $u^+$ and $u^-$ are Borel function, since $\gamma$ is continuous and injective.
Moreover both $u^+$ and $u^-$ are defined in every point of $C$.
In the main result of this section (Theorem~\ref{th:eq-along-curve}) we will provide sufficient condition for existence of these traces.

\begin{proposition}\label{prp:param}
If $\gamma$ is a Lipschitz parametrization of a simple $\Omega$-closed curve $C\subset \Omega$ then
\begin{itemize}
\item $\gamma_\# (|\gamma'| \Le) = \Ha^1 \rest C$;
\item for any measure $\mu$ on $\RR^2$ which is concentrated on $C$ there exists a measure $\nu=: \gamma^{-1}_\# \mu$ on $I_\gamma$ such that $\gamma_\# \nu = \mu$.
\end{itemize}
(Here and further $I_\gamma$ is the domain of $\gamma$ defined by \eqref{param-dom}.)
\end{proposition}
The first claim follows from one-dimensional area formula (see e.g. \cite{AFP}).
Since $\gamma$ is continuous and injective, the image $\gamma(B)$ of any Borel set $B\subset I_\gamma$ is also Borel (see e.g. \cite{Fremlin}, Theorem 423I). 
Therefore to prove the second claim it is sufficient to define $\nu(B):=\mu(\gamma(B))$ for any Borel set $B\subset I_\gamma$.

\begin{definition}\label{def:eq-along-curve}
Suppose that $C\subset \Omega$ is a simple $\Omega$-closed curve, $f\colon \Omega \to \RR$ is a Borel function, $\mu$ is a Radon measure on $\Omega$ and $\gamma$ is a Lipschitz parametrization of $C$. We will say that
\begin{equation*}
f'=\mu
\quad\text{holds in }\ss{D}'(C,\gamma)
\end{equation*}
if $f \circ \gamma \in L^1(I_\gamma)$ and
\begin{equation*}
(f\circ \gamma)' = \gamma^{-1}_\# \mu
\quad\text{in }\ss D'(I_\gamma),
\end{equation*}
where $I_\gamma$ was defined in \eqref{param-dom}.
\end{definition}

\begin{proposition}[Vol'pert's chain rule]\label{prp:Volpert}
Suppose that $u\colon I\to \RR^m$ belongs to $BV(I)$, where $I=(0,\ell)$ or $I=\RR/(\ell \ZZ)$, $m\in \NN$. Then $u$ has classical one-dimensional traces $u^\pm \colon I \to \RR^m$ and for any $\beta \in C^1(\RR^m)$ we have $\beta(u)\in BV(I)$ and
\begin{equation}\label{e-volp}
D(\beta(u)) = f(u^+,u^-) \cdot Du,
\end{equation}
where $Du$ is the $\RR^m$-valued Radon measure representing distributional derivative of $u$ and
\begin{equation} \label{f-def}
f(u^+,u^-) := \int_0^1 \nabla \beta (t u^+ + (1-t) u^-) \, dt
\end{equation}
is so-called Vol'pert \emph{averaged superposition} (see \cite{AFP}).
\end{proposition}

\begin{remark}
In the scalar case when $m=1$ we have
\begin{equation*}
f(u^+,u^-)=
\begin{cases}
\displaystyle
\frac{\beta(u^+)-\beta(u^-)}{u^+ - u^-} & u^+ \ne u^-, \\
\beta'(\tilde u) & u^+ = u^- =: \tilde u.
\end{cases}
\end{equation*}
\end{remark}

\begin{remark} \label{rmrk:traces-classic}
By classical trace $u^+$ of a function $u\in BV(I)$ when $I=(0,\ell)$ we mean the function $u^+(x):=c^+ + Du((0,x])$ where the constant $c^+$ is such that $u=u^+$ a.e. on $(0,\ell)$. Similarly, $u^-(x):=c^- + Du((x,\ell))$, where the constant $c^-$ is such that $u=u^-$ a.e. on $(0,\ell)$.
Note that $u^+$ is right-continuous function from $[0,\ell)$ to $\RR$ and $u^-$ is left-continuous function from $(0,\ell]$ to $\RR$.
When $I=\RR /(\ell \ZZ)$ the traces $u^\pm$ are defined analogously.
\end{remark}

\begin{theorem}\label{th:eq-along-curve}
Suppose that $C \subset \Omega$ is a simple $\Omega$-closed Lipschitz curve, $B\colon \Omega \to \RR^2$ is Borel vector field such that $|B(x)|=1$ and $B(x)$ is tangent to $C$ for $\Ha^1$-a.e. $x\in C$. Suppose that
\begin{equation} \label{divB-on-C}
\div (B \Ha^1\rest C) = 0 \qquad \text{in } \quad \ss D'(\Omega).
\end{equation}
Then the following statements hold:
\begin{enumerate}
\labitem{(i)}{th:eq-along-curve-orient} for any Lipschitz parametrization $\gamma$ of $C$ there exists a constant $\sigma_\gamma \in \{\pm 1\}$ such that
\begin{equation} \label{gamma-dot-vs-B}
\frac{\gamma'(t)}{|\gamma'(t)|} = \sigma_\gamma B(\gamma(t))
\end{equation}
for a.e. $t \in I_\gamma$.
\labitem{(ii)}{th:eq-along-curve-B-param} there exists a Lipschitz parametrization of $C$ which agrees with $B$;
\labitem{(iii)}{th:eq-along-curve-changevar} a Borel function $u\colon \Omega \to \RR$ such that $u \in L^1(\Ha^1\rest C)$ and a Radon measure $\mu$ on $C$ satisfy
\begin{equation} \label{divuB-on-C}
\div ( uB \Ha^1\rest C ) = \mu \qquad \text{in } \quad \ss D'(\Omega)
\end{equation}
if and only if $u\in L^\infty (\Ha^1 \rest C)$ and for any Lipschitz parametrization $\gamma$ of the curve $C$ which agrees with $B$
\begin{equation} \label{u-along-C}
u'=\mu
\quad\text{in }\ss D'(C,\gamma).
\end{equation}
\labitem{(iv)}{th:eq-along-curve-renorm} Suppose that $u\in L^1(\Ha^1\rest C)$ and a Radon measure $\mu$ on $C$ satisfy \eqref{divuB-on-C}. Then
$u$ has traces $u^\pm$ along $(C,B)$
and moreover
for any $\beta \in C^1(\RR)$ we have
\begin{equation} \label{divbetauB-on-C}
\div ( \beta(u)B \Ha^1\rest C) = f(u^+,u^-) \cdot \mu \qquad \text{in } \quad \ss D'(\Omega).
\end{equation}
where the function $f$ is given by \eqref{f-def}.
\end{enumerate}
\end{theorem}

\begin{remark}
From \ref{th:eq-along-curve-renorm} it follows that $\beta(u)$ has traces $\beta(u)^\pm$ along $(C,B)$. Since $\beta$ is continuous these traces are given by $\beta(u)^\pm=\beta(u^\pm)$.
\end{remark}

\begin{remark} \label{rmrk:vector-case}
One can easily generalize Theorem~\ref{th:eq-along-curve} to the case of vector-valued $u$ and $\mu$.
Namely if $u=(u_1, ..., u_m) \in L^1(\Ha^1\rest C)$ and the vector measure $\mu = (\mu_1, ..., \mu_m)$ satisfy \eqref{th:div-disint} \emph{componetent-wise} (i.e. $\div(u_i B \Ha^1 \rest C) = \mu_i$, $i=1, ..., m$) then for any
$\beta\in C^1(\RR^m)$ the function $\beta(u)$ solves \eqref{divbetauB-on-C}.
\end{remark}

\begin{remark}
Note that in Theorem~\ref{th:eq-along-curve} $\nabla \beta$ does not have to be uniformly bounded because, as it follows from \ref{th:eq-along-curve-changevar}, any solution to \eqref{divuB-on-C} has to belong to $\Linf(\Ha^1\rest C)$ and therefore the traces $u^\pm$ are automatically bounded.
\end{remark}

The proof of Theorem~\ref{th:eq-along-curve} is based on the so-called \emph{density Lemma}:
\begin{lemma}[see \cite{ABC1}, Lemma 4.3]\label{lem:density}
Let $a\in L^1(I)$ and let $\mu$ be a Radon measure on $I$, where $I=\RR/(\ell \ZZ)$ or $I=(0,\ell)$, $\ell>0$.
Suppose that $\gamma\colon I\to \Omega$ is an injective Lipschitz function such that $\gamma(0,\ell) \subset \Omega$ and $\gamma'\ne 0$ a.e. on $I$.
Consider the functional
\begin{equation*}
\Lambda(\phi):= \int_I \phi' a \, d\Le + \int_I \phi \, d\mu
\end{equation*}
defined on the space of compactly supported Lipschitz functions $\phi\colon I \to \RR$.

Suppose that $\Lambda(\pfi \circ \gamma)=0$ for any $\pfi\in C_0^1(\Omega)$.
Then $\Lambda(\phi)=0$ for any $\phi\in C_0^1(I)$.
\end{lemma}

In case when $\Omega=\RR^2$ the proof of this result can be found in the preprint \cite{ABC2_pre}.
For convenience of the reader and completeness of the paper we will give the proof of Lemma~\ref{lem:density}
in the end of this section.

\begin{proof}[Proof of Theorem~\ref{th:eq-along-curve}.]
By definition \eqref{divuB-on-C} holds if and only if for any $\pfi \in C_0^\infty(\Omega)$
\begin{equation} \label{divB-on-C-weak}
\int_C u B \cdot \nabla \pfi \, d\Ha^1 + \int_C \pfi \, d\mu = 0.
\end{equation}

Suppose that $\gamma$ is a Lipschitz parametrization of $C$ with domain $I_\gamma$. Then using Proposition~\ref{prp:param} we can rewrite \eqref{divB-on-C-weak} as
\begin{equation*}
0 = \int_C u B \cdot \nabla \pfi \, d\gamma_\# (|\gamma'| \Le)
+ \int_C \pfi \,d \gamma_\#\mu_\gamma 
= \int_{I_\gamma} (uB\cdot \nabla \pfi) \circ \gamma \cdot |\gamma'| \,d\Le
+ \int_{I_\gamma} \pfi \circ \gamma \, \,d\mu_\gamma
\end{equation*}
where $\mu_\gamma:=\gamma^{-1}_\# (\mu)$ (see Section~\ref{sec:param}).

Since $B(x)$ is unit tangent to $C$ for $\Ha^1$-a.e. $x\in C$, there exists a function $\sigma_\gamma\colon C \to \{\pm 1\}$ such that \eqref{gamma-dot-vs-B} holds for a.e. $t\in I_\gamma$. Hence we can further rewrite \eqref{divB-on-C-weak} as
\begin{equation} \label{changevar+}
0= \int_{I_\gamma} (u\sigma_\gamma \cdot \nabla \pfi) \circ \gamma \cdot \gamma' \,d\Le
+ \int_{I_\gamma} \pfi \circ \gamma \, \,d \mu_\gamma.
\end{equation}

In view of \eqref{divB-on-C} we can substitute into \eqref{changevar+} $u\equiv 1$ and $\mu \equiv 0$ and obtain thus
\begin{equation*}
\int_{I_\gamma} (\sigma_\gamma \circ \gamma) (\pfi \circ \gamma)' \, d\Le = 0.
\end{equation*}
In view of Lemma~\ref{lem:density} this equation holds for any $\pfi \in C_0^1(\Omega)$ if and only if
\begin{equation*}
\int_{I_\gamma} (\sigma_\gamma \circ \gamma) \phi' \, d\Le = 0
\end{equation*}
for any $\phi\in C_0^1(I_\gamma)$. But this is equivalent to $(\sigma_\gamma\circ \gamma)'=0$ in $\ss D'(I_\gamma)$. Therefore either $\sigma_\gamma \equiv +1$ or $\sigma_\gamma \equiv -1$. Hence \ref{th:eq-along-curve-orient} is proved.

By changing, if necessary, the orientation of $I_\gamma$ we can always achieve that $\sigma_\gamma \equiv +1$. Hence \ref{th:eq-along-curve-B-param} is proved.

Now we can turn back to \eqref{changevar+} with $u$ and $\mu_\gamma$. Assuming $\sigma \equiv 1$ we can rewrite \eqref{changevar+} as
\begin{equation} \label{changevar++}
\int_{I_\gamma} (u \circ \gamma) (\pfi \circ \gamma)' \, d \Le + \int_{I_\gamma} \pfi \circ \gamma \, \, d\mu_\gamma = 0.
\end{equation}

We would like to use again Lemma~\ref{lem:density}, but at this point we do not know if $u\circ \gamma \in L^1(I_\gamma)$; since $u \in \Ha^1 \rest C$ and $\gamma_\# (|\gamma'|\Le) = \Ha^1 \rest C$ we only know that $u\circ \gamma \in L^1(|\gamma'| \Le)$. Therefore let us assume for a moment that the parametrization $\gamma$ is natural (clearly such $\gamma$ exists). Then immediately we obtain that $u \circ \gamma \in L^1(I_\gamma)$. 

Now we can use Lemma~\ref{lem:density} in view of which \eqref{changevar++} holds for any $\pfi \in C_0^1(\Omega)$ if and only if
\begin{equation*}
\int_{I_\gamma} (u \circ \gamma) \phi' \, d \Le 
+ \int_{I_\gamma} \pfi \circ \gamma \, \, d\mu_\gamma = 0
\end{equation*}
for any $\phi\in C_0^1(I)$. But this is equivalent to $(u\circ \gamma)' = \mu_\gamma$ in $\ss D'(I_\gamma)$ (which, by definition, is equivalent to \eqref{u-along-C}).
Hence $u\circ \gamma \in BV(I_\gamma)$ and consequently there exists a constant $c>0$ such that $|u\circ \gamma|< c$ a.e. on $I_\gamma$. Since $\gamma$ is natural parametrization this means that $|u(x)|<c$ for $\Ha^1$-a.e. $x\in C$. Hence \ref{th:eq-along-curve-changevar} is proved when $\gamma$ is natural parametrization of $C$.

But if $u\in \Linf(\Ha^1 \rest C)$ then $u\circ \gamma \in L^1(I_\gamma)$ for \emph{any} Lipschitz parametrization of $C$, and hence we can repeat our argument and apply Lemma~\ref{lem:density} to \eqref{changevar++}.
This completes the proof of \ref{th:eq-along-curve-changevar}.

It remain to prove \ref{th:eq-along-curve-renorm}. Since $u\circ \gamma \in BV(I_\gamma)$ there exist classical one-dimensional traces $(u_\gamma)^\pm$ of $u\circ \gamma$ (in a sense we mentioned in Remark~\ref{rmrk:traces-classic}). Then one can easily check that $u^\pm$ defined by $u^\pm\circ \gamma = (u\circ \gamma)^\pm$ are traces of $u$ along $(C,B)$.

Finally applying Vol'pert's chain rule to $u\circ \gamma$ (see Proposition~\ref{prp:Volpert}) we obtain that
\[
\beta(u\circ \gamma)' = f((u\circ \gamma)^+, (u\circ \gamma)^- ) \cdot \mu_\gamma \quad \text{in} \; \ss D'(I_\gamma)
\]
(where $f$ is defined by \eqref{f-def}) which implies \eqref{divbetauB-on-C} in view of \ref{th:eq-along-curve-changevar}.
\end{proof}

\begin{proof}[Proof of Lemma~\ref{lem:density}]
Let $u\colon I \to \RR$ be a $BV$ function such that $u' = -\mu$ in $\ss D'(I)$. Then the functional $\Lambda$ can be written as
\[
\Lambda(\phi) = \int_0^\ell \phi' \alpha \, d \Le
\]
where the function $\alpha\in L^1(I)$ is given by $\alpha = a + u$.

If $\Lambda(\phi)\ne 0$ for some $\phi\in C_0^1(I)$ then $\alpha'\ne 0$ in $\ss D'(I)$.
In view of Rademacher's theorem there exists a negligible set $N\subset I$ such that for any $t\in I \setminus N$ there exists $\gamma'(t) \ne 0$. Since a.e. $t\in I\setminus N$ is a point of $L^1$ approximate continuity of $\alpha$ (with approximate limit $\tilde \alpha(t)$) and $\alpha ' \ne 0$ we can choose $t_1, t_2 \in I \setminus N$ in such way that $\alpha_1:=\tilde \alpha(t_1) \ne \alpha(t_2)=:\alpha_2$ and $0<t_1<t_2<\ell$. 

For given $\delta>0$ let $\Gamma_1:= \gamma([t_1+\delta, t_2-\delta])$ and $\Gamma_2:=\gamma([0,t_1 - \delta] \cup [t_2 + \delta,\ell])$ (in case when $I=(0,\ell)$ we assume that the Lipschitz function $\gamma$ is defined by continuity at $0$ and $\ell$). 
Since $\gamma$ is injective the sets $\Gamma_1$ and $\Gamma_2$ are disjoint compacts in $\RR^2$, there exists a function $\psi_1\in C_0^\infty(\RR^2)$ such that
$\psi_1$ is $1$ on $\Gamma_1$, $\psi_1$ is $0$ on $\Gamma_2$ and $\|\nabla \psi_1\|_\infty \le C/\dist(\Gamma_1, \Gamma_2)$. 
We can also find $\psi_2 \in C_0^\infty (\Omega)$ such that $\psi_2$ is $1$ on $\Gamma_1$ and $\|\nabla \psi_2\|_\infty \le \dist(\Gamma_1, \d \Omega) \le \dist (\gamma([t_1, t_2]), \d \Omega)$.
Then $\psi:= \psi_1 \psi_2 \in C_0^\infty(\Omega)$ satisfies
\begin{itemize}
\item $\psi$ is $1$ on $\Gamma_1$
\item $\psi$ is $0$ on $\Gamma_2$
\item $\|\nabla \psi\|_\infty \le C /\dist(\Gamma_1, \Gamma_2)$
\end{itemize}
provided that $\dist(\Gamma_1, \Gamma_2) < \dist(\gamma([t_1, t_2], \d \Omega)$ which clearly is satisfied when $\delta<\delta_1$ for some sufficiently small $\delta_1>0$, since $\gamma$ is Lipschitz.

The derivative of $\psi\circ \gamma$ is nonzero only on $(t_i-\delta, t_i+\delta)$, $i=1,2$, therefore by construction of $\psi$
\begin{equation*}
\Lambda(\psi \circ \gamma) = \int_{t_1-\delta}^{t_1+\delta} \alpha \cdot (\psi \circ \gamma)' \, dt + \int_{t_2-\delta}^{t_2+\delta} \alpha \cdot (\psi \circ \gamma)' \, dt = 
\alpha_1 - \alpha_2 + R
\end{equation*}
where the ``error term'' $R$ is given by
\begin{equation*}
R:= \int_{t_1-\delta}^{t_1+\delta} (\alpha(t)-\alpha_1) (\psi \circ \gamma)' \, dt  
+ \int_{t_2-\delta}^{t_2+\delta} (\alpha(t)-\alpha_2) (\psi \circ \gamma)' \, dt  
\end{equation*}

Since $\gamma$ is Lipschitz and $t_{1,2}$ are points of $L^1$-approximate continuity of $\alpha$ for any $\eps>0$ there exists $\delta_\eps>0$ such that if $\delta<\delta_\eps$ then
\begin{equation}\label{RMainEst}
|R| \le \eps \|\nabla \psi\|_\infty \|\gamma'\|_\infty \delta
\le C \eps \frac{\delta}{\dist(\Gamma_1, \Gamma_2)}
\end{equation}
in view of the definition of $\psi$. We claim that there exists $c>0$ such that
\begin{equation}\label{dist-est}
 \dist(\Gamma_1, \Gamma_2) \ge c \delta
\end{equation}
when $\delta$ is sufficiently small. 
From \eqref{RMainEst} and \eqref{dist-est} we immediately conclude that for any $\xi>0$ there exists $\delta>0$ such that the function $\psi$ constructed as described above satisfies
\begin{equation*}
|\Lambda(\psi \circ \gamma)| \ge |\alpha_1 - \alpha_2| - \xi
\end{equation*}
which contradicts the assumptions about the functional $\Lambda$ when $\xi < |\alpha_1-\alpha_2|$.

It remains to prove the estimate \eqref{dist-est}.

Since $\Gamma_1$ and $\Gamma_2$ are compacts there exist $t'\in \gamma^{-1}(\Gamma_1)$ and $t''\in \gamma^{-1}(\Gamma_2)$ such that $\dist(\Gamma_1,\Gamma_2) = |\gamma(t')-\gamma(t'')|$.
(Note that $t'$ and $t''$ depend on $\delta$.)

By definition of $t_{1,2}$ the number $s:= \min(|\gamma'(t_1)|, |\gamma'(t_2)|)$ is strictly positive and
there exists $\delta_2>0$ such that
\begin{equation}\label{gammadiff}
|\gamma(t_i+ \tau) - \gamma(t_i) - \gamma'(t_i)\tau|  < (s/2) \tau
\end{equation}
for any $\tau\in \RR$ with $|\tau|<\delta_2$, where $i=1,2$.

Since $\gamma$ is injective there exists $r>0$ such that for any $t, \tau \in I$ with $|t-\tau|\ge \delta_2$ we have $|\gamma(t) - \gamma(\tau)|\ge r$. Therefore if $|t'-t''|\ge \delta_2$ then we have $|\gamma(t')-\gamma(t'')| \ge r$, hence \eqref{dist-est} is satisfied with $c=1$ provided that $\delta<\delta_3:=r$.

Otherwise, if $|t'-t''|< \delta_2$ then \emph{both} $t'$ and $t''$ have to belong to \emph{the same} interval $(t_i-\delta_2, t_i+\delta_2)$, where $i=1,2$. Then in fact $t'\in (t_i + \delta, t_i + \delta_2)$, $t''\in (t_i - \delta_2, t_i-\delta)$ and so $2 \delta \le|t'-t''| = |t'-t_i| + |t''-t_i|$. Using \eqref{gammadiff} we can estimate
\begin{equation}
|\gamma(t') - \gamma(t'')|\ge s |t'-t''| - (s/2) |t' -t_i| - (s/2)|t''-t_i| = (s/2) |t'-t''| \ge s \delta
\end{equation}
hence \eqref{dist-est} is satisfied with $c=s$.
\end{proof}

%% file: StructureOfLevelSets.tex
\section{Disintegrations of measures and structure of level sets of Lipschitz functions}
In this section we present the main tools which we use later for analysis of steady nearly incompressible vector fields. Using these tools in the next section we will reduce \eqref{divuB} to \eqref{divuB-on-C}.

\subsection{Disintegration theorems}\label{ss:MeasureTheory}
Let $\Omega\subset \RR^2$ be an open set.
Recall that a family $\{\mu_h\}_{h\in \RR}$ of Radon measures on $\Omega$ is called \emph{Borel} if for any continuous and compactly supported test function $\pfi\colon \Omega \to \RR$ the map $h\mapsto \int \pfi \,d \mu_h$ is Borel measurable. The following proposition is a corollary of the well-known Disintegration Theorem (see e.g. \cite{AFP}, \S 2.5):

\begin{proposition}\label{prp:disint}
Suppose that $H\colon \Omega \to \RR$ is a Borel function, $\mu$ is a Radon measure on $\Omega$ and $\nu$ is a non-negative Radon measure on $\RR$ such that $H_\# |\mu| \ll \nu$.
Then there exists a Borel family $\{\mu_h\}_{h\in \RR}$ of Radon measures on $\Omega$ such that
\begin{itemize}
\item $\mu_h$ is concentrated on the level set $H^{-1}(h)$ for every $h\in\RR$
\item $\mu$ can be decomposed as
\begin{equation*}
\mu = \int_{\RR} \mu_h \,d \nu(h)
\end{equation*}
which means that for any Borel set $E\subset \Omega$ we have $\mu(E)=\int_\RR \mu_h(E) \,d \nu(h)$.
\end{itemize}
\end{proposition}
We will call the family $\{\mu_h\}_{h\in \RR}$ from Proposition~\ref{prp:disint} a \emph{disintegration} of $\mu$ with respect to $(H, \nu)$. Such disintegration is unique in the following sense: if $\{\bar \mu_h\}_{h\in \RR}$ is another disintegration then $\mu_h=\bar \mu_h$ for $\nu$-a.e. $h\in \RR$.

When the function $H$ is a Lipschitz function and $\mu = |\nabla H|\Le^2$, we can characterize the disintegration $\{\mu_h\}_{h\in \RR}$ more precisely using the coarea formula (see e.g. \cite{AFP}, Theorem~2.93):

\begin{proposition}\label{prp:coarea}
Suppose that $H\colon \Omega \to \RR$ is a Lipschitz function with bounded support.
Then
\begin{equation*}
|\nabla H|\Le^2 = \int_{\RR} \Ha^1 \rest E_h \, dh
\end{equation*}
where $E_h$ denotes the level set $H^{-1}(h)$.
\end{proposition}

\subsection{Connected components}

Later we will use a technical lemma, according to which any connected component of a compact set can be in some sense separated from its complement by an appropriate sequence of test functions:
\begin{lemma}[cf. \cite{ABC2}, Section~2.8]\label{lem:sep-by-test}
If $E\subset \RR^d$ is compact then for any connected component $C$ of $E$ there exists a sequence $(\pfi_n)_{n\in\NN} \subset C_0^\infty(\RR^d)$ such that 
\begin{enumerate}
  \item $0\le \pfi_n \le 1$ on $\RR^d$ and $\pfi_n\in \{0,1\}$ on $E$ for all $n \in \NN$;
  \item for any $x\in C$ we have $\pfi_n(x) = 1$ for all $n\in \NN$;
  \item for any $x\in E \setminus C$ we have $\pfi_n(x) \to 0$ as $n\to \infty$;
  \item for any $n \in \NN$ we have $\supp \nabla \pfi_n \cap E = \emptyset$.
\end{enumerate}
\end{lemma}
Though essentially this lemma is a corollary of the results from Section~2.8 of \cite{ABC2}, we include its elementary proof for convenience of the reader.
\begin{proof}
First we claim that for any $x\in E\setminus C$ there exists a closed and open in $E$ set $F_x$ such that
\begin{itemize}
\item $C \subset F_x$;
\item $x \in E \setminus F_x$.
\end{itemize}
Indeed, let $C_x$ denote the connected component of $E$ such that $x\in C_x$. Since $E$ is compact, $C$ can be written as an intersection of all closed and open in $E$ sets $F\subset E$ such that $C \subset F$ (see \cite{Eng}, Theorems 6.1.22 and 6.1.23). For any such set $F$ either $C_x \subset F$ or $C_x \subset E\setminus F$, because otherwise we obtain contradiction with the fact that $C_x$ is connected. If $C_x \subset F$ for all such sets $F$ then $C_x \subset C$, which contradicts the assumption that $x\notin C$ and $x\in C_x$. Therefore $x\in E \setminus F$ for at least one closed and open in $E$ set $F=F_x$ such that $C \subset F$.

Now we can write $C$ as the intersection
\begin{equation*}
C = \bigcap_{x\in E \setminus C} F_x.
\end{equation*}
Since $E$ is closed, every set closed in $E$ is closed.
By Lindel\"of's Lemma there exists a countable family $\{F_k\}_{k\in \NN} \subset \setof{F_x}{x\in E \setminus C}$ such that
$\cup_{x\in E \setminus C} E \setminus F_x = \cup_{k\in \NN} E \setminus F_k$
or in other words
\begin{equation*}
C = \cap_{k\in \NN} F_k.
\end{equation*}

Let us define $G_n := \cup_{k=1}^n F_k$, $n\in \NN$. By construction these sets are closed and open in $E$, hence for any $n$ we can write $E$ as a union of two disjoint closed sets, one of which contains $C$:
$E=G_n \cup (E\setminus G_n)$. Since $G_n$ and $E\setminus G_n$ are closed and bounded for any $n\in \NN$ we can construct a function $\pfi_n\in C_0^\infty(\RR^d)$ such that
\begin{itemize}
\item $0\le \pfi_n \le 1$ on $\RR^d$;
\item $\pfi_n(x) = 1$ for any $x\in G_n$;
\item $\pfi_n(x) = 0$ for any $x\in E \setminus G_n$;
\item $\supp \nabla \pfi_n \cap E = \emptyset$.
\end{itemize}
By construction of the sets $G_n$ we have $1_{G_n}(x) \to 1_C(x)$ for any $x\in E$ as $n\to \infty$, so the sequence $(\pfi_n)_{n\in \NN}$ we constructed has the desired properties.
\end{proof}

\subsection{Structure of level sets of Lipschitz functions}

Let $\Omega\subset \RR^2$ be a bounded open set.
In this section, following \cite{ABC2}, we characterize the structure of the level sets of a Lipschitz function $H\colon \Omega\to \RR$.
According to \cite{ABC2} such level sets essentially consist of closed simple curves which can be parametrized by injective Lipschitz functions (see Definition~\ref{def:param}):

\begin{theorem}[see \cite{ABC2}, Theorem 2.5]\label{th:ABC2}
Let $\Omega \subset \RR^2$ be an open set and suppose that $H\colon \Omega \to \RR$ is a Lipschitz function with bounded support.
Then the following statements hold for a.e. $h\in H(\Omega)$:
\begin{enumerate}
\labitem{1}{th:ABC2-1} the level set $E_h:=H^{-1}(h)$ is $1$-rectifiable and $\Ha^1(E_h)<+\infty$;
\labitem{2}{th:ABC2-2} for $\Ha^1$-a.e. $x\in E_h$ the function $H$ is differentiable at $x$ and $\nabla H(x) \ne 0$;
\labitem{3}{th:ABC2-3} $\ff{Conn^*}(E_h)$ is countable and $\Ha^1(E_h \setminus E_h^*)=0$;
\labitem{4}{th:ABC2-4} every $C\in \ff{Conn^*}(E_h)$ is an $\Omega$-closed simple Lipschitz curve.
\end{enumerate}
\end{theorem}

\begin{definition} \label{def:reg-val}
We will say that $h\in \RR$ is a \emph{regular value} of $H$ if the corresponding level set $E_h$ either is empty or satisfies the properties \ref{th:ABC2-1}--\ref{th:ABC2-4} of Theorem~\ref{th:ABC2}. In this case we will also call the level set $E_h$ \emph{regular}.
\end{definition}
Note that we consider empty level sets as regular only for brevity. 
This convention allows us to summarize Theorem~\ref{th:ABC2} as follows: for a.e. $h\in \RR$ the corresponding level set $H^{-1}(h)$ is regular.

When $\Omega=\RR^2$ Theorem~\ref{th:ABC2} was proved in \cite{ABC2}.
For a generic open set $\Omega$ we will prove Theorem~\ref{th:ABC2} by reduction to this case using the following simple lemma:

\begin{lemma}\label{lem:conn-comp-cap}
If $E \subset \RR^d$ and $\Omega \subset \RR^d$ then $I$ is a connected component of $E\cap \Omega$ if and only if $I$ is a connected component of $C\cap \Omega$ for some connected component $C$ of $E$.
\end{lemma}
\begin{proof}
Suppose that $I$ is a connected component of $x$ in $E\cap \Omega$ for some $x\in E\cap \Omega$.
Let $C$ denote the connected component of $x$ in $E$. 
Then $I\subset C$, since $I$ is a connected subset of $E$ (and $I\subset \Omega$).

Let $C'$ denote the connected component of $x$ in $C\cap \Omega$.
Then $I \subset C'$, since $I$ is a connected subset of $C$ (and $I\subset \Omega$).

On the other hand $C'$ is a connected subset of $E$ and $C' \subset \Omega$, so $C' \subset I$.
Therefore $I=C'$.
\end{proof}

\begin{proof}[Proof of Theorem~\ref{th:ABC2}]

Let $\Theta\colon \RR^2 \to \RR$ be a Lipshitz extension of $H$ to $\RR^2$ (see \cite{AFP}, Proposition 2.12) and let $\chi\colon \RR^2 \to \RR$ be a smooth function with compact support such that $\chi(x) = 1$ for any $x\in \supp H$. Then the function $G(x):=\chi(x) \Theta(x)$ is Lipschitz and has compact support, therefore we can apply the original Theorem 2.5 from \cite{ABC2} for a.e. $h\in \RR$.

For any $h\in \RR$ let $F_h:=G^{-1}(h)$. Then $E_h = F_h \cap \Omega$. By Theorem 2.5 from \cite{ABC2} the level set $F_h$ has the properties stated in Theorem~\ref{th:ABC2}.

Observe that the following implications hold:
\begin{enumerate}
\item if $F_h$ is $1$-rectifiable and $\Ha^1(F_h)<+\infty$ then $E_h$ is $1$-rectifiable and $\Ha^1(E_h)<+\infty$;
\item if $G$ is differentiable $\Ha^1$-a.e. on $F_h$ then $H$ is differentiable $\Ha^1$-a.e. on $E_h$;
\item if $\ff{Conn^*}(F_h)$ is countable and every $C\in \ff{Conn^*}(F_h)$ is closed simple Lipschitz curve then any $C\in \ff{Conn^*}(E_h)$ is an $\Omega$-closed simple Lipschitz curve and $\ff{Conn^*}(E_h)$ is countable.
\end{enumerate}
Only the last implication is not trivial, so let us prove it. Suppose that $C\in \ff{Conn^*}(F_h)$ is a closed simple Lipschitz curve on $\RR^2$. Suppose that $\gamma \colon I\to \RR^2$ is a Lipschitz parametrization of $C$, where $I=\RR/(L\ZZ)$ and $L>0$. Since $C\cap \Omega$ is open in the induced topology on $C$, $\gamma^{-1}(C \cap \Omega)$ is open in $I$. Since $I_\gamma$ is one-dimensional, $\gamma^{-1}(C \cap \Omega)$ is a countable union of disjoint open subsets of $I_\gamma$.
Image of any of these subsets under $\gamma$ clearly is a simple $\Omega$-closed Lipschitz curve.

By Lemma~\ref{lem:conn-comp-cap} any $C'\in \ff{Conn^*}(F_h\cap \Omega)$ is a connected component of $C\cap \Omega$ for some $C\in \ff{Conn^*}(F_h)$. Therefore any $C'\in \ff{Conn^*}(E_h)$ is a simple $\Omega$-closed Lipschitz curve and moreover $\ff{Conn^*}(E_h)$ is countable, since $\ff{Conn^*}(F_h)$ is countable.
\end{proof}

\subsection{Monotone Lipschitz functions and regular level sets}

In the proofs of some technical statements (related to measurability of the functions which we construct)
it will be convenient to decompose a Lipschitz function $H$ into a sum of so-called \emph{monotone} ones.
In this section we recall (without proof) a well-known result which generalizes the classical one-dimensional Jordan's decomposition.

We also prove in this section some auxiliary results which will allow us later to reduce the case of general Lipschitz function $H$ to the case of monotone one.

\begin{definition} \label{def:monotone}
A Lipschitz function $H\colon \RR^2 \to \RR$ is called \emph{monotone} if the level sets $H^{-1}(h)$ are connected for any $h\in \RR$.
\end{definition}

\begin{theorem}[see \cite{BT}] \label{th:decomposition}
Suppose that $H\colon \RR^2\to \RR$ is compactly supported Lipschitz function.
Then there exists a countable family $(H_i)_{i\in \NN}$ of compactly supported monotone Lipschitz functions on $\RR^2$
such that $H = \sum_{i=1}^\infty H_i$ and for $i\ne j$ the set $\{\nabla H_i \ne 0\} \cap \{\nabla H_j \ne 0\}$ is $\Le^2$-negligible.
\end{theorem}

%% file: DisintegrationOfDiv.tex
\section{Disintegration of the divergence equation with divergence-free vector field in $\RR^2$}\label{sec:divub}
Let $\Omega \subset \RR^2$ be an open set. 
In this section we derive a new necessary and sufficient condition for a bounded function $u$ and a Radon measure $\mu$ to satisfy \eqref{divuB} in $\Omega$
with a divergence-free vector field $B$. 
The core result of this section is Theorem~\ref{th:div-disint},
which links \eqref{divuB} with a family of equations \eqref{divuB-on-C} along the level sets of the Lipschitz function $H\colon \Omega \to \RR$ such that $B = \nabla^\perp H$.

Using this criteria we prove existence of a disjoint set of trajectories of $B$ which cover the set $\Omega$,
up to a $(|B|\Le^2+|\mu|)$-negligible set. 
Our construction relies on the observation that the regular level sets of the function $H$ in fact are the integral curves of $B$ (this result also extends to nearly incompressible vector fields).

Applying Theorem~\ref{th:eq-along-curve} we prove that any solution $u$ of \eqref{divuB} has bounded variation along these trajectories,
and therefore there exist classical traces $u^\pm$ of $u$ along the integral curves of $B$.
Finally, we use the functions $u^\pm$ to solve the chain rule problem for divergence-free vector fields $B$.
Namely, we prove that for any measure Radon $\mu$ and $\beta\in C^1(\RR)$ there exists a Radon measure $\nu$ for which \eqref{divbeta} holds,
$\nu$ is absolutely continuous with respect to $\mu$ and the density of $\nu$ with respect to $\mu$
can be characterized using the functions $u^\pm$.

\subsection{Reduction to equation on connected components of level sets} \label{sec:disintegration}
Instead of assuming that $B$ is divergence-free here we will suppose directly that $B=\nabla^\perp H$ for an appropriate Lipschitz function.
In case of simply connected domain clearly there is no difference between these assumptions.

\begin{theorem} \label{th:div-disint}
Suppose that $B\colon \Omega\to \RR^2$ is a bounded vector field such that
\begin{equation} \label{B=nabla}
B = \nabla^\perp H \quad \text{a.e. in } \; \Omega
\end{equation}
for some Lipschitz function $H\colon \Omega \to \RR$ with bounded support.
Suppose that $\mu$ is a Radon measure on $\Omega$. Let $\sigma_\mu$ denote a Radon measure on $\RR$ such that $\sigma_\mu \perp \Le^1$ and $H_\# |\mu| \ll \Le^1 + \sigma_\mu$.
Let $\{\mu_h\}_{h\in\RR}$ be a disintegration of $\mu$ with respect to $(H, \Le^1+\sigma_\mu)$. 
Then $u\in L^1(|B|\Le^2)$ solves \eqref{divuB} if~and~only~if
\begin{enumerate}
\labitem{(1)}{th:div-disint-1} for a.e. $h\in \RR$
\begin{enumerate}
\labitem{(1a)}{th:div-disint-1a} $|\mu_h| (E_h \setminus E_h^*) = 0$, where $E_h:= H^{-1}(h)$;
\labitem{(1b)}{th:div-disint-1b} for any $C\in \ff{Conn^*}(E_h)$ which is a closed simple curve we have
\begin{equation} \label{divub-on-curve}
\div \rb{u \frac{B}{|B|} \Ha^1 \rest C} = \mu_h \rest C
\quad \text{in } \; \ss D'(\Omega)
\end{equation}
\end{enumerate}
\labitem{(2)}{th:div-disint-2} for $\sigma_\mu$-a.e. $h$ we have $\mu_h = 0$.
\end{enumerate}
\end{theorem}

Before proving this theorem we would like to note the following corollary:

\begin{theorem} \label{th:flow}
Suppose that $B$ and $H$ satisfy the assumptions of Theorem~\ref{th:div-disint}.
Then there exists a negligible set $N \subset \RR$ such that
\begin{itemize}
\item $\gg F:= \setof{C}{C\in \ff{Conn^*}(E_h), \; h\in \RR\setminus N}$ is a disjoint family of trajectories of $B$
\item $F:=\cup_{C \in \gg F} C$ is Borel and $\Le^2(\{B\ne 0\} \setminus F)=0$
\end{itemize}
Moreover, if $u\in L^1(|B| \Le^2)$ and Radon measure $\mu$ satisfy \eqref{divuB} then $(|B|\Le^2+ |\mu|)(\Omega \setminus F) = 0$.
\end{theorem}

\begin{proof}
By definition of $B$ we have $\div B = 0$ hence $u\equiv 1$ solves \eqref{divuB}.
Therefore by Theorem~\ref{th:div-disint} there exists a negligible set $N\subset \RR$ such that any $h\in \RR\setminus N$ the level set $E_h:= H^{-1}(h)$ is regular
and any nontrivial connected component $C\in \ff{Conn^*}(E_h)$ is an $\Omega$-closed simple Lipschitz curve which satisfies
\begin{equation*}
\div\rb{\frac{B}{|B|} \Ha^1\rest C} = 0.
\end{equation*}
By Theorem~\ref{th:eq-along-curve} for any such $C$ there exists a Lipschitz parametrization $\gamma \colon[0,\ell] \to \RR^2$ of $C$ such that
\eqref{gamma-dot-vs-B} holds with $\sigma_\gamma \equiv 1$.
Then it is possible to redefine the function $\gamma$ in such a way that it will satisfy \eqref{ode}.
This was proved in \cite{ABC1} (see Lemma 2.11), but for convenience of the reader let us recall the argument.

Since $1/|B| \in L^1(|B| \Le^2)$, by Proposition~\ref{prp:coarea} for a.e. $h\in \RR$ we have $|B(x)|\ne 0$ for $\Ha^1$-a.e. $x\in E_h$ and $1/|B| \in L^1(\Ha^1 \rest E_h)$.
Hence without loss of generality we can assume that $1/|B| \in L^1(\Ha^1 \rest C)$.
On the other hand by Proposition~\ref{prp:param} we have $\gamma_\# (|\gamma'| \Le) = \Ha^1 \rest C$
and by \eqref{gamma-dot-vs-B} $|\gamma'|=1$ a.e. on $(0,\ell)$. Therefore $1/|B\circ \gamma| \in L^1(0, \ell)$.

Since the function $f:=|B\circ \gamma|$ is strictly positive a.e. on $(0,\ell)$ and $1/f$ is integrable,
the function $F(t):=\int_0^t \frac{d\xi}{f(\xi)}$ is strictly increasing and it maps $[0,\ell]$ to $[0,L]$, where $L:=F(\ell)$.
Hence there exists a strictly increasing inverse function $\tau:=F^{-1}$ which maps $[0,L]$ to $[0,\ell]$.

Let $M\subset [0,\ell]$ denote a negligible set such that any $t\in [0,\ell] \setminus M$ is a Lebesgue point of $f$ and $f(t)\ne 0$.
Since $F_\# (F'\Le \rest [0,\ell]) = \Le \rest [0,L]$, the set $F(M) \subset [0,L]$ is also negligible.
Therefore for a.e. $t\in [0,L]$ there exists
\begin{equation*}
\tau '(t) = f(\tau(t)) > 0.
\end{equation*}
Since $f$ is bounded, the function $\tau$ is Lipschitz.
Finally we remark that $\tau$ satisfies the classical Barrow's formula $t = \int_0^{\tau(t)} \frac{d\xi}{f(\xi)}$
for any $t\in [0,L]$.

Now when the function $\tau$ is constructed we can compute
\begin{equation*}
\wave \gamma '(t) = \gamma ' (\tau(t)) \cdot \tau'(t) = B(\gamma(\tau(t))) = B(\wave \gamma(t))
\end{equation*}
for a.e. $t\in(0,\tau^{-1}(\ell))$.

Let $\gg F:= \setof{C}{C\in \ff{Conn^*}(E_h), \; h\in \RR\setminus N}$. Since connected components of level sets are pairwise disjoint, the family $\gg F$ is disjoint. Since elements of $\gg F$ are integral curves of $B$, we conclude that $\gg F$ is a disjoint family of trajectories of $B$.

Let $E$ denote the union of \emph{all} connected the components $C$ of the level sets of $H$ such that $\Ha^1(C)>0$. By Proposition 6.1 from \cite{ABC2} the set $E$ is Borel.
Therefore the set $F = E \setminus H^{-1}(N)$ is Borel, since $N$ can always be chosen to be Borel.

Finally by Theorem~\ref{th:div-disint} we have
\begin{itemize}
\item $\Ha^1(E_h \setminus E_h^*) = 0$ for a.e. $h\in \RR$ hence by Coarea formula $(|B|\Le^2)(\Omega \setminus F) =0$;
\item $|\mu_h|(E_h \setminus E_h^*) = 0$ for a.e. $h\in \RR$;
\item in fact \ref{th:div-disint-2} implies $H_\# |\mu| \ll \Le^1$, hence $|\mu|\rest H^{-1}(N) = 0$.
\end{itemize}
Therefore $|\mu|(\Omega \setminus F) = 0$.
\end{proof}

The theorem above implies existence of disjoint family of trajectories of nearly incompressible vector fields
in view of the following remark:
\begin{remark} \label{r:flow}
The elements of $\gg F$ also are the trajectories of $r B$ for any function $r\colon \Omega \to \RR$
such that $0<C_1 \le r \le C_2$ a.e. in $\Omega$ for some constants $C_1$ and $C_2$.
\end{remark}
Indeed, one only has to appropriately reparametrize each connected component $C\in \ff{Conn^*}(E_h)$.
To do this it is sufficient to construct a function $\tau$ as in the proof above, setting $f:=r(\gamma)|B(\gamma)|$
instead of $f=|B(\gamma)|$.



\begin{proof}[Proof of Theorem~\ref{th:div-disint}]
We first prove that ``if'' part of the theorem.

\emph{Step 1.}
The distributional formulation of \eqref{divuB} reads as
\begin{equation}\label{divuB-weak}
\int_{\Omega} u B \cdot \nabla \phi \,d \Le^2 + \int_{\Omega} \phi \,d \mu = 0
\end{equation}
for any $\phi \in C_0^\infty(\Omega)$. Using mollifiers and passing to the limit in \eqref{divuB-weak} we can prove that \eqref{divuB-weak} holds also for any compactly supported $\phi \in \Lip(\Omega)$. Therefore we can consider the test functions $\phi$ of the form $\phi(x) = \pfi(x) \psi(H(x))$, where $\pfi\in C_0^1(\Omega)$ and $\psi\in C_0^\infty(\RR)$. Since $B \cdot \nabla H = 0$ a.e., for such test functions \eqref{divuB-weak} takes the form
\begin{equation}\label{divuB-weak1}
\int_{\Omega} u \psi(H) B \cdot \nabla \pfi \, d\Le^2 + \int_{\Omega} \psi(H) \pfi \,d \mu = 0.
\end{equation}
Now we disintegrate the measure $|B|\Le^2$ using the coarea formula (see Proposition~\ref{prp:coarea})
and the measure $\mu$ using Proposition~\ref{prp:disint}:
\begin{gather*}
|B|\Le^2 = \int_{\RR} \Ha^1\rest E_h \, dh , \\
\mu = \int_{\RR} \mu_h \, dh + \int_{\RR} \mu_h \, d\sigma(h),
\end{gather*}
where $E_h$ denotes the level set $H^{-1}(h)$. Then we can rewrite \eqref{divuB-weak1} as
\begin{equation}\label{eq:total-disint}
\int_\RR \rb{\int_{E_h} u \frac{B}{|B|} \cdot \nabla \pfi \, d\Ha^1 + \int_{E_h} \pfi \, d\mu_h} \psi(h) \,dh +
\int_\RR \rb{\int_{E_h} \pfi \, d\mu_h } \psi(h) \,d\sigma(h) = 0.
\end{equation}

\emph{Step 2.}
Since $\sigma \perp \Le^1$ and \eqref{eq:total-disint} holds for any $\psi\in C_0^\infty(\RR)$
we deduce that for $\Le^1$-a.e. $h\in \RR$
\begin{equation}\label{eq:level-set}
\int_{E_h} u \frac{B}{|B|} \cdot \nabla \pfi \,d \Ha^1
+ \int_{E_h} \pfi \,d \mu_h =0
\end{equation}
and for $\sigma$-a.e. $h\in \RR$
\begin{equation}\label{eq:level-set-sing}
\int_{E_h} \pfi \,d \mu_h  = 0.
\end{equation}
Let $N_\pfi\subset \RR$ denote the negligible set such that \eqref{eq:level-set} holds for all $h\in \RR\setminus N_\pfi$. Note that in general $N_\pfi$ depends on the choice of $\pfi\in C_0^1(\Omega)$. However we can repeat the above computations for a countable set of functions of the form $\pfi \in \{\chi_k \pfi_m\}_{k,m\in \NN}$ where $\{\pfi_m\}_{m\in \NN}$ is dense in $C^1(\overline{\Omega})$ and $\{\chi_k\}_{k\in \NN}\subset C_0^\infty(\Omega)$ are cutoff functions such that for any compact $K\subset \Omega$ there exists $k_K$ such that $\chi_k=1$ in a neighborhood of $K$ for all $k\ge k_K$. Then the set $N:= \cup_{m\in \NN} N_{\pfi_m}$ is still negligible, and \eqref{eq:level-set} holds for any $h\in \RR\setminus N$ for all $\pfi \in \cup_{m\in \NN} \pfi_m$. Passing to the limit in \eqref{eq:level-set} we conclude that for any $h\in \RR\setminus N$ the equation \eqref{eq:level-set} holds for all $\pfi\in C_0^1(\Omega)$.
The same argument applies to \eqref{eq:level-set-sing}.

\emph{Step 3.}
Now are going to argue as in \cite{ABC1} (see Lemma 3.8) and deduce from \eqref{eq:level-set} that for every nontrivial connected component $C\in \ff{Conn^*}(E_h)$ of the level set $E_h$ we have
\begin{equation}\label{eq:connected-component}
\int_C u \frac{B}{|B|} \cdot \nabla \pfi \, d\Ha^1
+ \int_C \pfi \, d\mu_h =0.
\end{equation}

Indeed, fix $\pfi \in C_0^1(\Omega)$ and let us prove that \eqref{eq:level-set} implies \eqref{eq:connected-component}. For any (Borel) subset $A\subset E_h$ let us denote
\begin{equation}\label{PhiDef}
\Phi_A(\psi):=\int_A (\psi\pfi) \, d\mu_h + \int_A u b \cdot \nabla (\psi\pfi) \frac{1}{|\rho b|} \, d\Ha^1 + \int_A u b \cdot \nabla (\psi\pfi) \, d\eta_h.
\end{equation}
We know that $\Phi_{E_h}(\psi)=0$ for any $\psi \in C^1(\RR^2)$, and we need to prove that $\Phi_C(1)=0$.

Let $G$ denote a compactly supported Lipschitz extension of $H$ to $\RR^2$, constructed as in the proof of Theorem~\ref{th:ABC2}. Let $D$ denote the connected component of the level set $F_h:=\setof{x\in \RR^2}{G(x)=h}$ such that $C$ is a connected component of $D \cap \Omega$ (see Lemma~\ref{lem:conn-comp-cap}). Since $G$ is Lipschitz and compactly supported, $F_h$ is compact (without loss of generality we assume that $h\ne 0$). 
Therefore we can use Lemma~\ref{lem:sep-by-test} to construct a sequence $(\pfi_n)_{n\in \NN} \subset C_0^\infty(\RR^2)$ such that $\pfi_n=1$ on $D$, $\pfi_n(x)\to 0$ as $n\to \infty$ for any $x\in F_h \setminus D$ and $F_h \cap \supp \nabla \pfi_n = \emptyset$ for any $n\in \NN$. By \eqref{eq:level-set} we have $\Phi_{F_h}(\pfi_n \psi) = 0$ for any $n \in \NN$, hence passing to the limit as $n \to \infty$ we obtain
\begin{equation*}
\Phi_D(\psi)=0.
\end{equation*}

\begin{figure}[h]
\label{fig:LocalLevelSets}
\centering
\includegraphics{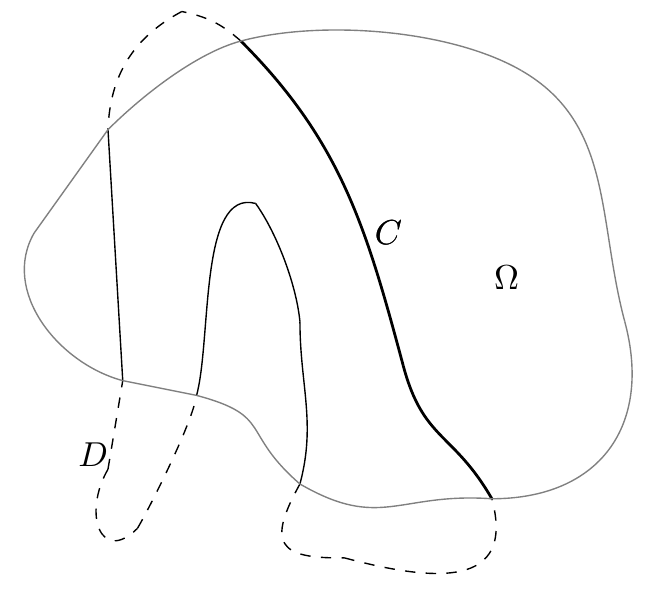}
\caption{Level sets of $H$ and $G$.}
\end{figure}

If $D = C$ then \eqref{eq:connected-component} holds (one can simply take $\psi \equiv 1$). Suppose that $D \ne C$. In this case repeating the argument from the proof of Theorem \ref{th:ABC2} one can show that
there exist a Lipschitz parametrization $\gamma \colon I_\gamma \to \RR^2$ of $D$
and an open interval $J \subset I_\gamma$ such that $J\ne I_\gamma$ and $C = \gamma(J)$. (Here $I_\gamma = \RR/(L\ZZ)$ for some $L>0$.)

Let $K:=\supp \pfi$.
Since $C$ is closed in the induced topology of $\Omega$ and $K\subset \Omega$ is compact, $C\cap K$ is compact. Since $\gamma$ is continuous and injective $\gamma(I_\gamma \setminus J)$ is closed.
Therefore, since $C\cap K$ and $\gamma(I_\gamma \setminus J)$ are disjoint compacts, we can find $\psi \in C_0^\infty(\RR^2)$ such that $\psi=1$ on $C \cap K$ and $\psi = 0$ on $\gamma(I_\gamma \setminus J)$.
Since $D\cap K = (\gamma(I_\gamma\setminus J) \cap K) \cup (C \cap K)$ we have
\begin{equation*}
\Phi_C(1) = \Phi_{C\cap K}(1) = \Phi_{D \cap K}(\psi) = \Phi_D(\psi) = 0
\end{equation*}
which concludes the proof of \eqref{eq:connected-component}.

\emph{Step 4.}
By Theorem~\ref{th:ABC2} for a.e. $h\in \RR$ we have $\Ha^1(E_h \setminus E_h^*) =0$, hence from \eqref{eq:level-set} and \eqref{eq:connected-component} we can deduce that $\int_{E_h\setminus E_h^*} \pfi \mu_h =0$. Since $\pfi$ is arbitrary this implies that $\mu_h \rest (E_h \setminus E_h^*) = 0$.

It remains to prove the ``only if'' part of the theorem.

Let us compute the left-hand side $F(\phi)$ of \eqref{divuB-weak} for a given test function $\phi\in C_0^\infty(\Omega)$. Repeating the computations from Step~1 we can see that $F(\phi)$ equals to the left-hand side of \eqref{eq:total-disint} with $\pfi\equiv \phi$ and $\psi \equiv 1$. Hence to prove that $F(\phi)=0$ it suffices to prove that the equality \eqref{eq:total-disint} holds.

In order to prove \eqref{eq:total-disint} it suffices to prove that \eqref{eq:level-set} holds for a.e. $h\in \RR$. (We can assume that $\sigma=0$ in view of our assumptions.)

By Theorem~\ref{th:ABC2} for a.e. $h\in \RR$ the set $\ff{Conn^*}(E_h)$ is countable and $\Ha^1(E_h \setminus E_h^*) =0$. By our assumptions moreover $\mu_h \rest (E_h \setminus E_h^*) = 0$ for a.e. $h\in \RR$. Therefore \eqref{eq:connected-component} implies \eqref{eq:level-set}.
\end{proof}

\subsection{Traces and chain rule for incompressible vector fields}

In this section we combine Theorems~\ref{th:div-disint} and \ref{th:eq-along-curve} and prove that
for a.e. $h\in \RR$ for any nontrivial connected component $C\in \ff{Conn^*}(E_h)$ of the level set $E_h:=H^{-1}(h)$
the solution $u$ of \eqref{divuB} has traces $u^\pm$ along $C$. 
Then we use these functions $u^\pm$ to solve the chain rule problem for the divergence operator.
Namely, we prove that $\div(\beta(u)B)$ is (represented by) a Radon measure $\nu$ such that $\nu \ll \mu$
and the density of $\nu$ with respect to $\mu$ is a function of $u^+$ and $u^-$.

Existence of such functions $u^\pm$ for every fixed $C$ is a rather straightforward implication of Theorems~\ref{th:div-disint} and \ref{th:eq-along-curve}.
However the technical difficulty here is to show that these traces can be seen as restrictions to $C$ of some Borel functions
defined in $\Omega$ (which, of course, are independent of $C$). These functions formally are defined as follows:

\begin{definition} \label{def:traces-global}
Suppose that $B\colon \Omega \to \RR$ is a bounded vector field which satisfies \eqref{B=nabla} for some Lipschitz function $H\colon \Omega \to \RR$ with bounded support. Suppose that $u\colon \Omega \to \RR$ is a Borel function.

We say that the Borel function $u^+ \colon \Omega \to \RR$ is the \emph{(positive) trace of $u$ along $B$} in $\Omega$ if for a.e. $h\in \RR$ for any $C\in \ff{Conn^*}(H^{-1}(h))$ there exists a trace $u^+_{C,B}$ of $u$ along $(C,B)$ and $u^+(x)=u^+_{C,B}(x)$ for any $x\in C$. The trace $u^-\colon \Omega \to \RR$ is defined analogously.
\end{definition}

\begin{remark}
In view of Theorem~\ref{th:ABC2} if $u^+$ and $\hat u^+$ are traces of $u$ along $B$ then $u^+ = \hat u^+$ for $|B|\Le^2$-a.e. $x\in \Omega$.
In other words, the traces $u^\pm$ (if they exist) are defined up to $|B|\Le^2$-negligible subset of $\Omega$. We will show that the traces exist when $\div(uB)=\mu$ for some Radon measure $\mu$, and in this case the traces will be defined up to $(|B|\Le^2+ |\mu|)$-negligible sets.
\end{remark}

\begin{remark}
It would be interesting to extend definition of traces to $\d \Omega$, and develop the theory of initial-boundary value problem, but this goes beyond the scope of the present paper.
\end{remark}

The following theorem is the main result of this section:

\begin{theorem} \label{th:traces-chain-rule} 
Suppose that $B\colon \Omega\to \RR^2$ is a bounded vector field which satisfies \eqref{B=nabla}
for some Lipschitz function $H\colon \Omega \to \RR$ with bounded support.
Suppose that Radon measures $\mu_i$ on $\Omega$ and functions $u_i\in L^1(|B|\Le^2)$ satisfy
\begin{equation} \label{divuB-sys}
\div(u_i B) = \mu_i \quad \text{in } \; \ss D'(\Omega)
\end{equation}
for $i=1,2,...,m$.
Then $u_i$ has traces $u_i^\pm \colon \Omega \to \RR$ along $B$ (see Definition~\ref{def:traces-global}).

Moreover, let $u=(u_1, \dots, u_m)$ and $\mu = (\mu_1, \dots, \mu_m)$.
Then for any bounded $\beta \in C^1(\RR^m)$ with bounded $\nabla \beta$
the function $\beta(u)$ has traces along $B$ which are given by
\begin{equation} \label{beta-traces-m}
\beta(u)^\pm = \beta(u^\pm)
\end{equation}
and $\beta(u)$ satisfies \eqref{divbeta} with
\begin{equation} \label{nu-char-div-free-m}
\nu = f(u^+,u^-) \cdot \mu
\end{equation}
where $f$ is given by \eqref{f-def}.
\end{theorem}

\begin{remark}
The functions $u^\pm$ can be interpreted as the traces of $u$ along the trajectories of $B$
(see  Remark~\ref{r:flow}).
\end{remark}

\begin{remark} \label{rem:traces-divuB}
In view of Theorem~\ref{th:flow} if $u$ solves \eqref{divuB} and $u^+$ and $\hat u^+$ are traces of $u$ along $B$ then
$u^+(x) = \hat u^+(x)$ for $(|B|\Le^2+ |\mu|)$-a.e. $x\in \Omega$.
In other words, the traces of solutions of \eqref{divuB} are defined up to $(|B|\Le^2+ |\mu|)$-negligible sets. Moreover, in case of system \eqref{divuB-sys} \emph{for each $i$} the traces $u_i^\pm$ are defined up to $(|B|\Le^2 + |\mu_1| + ... + |\mu_m|)$-negligible sets.
\end{remark}

\begin{proof}[Proof of Theorem~\ref{th:traces-chain-rule}]
\emph{Step 1.}
Using Theorem~\ref{th:div-disint} as a \emph{necessary condition} we can find a negligible set $N\subset \RR$ such that for \emph{any} $h\in \RR \setminus N$ the statements \ref{th:div-disint-1a} and \ref{th:div-disint-1b} of Theorem~\ref{th:div-disint} hold.

By Theorem~\ref{th:eq-along-curve} for any $h\in \RR \setminus N$ for any $C \in \ff{Conn^*}(E_h)$ there exist traces $u^\pm_C$ of $u$ along $(C,B)$ (see Definition~\ref{def:traces})
such that \eqref{divbetauB-on-C} holds.

\emph{Step 2.}
Since the connected components of $E_h^*$ are pairwise disjoint and $E_h \cap E_{h'} = \emptyset$ for $h\ne h'$ for any $x\in E:=\cup_{h\in \RR \setminus N} E_h^*$ there exist unique $h=h(x)$ and $C_x \in \ff{Conn^*}(E_h)$ such that $x\in C_x$. Therefore we can define
\begin{equation} \label{global-traces-def}
u^\pm(x) := u^\pm_{C_x}(x).
\end{equation}

We claim that there exists a $\Le^1$-negligible set $\hat N\subset \RR$ such that
on $E \setminus H^{-1}(\hat N)$ the functions $u^\pm$ defined by \eqref{global-traces-def} agree with some Borel function $\hat u^\pm$. For convenience of the reader we present the proof of this technical fact in the Appendix.

In view of the point \ref{th:div-disint-2} of Theorem~\ref{th:div-disint} we have
$|\mu|(H^{-1}(\hat N)) = 0$. On the other hand by coarea formula $|B| \Le^2(H^{-1}(\hat N)) = 0$. Hence the equality $u^\pm = \hat u^\pm$ holds $(|B|\Le^2 + |\mu|)$-a.e. and without loss of generality we can assume that $u^\pm$ are Borel.

\emph{Step 3.}  Since $\nabla \beta$ is continuous and bounded and the functions $u^\pm$ are Borel we have $f(u^+, u^-) \in \Linf (|\mu|)$ and hence $\nu$ defined by \eqref{nu-char-div-free-m} is finite Radon measure. Then we can apply Theorems~\ref{th:div-disint} and \ref{th:eq-along-curve} as \emph{sufficient conditions} to conclude that $\div (\beta(u) B) = \nu$.
\end{proof}

%% file: ChainRuleNI.tex
\section{Chain rule for steady nearly incompressible vector fields}\label{sec:chain-rule}

In this section we solve the chain rule problem for nearly incompressible vector fields
by reducing it to the case of incompressible ones, which we considered in the previous section.
Our main result is the following:

\begin{theorem} \label{th:chain-rule-ni}
Let $\Omega \subset \RR^2$ be a simply connected open set.
Suppose that $B \in \Linf(\Omega; \RR^2)$ has bounded support
and $u\in \Linf(\Omega)$.
Suppose that there exists $\rho \in \Linf(\Omega)$ such that $\rho \ne 0$ a.e. in $\Omega$ and $\rho$ satisfies \eqref{divrhob}.
Suppose that $\lambda$ and $\mu$ are Radon measures on $\Omega$ and \eqref{divB} and \eqref{divuB} hold.
Then there exist bounded functions $u^\pm \colon \Omega \to \RR$ such that for any $\beta\in C^1(\RR)$
the Radon measure $\nu$ given by \eqref{nu-char} satisfies \eqref{divbeta}.
\end{theorem}
\begin{remark}
By \eqref{f0} and \eqref{f1} the measure $\nu$ is always absolutely continuous with respect to $|\lambda|+|\mu|$.
\end{remark}
\begin{remark}
When $u$ is not bounded but belongs to $L^1(\Omega)$ the theorem still holds if $\beta$ is bounded and has compactly supported derivative.
\end{remark}
\begin{proof}
Let us denote $A:=\rho B$. Clearly we have $\div A = 0$.

Since $A^\perp$ is curl-free and $\Omega$ is simply connected there exists\footnote{This can be proved by reduction to the smooth case using mollified vector fields and passage to the limit using Arzela-Ascoli theorem} a function $H \in \Lip(\Omega)$ such that $A^\perp = - \nabla H$, or, equivalently, $A=\nabla^\perp H$. Such $H$ is unique up to an additive constant, which is uniquely determined by the condition $\int H \, dx = 0$.

Let us introduce $r:= 1/\rho$ and $v:=u/\rho$. Observe that $r,v \in L^1(|A| \Le^2)$ and moreover
\begin{gather*}
\div (r A) = \lambda, \\
\div (v A) = \mu.
\end{gather*}

Let $r^\pm$ and $v^\pm$ denote the traces of $r$ and $v$ given by Theorem~\ref{th:traces-chain-rule}.
Let $u^\pm := v^\pm / r^\pm$. Since $u$ is bounded and $v=\rho u$ the functions $u^\pm$ also are bounded (here we use Theorem~\ref{th:ABC2} because $u$ can be unbounded on some $\Le^2$-negligible set).

Since $\beta \in C^1$ and $u$ is bounded, we have
\begin{gather*}
\abs{\frac{\beta(u^+)-\beta(u^-)}{u^+ - u^-}} \le C,\\
\abs{\frac{u^+ \beta(u^-)  -  u^- \beta(u^+)}{u^+ - u^-}}=
\abs{\beta(u^-) + u^-\frac{\beta(u^-)  - \beta(u^+)}{u^+ - u^-}}
\le C
\end{gather*}
where $C>0$ depends only on $\|u\|_\infty$ and $\|\beta\|_{C^1([-\|u\|_\infty,\|u\|_\infty])}$.
Similarly, $|\beta(u^\pm) - u^\pm \beta'(u^\pm)| \le C$ and $|\beta'(u^\pm)| \le C$. Hence the functions \eqref{f0} and \eqref{f1} are bounded, therefore $\nu$ defined by \eqref{nu-char} is absolutely continuous with respect to $|\lambda|+|\mu|$.

Now when $\nu$ is constructed it remains to show that $\nu$ satisfies \eqref{divbeta},
which is equivalent to
\begin{equation*}
\div(r \beta(u) A) = \nu.
\end{equation*}
In view of Theorems~\ref{th:div-disint} and \ref{th:eq-along-curve}
it is sufficient to prove that for a.e. $h\in \RR$ for any $C \in \ff{Conn^*}(H^{-1}(h))$ we have
\begin{equation} \label{rbetau}
(r \beta(u))' = \nu
\end{equation}
in $\ss D'(C,\gamma)$
for any Lipschitz parametrization $\gamma$ of the curve $C$.

On the other hand by Theorems~\ref{th:div-disint} and \ref{th:eq-along-curve} there exists a negligible set $N \subset \RR$ such that for any $h\in \RR\setminus N$ and any $C \in \ff{Conn^*}(H^{-1}(h))$ we have
\begin{equation} \label{rlambdavmu}
r' = \lambda, \quad (ru)' = \mu
\end{equation}
in $\ss D'(C,\gamma)$. Therefore it is sufficient to prove that \eqref{rlambdavmu} implies \eqref{rbetau}
for any fixed $C \in \ff{Conn^*}(H^{-1}(h))$ for any $h \in \RR \setminus N$.

Let us fix some Lipschitz parametrization $\gamma$ of $C$.
For brevity we will
denote $u \circ \gamma$, $u^\pm \circ \gamma$, $\rho \circ \gamma$, $r\circ \gamma$, $\gamma^{-1}_\# \mu \rest C$, $\gamma^{-1}_\# \lambda \rest C$ and $\gamma^{-1}_\#\nu \rest C$ simply by $u$, $u^\pm$, $\rho$, $r$, $\mu$, $\lambda$ and $\nu$ respectively. Then \eqref{rbetau} and \eqref{rlambdavmu} hold in $\ss D'(I)$, where $I$ is the domain of $\gamma$.

In view of \eqref{rlambdavmu} the functions $r$ and $ru$ belong to $BV(I)$.
By our assumptions $|\rho|$ is bounded from above, hence $|r|$ is bounded from below (a.e. on $I$).
Then by Vol'pert's chain rule (see e.g. \cite{AFP}, Theorem~3.96) we have $\rho \in BV(I)$.
Consequently (again by Vol'pert's chain rule) $u = (ru) \cdot \rho$ also belongs to $BV(I)$.

Now we expand the left-hand sides of \eqref{rbetau} and \eqref{rlambdavmu} using Vol'pert's chain rule.

We first consider the \emph{diffuse part}:
\begin{gather*}
\lambda^d = (r')^d,
\\
\mu^d = \tilde u (r')^d + \tilde r (u')^d,
\\
[(r\beta(u))']^d = \beta(\tilde u) (r')^d + \tilde r \beta'(\tilde u) (u')^d.
\end{gather*}
(Here $\tilde u(t)$ and $\tilde r(t)$ denote the limits of $u$ and $r$ at continuity point $t\in I$.)
Subtracting from the third equation the second one multiplied by $\beta'(\tilde u)$ and the first one multiplied by $(\beta(\tilde u)-\tilde u\beta'(\tilde u))$ we obtain
that $[(r\beta(u))']^d = (\beta(\tilde u)-\tilde u\beta'(\tilde u)) \lambda^d + \beta'(\tilde u) \mu^d$.
On the other hand by definition of $\nu$ we have $\nu^d = (\beta(\tilde u)-\tilde u\beta'(\tilde u)) \lambda^d + \beta'(\tilde u) \mu^d$. Therefore 
\begin{equation} \label{nu-diffuse}
\nu^d = [(r\beta(u))']^d.
\end{equation}

Now let us consider the \emph{jump part}:
\begin{gather*}
\lambda^j = \rb{r^+ - r^-} \Ha^0,
\\
\mu^j = \rb{r^+ u^+ - r^- u^-} \Ha^0,
\\
[(r\beta(u))']^j = \rb{r^+ \beta(u^+) - r^- \beta(u^-)} \Ha^0.
\end{gather*}
Let $J_u$ denote the jump set of $u$. From the equations above we have
\begin{equation*}
(u^+ - u^-) [(r\beta(u))']^j = (u^+ \beta(u^-) - u^- \beta(u^+)) \lambda^j + (\beta(u^+)-\beta(u^-)) \mu^j.
\end{equation*}
On the other hand by definition of $\nu$
\begin{equation*}
\nu^j \rest J_u = \frac{u^+ \beta(u^-) - u^- \beta(u^+)}{u^+ - u^-} \lambda^j \rest J_u + \frac{\beta(u^+)-\beta(u^-)}{u^+ - u^-} \mu^j \rest J_u.
\end{equation*}
Therefore $\nu^j \rest J_u = [(r\beta(u))']^j \rest J_u$.
It remains to check that $\nu^j \rest (J_u)^c = [(r\beta(u))']^j \rest (J_u)^c$,
where $(J_u)^c = I \setminus J_u$.

Observe that
\begin{equation} \label{strage-cancellation}
\tilde u \lambda^j \rest (J_u)^c  = \mu^j \rest (J_u)^c 
\end{equation}
therefore
\begin{equation*}
[(r \beta(u))']^j \rest (J_u)^c =\beta(\tilde u) \lambda^j \rest (J_u)^c  =
(\beta(\tilde u) - \tilde u \beta'(\tilde u)) \lambda^j\rest (J_u)^c 
 + \beta'(\tilde u) \mu^j \rest (J_u)^c 
\end{equation*}
hence $\nu^j \rest (J_u)^c = [(r\beta(u))']^j \rest (J_u)^c$ and
we conclude that
\begin{equation} \label{nu-jump}
\nu^j = [(r\beta(u))']^j.
\end{equation}
By \eqref{nu-diffuse} and \eqref{nu-jump} we have $\nu = (r \beta(u))'$ which completes the proof.
\end{proof}

\begin{remark}
Assumptions of Theorem~\ref{th:chain-rule-ni} allow $\rho$ to take both positive and negative values.

In order to demonstrate what happens in this case
for simplicity let us consider a vector field $B\colon \RR \to \RR$ given by $B(x):=\sign(x)$
and a function $u(x):=c$, $x\in \RR$, where $c\in \RR$ is a constant.
Then $\mu:=\div(u B) = (uB)' = 2c \delta$ and $\lambda:= \div B = 2 \delta$, where $\delta$ is the Dirac delta.
Clearly the function $\rho(x):=\sign(x)$ satisfies $\div(\rho B)=0$ in $\ss D'(\RR)$.
Since evidently $u^+=u^- =c$ by \eqref{f0} and \eqref{f1} we have $f_1(u^+,u^-)= \beta(c) - c \beta'(c)$,
$f_0(u^+,u^-) = \beta'(c)$ and therefore \eqref{nu-char} reads as
\begin{align*}
\nu =&~ (\beta(c) - c \beta'(c)) \lambda + \beta'(c) \mu \\
=&~ (\beta(c) - c \beta'(c)) 2 \delta + \beta'(c) 2 c \delta = 2 \beta(c) \delta
\end{align*}
so indeed $\div(\beta(u) B) = \nu$, as Theorem~\ref{th:chain-rule-ni} would predict.
Clearly a similar phenomenon can occur in dimension 2.
\end{remark}

\begin{remark}
Suppose that $\lambda$, $\mu$ and $\nu$ satisfy \eqref{divB}, \eqref{divuB} and \eqref{divbeta}.
Since the measures $\mu$ and $\nu$ in general are not mutually singular, the functions $g_0$ and $g_1$ such that $\nu = g_0 \lambda + g_1 \mu$ are not defined in a unique way.

For instance, in the example from the previous remark we could have written both
$\nu = \beta(c) \lambda$ and $\nu = \frac{\beta(c)}{c} \mu$ provided that $c \ne 0$.
This shows that non-uniqueness of $g_0$ and $g_1$ is related to the cancellation property \eqref{strage-cancellation}.

We would like also to demonstrate that another such pair $(g_0,g_1)$ could be constructed directly using Theorem~\ref{th:traces-chain-rule}.
Let us denote $w:=(r,v)$.
By Theorem~\ref{th:traces-chain-rule} there exist functions $w^\pm =(r^\pm,v^\pm)$ such that for any bounded $\alpha\in C^1(\RR^2)$ with bounded derivatives the measure
\begin{equation*}
\hat \nu = g_0(w^\pm) \lambda + g_1(w^\pm) \mu
\end{equation*}
satisfies
\begin{equation*}
\div(\alpha(w) A) = \hat \nu,
\end{equation*}
where
\begin{gather*}
g_0(w^+,w^-) := \int_0^1 \d_r \alpha (t w^+ + (1-t) w^-) \, dt, \\
g_1(w^+,w^-) := \int_0^1 \d_v \alpha (t w^+ + (1-t) w^-) \, dt.
\end{gather*}

We would like to take $\alpha(r,v):=r \beta(v/r)$, because then $\alpha(1/\rho, u/\rho) = \beta(u)$ and hence
\begin{equation*}
\nu = \div(\beta(u) B) = \div (r \beta(u) \rho B) = \div (r \beta(u) A) = \div (r \beta(v/r) A) = \div (\alpha(r,v)A) =\hat \nu.
\end{equation*}
Note that $\alpha$ is not $C^1$, but this difficulty can be overcome using appropriate approximations of $\beta$.

However one can notice that
\begin{equation*}
g_1(w^+,w^-) = \int_0^1 \beta'\rb{\frac{ r^+ u^+ t + r^- u^- (1-t) }{ r^+t + r^- (1-t)}} \, dt
\end{equation*}
is (in general) different from
\begin{equation*}
f_1(u^+,u^-) = \frac{\beta(u^+) - \beta(u^-)}{u^+ - u^-}.
\end{equation*}
Indeed, when $u^+ = r^+ = 2$ and $u^- = r^- = 1$ we have $g_1(w^+, w^-) = \beta'(2)$
while $f_1(u^+, u^-) = \beta(2) - \beta(1)$.
\end{remark}

\subsection{Analysis of the discontinuity set revisited}

In this section we turn back to the vector field $B$ constructed in Theorem~\ref{th:counterexample} and study it from the viewpoint of Theorem~\ref{th:chain-rule-ni}. 
We discuss the chain rule problem for this $B$ and characterize the error term $\sigma$ in \eqref{chain-rule-c}.

Let $\Omega$, $\rho$ and $B$ be as in Theorem~\ref{th:counterexample}.
We define $u:=\rho$.
Clearly this function $u$ solves \eqref{divuB} with $\mu =0$.
Since $B\in BV(\Omega)$ we can define the measure $\lambda$ by \eqref{divB}.
Moreover, by Lemma~\ref{lem:B} we have $\lambda^a = \lambda^j = 0$
and $\lambda^c\ne 0$ is concentrated on the set $S$ defined by \eqref{def:S}.

In view of the results of \cite{AdLM} (see Theorems~3, 4 and 7) for any $\beta\in C^1(\RR)$ there exist Radon measure $\nu$ such that \eqref{divbeta} holds and $\nu$ satisfies \eqref{chain-rule-a}, \eqref{chain-rule-j} and \eqref{chain-rule-c}.
Since $\mu=0$, $\lambda^a=\lambda^j=0$ and $\lambda^c$ is concentrated on $S$,
the equations \eqref{chain-rule-a}, \eqref{chain-rule-j} and \eqref{chain-rule-c} take the form
\begin{equation*}
\nu^a = \nu^j = 0, \quad
\nu^c = \sigma\rest S_u + (\beta(\tilde u) - \tilde u \beta'(\tilde u)) \lambda^c \rest (S \setminus S_u).
\end{equation*}
(Recall that $\tilde u(x)$ denotes the approximate $L^1$ limit of $u$ at $x$.)
By Theorem~\ref{th:counterexample} the set $S_u$ coincides with $S$ up to a $|\lambda^c|$-negligible set, hence
\begin{equation*}
\nu = \nu^c = \sigma.
\end{equation*}

In view of Lemma~\ref{lem:flow} (see also Section~\ref{sec:Analysis-of-S}) and Theorem~\ref{th:chain-rule-ni} the functions
\begin{equation*}
\begin{gathered}
u^+(x) := 
\begin{cases}
1, & x \in \rho^{-1}(1); \\
2, & x \in S \cup \rho^{-1}(2),
\end{cases} \\
u^-(x) := 
\begin{cases}
1, & x \in S \cup \rho^{-1}(1); \\
2, & x \in \rho^{-1}(2)
\end{cases}
\end{gathered}
\end{equation*}
are the traces of $u$ along $B$ (recall Definition~\ref{def:traces-global}).
Then by \eqref{f0}
\begin{equation*}
f_0(u^+(x), u^-(x)) = 
\begin{cases}
\beta'(u(x)), & x \notin S; \\
2\beta(1) - \beta(2), & x \in S
\end{cases}
\end{equation*}
and therefore by Theorem~\ref{th:chain-rule-ni} we have
\begin{equation} \label{sigma-char}
\nu = (2\beta(1)-\beta(2)) \lambda.
\end{equation}
Since $\nu=\sigma$ the formula \eqref{sigma-char} completely characterizes the term $\sigma$.

Finally, using Theorem~\ref{th:counterexample} and results of \cite{AdLM} it is possible
to provide an alternative proof of the last claim of Theorem~\ref{th:counterexample}:
\begin{proposition}
The set $S_u$ coincides with $S$ up to a $|\lambda^c|$-negligible set.
\end{proposition}
\begin{proof}
Since $S_u \subset S$ we need to prove that $|\lambda|(S \setminus S_u) = 0$ (recall that $\lambda = \lambda^c$).
If $x\in S$ is a point of $L^1$ approximate continuity of $u$ then either $\tilde u(x) = 1$ or $\tilde u(x) = 2$,
because $u$ takes only values $1$ and $2$ (see page~\pageref{app-cont-points} for the details).
Let $S_i$ denote the set of points $x\in S$ such that $\exists \tilde u(x) = i$, $i=1,2$.
By \eqref{chain-rule-c} (see Theorem~7 in \cite{AdLM}) we have
\begin{equation*}
\nu \rest S_1 = (\beta(1) - \beta'(1)) \lambda \rest S_1.
\end{equation*}
On the other hand applying Theorem~\ref{th:chain-rule-ni} we obtain \eqref{sigma-char} which implies that
\begin{equation*}
\nu \rest S_1 = (2\beta(1)-\beta(2)) \lambda \rest S_1.
\end{equation*}
Taking for instance $\beta(u)=u^2$ we obtain a contradiction unless $|\lambda|(S_1)=0$.
The same way one proves that $|\lambda|(S_2)=0$.
\end{proof}

%% file: RenormalizationNI.tex
\section{Renormalization property of steady nearly incompressible vector fields}

In this section we consider transport equation \eqref{transport} with steady nearly incompressible vector fields.
Though the chain rule problem for such vector fields can be solved using Theorem~\ref{th:chain-rule-ni},
the near incompressibility assumption is still to weak to guarantee uniqueness of weak solutions of \eqref{transport}
(recall for instance the well-known DePauw's counterexample \cite{Depauw}).
Therefore one has to require some extra regularity of $B$.

We show that if $B$ has, in addition, bounded variation, then uniqueness of weak solutions of \eqref{transport} holds.
Our results rely on the framework developed in \cite{ABC1}, according to which uniqueness of weak solutions of \eqref{transport}
with vector field of the form $A= \nabla^\perp H$ holds if and only if the function $H$ has so-called weak Sard property.
In this section we extend this framework to nearly incompressible vector fields $B$, for which we introduce the function $H$
by $\rho B = \nabla^\perp H$. We prove that if $B$ is nearly incompressible and has bounded variation then $H$ has the weak Sard property and, consequently, weak solutions of \eqref{transport} with this $B$ are unique.

Let $\Omega\subset \RR^2$ be an open set.

\begin{definition}
We say that $H\in \Lip(\Omega)$ has the \emph{weak Sard property} if
\[
H_\# \Le^2 \rest \cb{\nabla H = 0} \perp \Le^1.
\]
\end{definition}

(Note that the property defined above is slightly stronger than the original weak Sard property introduced in \cite{ABC2}.)

When $A = \nabla^\perp H$ and $H$ has the weak Sard property, transport equation \eqref{transport} is
equivalent to a family of one-dimensional transport equations along the level sets of $H$:

\begin{theorem} \label{th:transport-with-wsp}
Suppose that $H \colon \Omega \to \RR$ is a Lipschitz function with compact support and $A\colon \Omega \to \RR^2$ is a bounded vector field such that $A = \nabla^\perp H$ a.e. in $\Omega$.
Suppose that $\eta \colon \Omega \to \RR$ is a bounded function and $T>0$. If $H$ has the weak Sard property then a bounded function $u\colon (0,T) \times \Omega \to \RR^2$ solves
\begin{equation} \label{transport-generalized}
\d_t (\eta u) + \div (u A) = 0
\end{equation}
in $\ss D'((0,T)\times \Omega)$ if and only for a.e. $x\in \{A=0\}$ we have $\d_t u(t,x)=0$ in $\ss D'((0,T))$ and for a.e. $h\in \RR$ for any $C\in \ff{Conn^*}(E_h)$ (which is closed simple curve) we have
\begin{equation} \label{transport-generalized-along-curve}
\d_t (\eta_\gamma u_\gamma ) + \d_s u_\gamma = 0
\end{equation}
in $\ss D'((0,T) \times I_\gamma)$,
where $\gamma \colon I_\gamma \to \Omega$ is a Lipschitz parametrization of $C$ which satisfies $\gamma'=A\circ \gamma$ a.e. on $I_\gamma$, $\eta_\gamma := \eta \circ \gamma$ and $u_\gamma(t,s):=u(t,\gamma(s))$.
\end{theorem}

\begin{proof}
Since for $\eta \equiv 1$ and $\Omega = \RR^2$ this result was proved in \cite{ABC1} we only sketch the proof in order to show the connection with the chain rule problem.
Suppose that $\Psi \subset C_0^1((0,T))$ is a countable dense set and let us fix $\psi \in \Psi$.
Let $v\colon \Omega \to \RR$ and $w\colon \Omega \to \RR$ denote Borel functions such that
$v(x) = \int_\RR u(\tau, x) \psi(\tau) \, d \tau$
and $w(x) = \int_\RR u(\tau, x) \psi'(\tau) \, d \tau$ for a.e. $x\in \Omega$.
Clearly \eqref{transport-generalized} holds if and only if
\begin{equation} \label{transport-time-int}
- \eta w + \div (v A) = 0
\end{equation}
in $\ss D'(\Omega)$ for any $\psi \in \Psi$.

Let us decompose $\mu:=\eta w \Le^2$ as $\mu = \mu_1 + \mu_2$, where
\begin{equation*}
\mu_1:= \frac{\eta}{|A|} w |A| \Le^2 \rest \{\nabla H \ne 0\},
\quad
\mu_2:=\eta w \Le^2 \rest \{\nabla H =0\}.
\end{equation*}
We can disintegrate with respect to $H$ the first term using the coarea formula:
\begin{equation*}
\frac{\eta}{|A|} w |A| \Le^2 \rest \{\nabla H \ne 0\} = \int_\RR \frac{\eta}{|A|} w \, \Ha^1 \rest H^{-1}(h) \, dh
\end{equation*}
i.e.
\begin{equation*}
H_\# \mu_1 \ll \Le^1.
\end{equation*}
On the other hand by weak Sard property
\begin{equation*}
H_\# \mu_2 \perp \Le^1.
\end{equation*}
Therefore by Theorem~\ref{th:div-disint} equation \eqref{transport-time-int} holds if and only if $\mu_2=0$ and
\begin{equation} \label{stuff-along-curve}
\div \rb{ v \frac{A}{|A|} \Ha^1 \rest C } =  \frac{\eta}{|A|} w \Ha^1 \rest C
\end{equation}
holds for any $C\in \ff{Conn^*}(E_h)$ (which is closed simple curve) for a.e. $h\in \RR$.
(Here we use the fact that $\Psi$ is countable.)
Note that by coarea formula $A(x)\ne 0$ for $\Ha^1$-a.e. $x\in C$ and $1/|A| \in \Le^1(\Ha^1 \rest C)$ (for a.e. $h\in \RR$).
By Theorem~\ref{th:eq-along-curve} \eqref{stuff-along-curve} holds if and only if for any parametrization $\gamma \colon I_\gamma \to \Omega$ of $C$ which satisfies $\gamma'=A\circ \gamma$ a.e. on $I_\gamma$ we have
\begin{equation*}
\d_s (v\circ \gamma) = (\eta \circ \gamma) (w\circ \gamma)
\end{equation*}
 in $\ss D'(I_\gamma)$.
 Here we used that by Proposition~\ref{prp:param} we have $\Ha^1 \rest C = \gamma_\# (|\gamma'| \Le)$ and on the other hand $|\gamma'|= |A\circ \gamma|$ a.e. on $I_\gamma$.
\end{proof}

It turns out that one-dimensional equation \eqref{transport-generalized-along-curve} has the renormalization property:

\begin{lemma} \label{lem:renormalization-1D}
Suppose that $I$ is an open interval or a circle, $T>0$ and $\eta \in \Linf(I)$. If $u\colon (0,T) \times I\to \RR$ solves
\begin{equation*}
\d_t (\eta u) + \d_s u = 0
\end{equation*}
in $\ss D'((0,T)\times I)$ then for any $\beta\in C^1(\RR)$ we have
\begin{equation*}
\d_t (\eta \beta(u)) + \d_s \beta(u) = 0.
\end{equation*}
\end{lemma}

\begin{proof}
Let $u^\eps(\cdot, x) := \omega_\eps * u(\cdot ,x)$ denote the mollification of $u$ with respect to time $t$. Then
\begin{equation*}
\d_t (\eta u)^\eps + \d_s u^\eps = 0
\end{equation*}
in $\ss D'((\eps,T-\eps)\times I)$. Since $\eta$ does not depend on time we can write
\begin{equation*}
\eta \d_t u^\eps + \d_s u^\eps = 0
\end{equation*}
and therefore for a.e. $t\in (\eps, T-\eps)$ the function $s\mapsto u^\eps(t,s)$ is Lipschitz, since $\eta$ is bounded.
Therefore
\begin{equation*}
\d_s \beta(u^\eps) = \beta'(u^\eps) \d_s u^\eps = - \eta \beta'(u^\eps) \d_t u^\eps = - \eta \d_t \beta(u^\eps) = -\d_t(\eta \beta(u^\eps))
\end{equation*}
and it remains to pass to the limit as $\eps \to 0$.
\end{proof}

The core result of this section is the following:

\begin{theorem} \label{th:WSP-ni}
Suppose that $\Omega$ is simply connected and $B\colon \Omega \to \RR^2$ is a bounded BV vector field with bounded support. 
Suppose that $\rho \in \Linf(\Omega)$ satisfies \eqref{divrhob} in $\ss D'(\Omega)$ and $\rho \ne 0$ a.e. in $\Omega$.
Then there exists a Lipschitz function $H \colon \Omega \to \RR$ with bounded support such that \eqref{rhob=nablaperpH} holds a.e. in $\Omega$
and $H$ has the weak Sard property.
\end{theorem}

Before proving this theorem let us recall some auxiliary facts.
Given a point $x\in \Omega$ a number $k\in \NN \cup \{0\}$ and a function $f\colon \Omega \to \RR$ we will write $f\in t^{k,1}(x)$
if there exists a polynomial $p_x$ on $\RR^2$ of order less or equal than $k$ such that
\begin{equation*}
\frac{1}{r^2} \int_{B_r(x)} |f(y) - p_x(y)| \, dy = o(r^k)
\end{equation*}
as $r\to 0$. (If such polynomial exists, clearly it is unique.)
We also denote the polynomial $p_x$ as $P^k_x[f]$ to indicate the function $f$ which is approximated by $p_x$.

Given a subset $E\subset \Omega$ we write $f\in t^{k,1}(E)$ if for any $x \in E$ we have $f\in t^{k,1}(x)$.

\begin{lemma} \label{lem:polynomials}
Let $x\in \Omega$. Suppose that $f\in t^{1,1}(x) \cap W^{1,1}(\Omega)$ and $\nabla f \in t^{1,1}(x)$. Then $f\in t^{2,1}(x)$ and
$\nabla_y P^2_x[f](y) = P^1_x[f](y)$ for any $y\in \RR^2$.
\end{lemma}

\begin{proof}
Since $f\in t^{1,1}(x)$, the point $x$ is a point of $L^1$ approximate continuity of $f$.
Hence there exists $L^1$ approximate limit $\tilde f(x)$ of $f$ at $x$.
Since $\nabla f \in t^{1,1}(x)$ there exists a first order polynomial $p_x=p_x(y)$ such that
\begin{equation*}
\alpha(r):= \frac{1}{r^2} \int_{B_r(x)} |\nabla f(y) - p_x(y)| \, dy = o(r).
\end{equation*}
Let us define the second order polynomial $q_x=q_x(y)$ by
\begin{equation*}
q_x(y):= \tilde f(x) + \int_0^1\ab{p_x(t y + (1-t)x), y-x} \, dt.
\end{equation*}
For a.e. $y\in B_r(x)$ we have
\begin{equation*}
f(y) - \tilde f(x) = \int_0^1 \ab{(\nabla f)(t y + (1-t) x), y-x} \, dt
\end{equation*}
(this can be proved using standard mollifiers and multiplication by test functions). Therefore
\begin{align*}
\beta(r):=&~
\frac{1}{r^2} \int_{B_r(x)} |f(y) - q_x(y)| \, dy\\
\le&~ \frac{1}{r^2} \int_{B_r(x)} \int_0^1 |(\nabla f(t y + (1-t)x)) - p_x(t y + (1-t) x)| \cdot |y-x| \, dt \, dy \\
\le&~ r \int_0^1 \frac{1}{r^2} \int_{B_{rt}(x)} |\nabla f(z) - p_x(z)| \frac{1}{t^2} \, dz \, dt = r \int_0^1 \alpha(rt) \, dt.
\end{align*}
Since for any $\eps>0$ there exists $\delta>0$ such that $0\le \alpha(r)<\eps r$ for all $r\in (0,\delta)$, we also have
\begin{equation*}
\int_0^1 \alpha (rt) \, dt < \eps \int_0^1 rt \, dt = \frac{\eps}{2} r
\end{equation*}
so we conclude that $\beta(r) = o(r^2)$.
\end{proof}

\begin{proof}[Proof of Theorem~\ref{th:WSP-ni}]
In the proof of Theorem~\ref{th:chain-rule-ni} we have already demonstrated existence of a Lipschitz function for which \eqref{rhob=nablaperpH} holds. Therefore we only need to prove that $H$ has the weak Sard property.

Since $B\in BV$, by Calderon-Zygmund theorem (see \cite{AFP}, Theorem 3.83) the function $B$ is approximately differentiable in a.e. $x\in \Omega$:
\begin{equation}\label{eq:app-diff}
\frac{1}{r^2} \int_{B_r(x)} \abs{B(y) - \wave{B}(x) - M(x)\cdot(y-x)} \, dy = o(r).
\end{equation}
Let $N:= \setof{x\in \Omega}{\nabla H(x) = 0}$. 
Since $\rho \ne 0$ a.e. the set $N$ coincides, up to a $\Le^2$-negligible set, with the subset $\setof{x\in \Omega}{B(x)=0}$.
Hence $\tilde B(x)=0$ for a.e. $x\in N$.
Then due to locality of approximate derivative (see \cite{AFP}, Proposition 3.73) we have $M(x)=0$ for a.e. $x\in N$.
Consequently \eqref{eq:app-diff} for a.e. $x\in N$ takes the form
\begin{equation*}
\frac{1}{r^2} \int_{B_r(x)} |B(y)| \, dy = o(r).
\end{equation*}
Since $|\rho| \le C$ for some constant $C>0$ a.e. in $\Omega$, the equation above implies that
\begin{equation*}
\frac{1}{r^2} \int_{B_r(x)} |\rho(y) B(y)| \, dy = o(r)
\end{equation*}
hence $\rho B \in t^{1,1}(x)$ for a.e. $x\in N$ and $P^1_x[\rho B] \equiv 0$.

But from Rademacher's theorem $H\in t^{1,1}(x)$ for a.e. $x\in N$. Then by Lemma~\ref{lem:polynomials} 
\begin{equation} \label{eq:2nd-oe}
H\in t^{2,1}(x) \quad
\text{and}
\quad
P_x^2[H](y) = P_x^2[H](0) \quad
\forall y\in \RR^2
\end{equation}
for a.e. $x\in N$.

For any $\eps>0$ there exists a compact $K\subset N$ such that $\Le^2(N \setminus K) < \eps$ and \eqref{eq:2nd-oe} holds for every $x\in K$.
Then by $L^1$ version of Whitney's extension theorem (see \cite{Zmr}, Proposition 3.6.3)
there exist an open set $U \subset \Omega$ and a function $\tilde H\in C^2(U)$
such that $K \subset U$ and $H=\tilde H$ on $K$. Hence $H_\# \Le^2 \rest K = \tilde H_\# \Le^2 \perp \Le^1$ by classical Sard's theorem.
Since $\eps>0$ is arbitrary, the proof is complete.
\end{proof}

The following result immediately follows from Theorem~\ref{th:WSP-ni}, Theorem~\ref{th:transport-with-wsp} and Lemma~\ref{lem:renormalization-1D}:
\begin{theorem}
Suppose that $\Omega\subset \RR^2$ is a simply connected open set and $B\colon \Omega \to \RR^2$ is a steady nearly incompressible $BV$ vector field with bounded support.
Then $B$ has renormalization property.
\end{theorem}

\begin{remark}
In fact from the proofs presented above it follows that $B$ has the renormalization property under the following weaker assumptions:
$B\in \Linf(\Omega)$ is approximately differentiable a.e. in $\{B=0\}$
and there exists $\rho \in \Linf(\Omega)$ (not necessarily non-negative)
such that $\rho\ne 0$ a.e. in $\Omega$ and $\div(\rho B)=0$.
The approximate differentiability of $B$ holds, for instance, when $\div B$ and $\curl B$ are Radon measures \cite{ABC3}.
\end{remark}

%% file: Acknowledgements.tex
\section{Acknowledgements}
The work was supported by ERC Starting Grant ConsLaw 2009-2013 and PRIN projects ``Systems of Conservation Laws and Fluid Dynamics: Methods and Applications'' (2011-2013) and ``Nonlinear Hyperbolic Partial Differential Equations, Dispersive And Transport Equations: Theoretical And Applicative Aspects'' (2013-2016).
The first author is also grateful to ENS Paris.
The second author was partially supported by grant RFBR 13-01-12460 OFI-m and the grant of the Dynasty Foundation.

%% file: Appendix.tex
\appendix
\section{Appendix}

In this Appendix we prove the following claim which was stated in the proof of Theorem~\ref{th:traces-chain-rule}:

\begin{claim}\label{measurablity-of-traces}
There exist a $\Le^1$-negligible set $\hat N\subset \RR$ and Borel functions $\hat u^\pm$ such that $u^\pm=\hat u^\pm$ on~$E \setminus H^{-1}(\hat N)$.
\end{claim}

The proof will be based on several auxiliary statements and some standard facts, which we would like to recall for completeness.

For any $X\subset \RR^2$
let $\ss F(X)$ denote the set of nonempty compact subsets of $X$
and let us endow $\ss F(X)$ with Hausdorff metric $d_\cc{H}$.
It is well-known that if $X$ is compact then $\ss F(X)$ is also compact.
Moreover, the subclass $\ss F_c(X)$ of \emph{connected} elements of $\ss F(X)$ is closed (see e.g. \cite{Falconer}, Theorems 3.16 and 3.18).
Finally, according to the well-known Golab's theorem the map $C \mapsto \Ha^1(C)$ is lower semicontinuous on $\ss F_c(X)$.

\begin{lemma}[See \cite{ABC2}, Lemma 6.2]\label{l:family-projection}
Suppose that $X\subset \RR^2$ is a closed set and a family $\ss C \subset \ss F(X)$ is closed in $\ss F(X)$.
Then the set $S:=\cup_{C\in \ss C} C$ is closed.
\end{lemma}
This lemma follows directly from the compactness of $\ss F(X)$. We refer to \cite{ABC2} for the details.

\begin{lemma}\label{lem:rigidity}
Suppose that $\gamma_n \colon [0,\ell_n]\to \RR^2$ and $\gamma\colon [0,\ell]\to \RR^2$ are positively oriented natural parametrizations of closed simple Lipschitz curves $C_n \subset \RR^2$ and $C \subset \RR^2$ respectively. Suppose that $C_n \to C$ in $d_\cc{H}$ and $\Ha^1(C_n) \to \Ha^1(C)$ as $n\to \infty$. If $\gamma_n(0) \to \gamma(0)$ as $n\to \infty$ then for any $\tau\in (0,\ell)$ we have 
$\gamma_n([0,\tau]) \to \gamma([0,\tau])$ in $d_\cc{H}$ as $n\to \infty$.
\end{lemma}
\begin{proof}
We define $\wave\gamma_n(t):=\gamma_n(t \cdot \ell_n/\ell)$. Since $\ell_n/\ell \to 1$ as $n \to \infty$, using Arzela-Ascoli theorem we conclude (without renumbering) that $\wave \gamma_n \to \wave \gamma$ in $C([0,\ell])$ for some $1$-Lipschitz function $\wave\gamma$. Since homotopy preserves orientation $\wave \gamma$ is positively oriented.
Since $\wave\gamma([0,\ell]) =\lim_{n\to \infty}\wave\gamma_n([0, \ell]) =\lim_{n\to \infty}\gamma_n([0, \ell_n]) = \gamma([0,\ell])=C$ (here limit is taken w.r.t. $d_\cc{H}$), it remains to prove that $\wave \gamma$ is injective.
Suppose that $\wave\gamma(a) = \wave \gamma(b)$ for some $a,b \in [0, \ell)$ with $a<b$. Then $(\RR/(\ell \ZZ)) \setminus (\{a\}\cup\{b\})$ has two connected components: $I_1$ and $I_2$. Since $\gamma$ is positively oriented, at least for some $i\in 1,2$ we have $\gamma(\overline{I_i}) = C$. But then, since $\gamma$ is $1$-Lipschitz, we have
$\Ha^1(C) \le \int_{I_i} |\gamma'| \, d\tau \le \max(|b-a|,\ell - |b-a|) < \ell$ which contradicts the assumptions.
\end{proof}

The following lemma is a variant of Proposition 6.1 from \cite{ABC2}:

\begin{lemma} \label{l:sigma-closed-family}
Suppose that $H\colon \RR^2 \to \RR$ is a compactly supported monotone Lipschitz function.
Then for any negligible set $N'\subset \RR$ there exists a negligible Borel set $N \supset N'$ such that
\begin{itemize}
\item any $h\in \RR \setminus N$ is a regular value of $H$;
\item the family 
\begin{equation*}
\ss C := \setof{H^{-1}(h)}{h \in \RR \setminus N}.
\end{equation*}
is $\sigma$-compact in $\ss F(\supp H)$;
\item there exist compact families $\ss C_j \subset \ss C$ such that $\ss C = \cup_{j\in \NN} \ss C_j$
 and the map $C \mapsto \Ha^1(C)$ is continuous on $\ss C_j$ (with respect to convergence in $d_{\cc H}$) for each $j\in \NN$;
\item the set $E:= \cup_{C \in \ss C} C$ is Borel.
\end{itemize}
\end{lemma}
\begin{proof}
By Theorem~\ref{th:ABC2} there exists a negligible set $N_0 \subset \RR^2$ such that for any $h\notin N_0$
the corresponding level set $E_h:=H^{-1}(h)$ is regular.
From coarea formula we know that the map $f\colon h \mapsto \Ha^1(E_h)$ is Borel. Hence there exists a sequence of compacts $K_j\subset \RR \setminus N_0$ such that for any $j \in \NN$ we have $K_j \subset K_{j+1}$, the restriction of $f$ on $K_j$ is continuous and $\Le^1(\RR\setminus \cup_{j\in \NN} K_j) = 0$.

Let us define $\ss C_j := \setof{E_h}{h \in K_j}$.
We claim that $\ss C_j$ is compact in $\ss F(\supp H)$ and $\Ha^1$ is continuous on $\ss C_j$.

Indeed, for any sequence $C_n \in \ss C_j$, $n\in \NN$, there exists a subsequence (we omit renumbering) such that $C_n \to C$ in $d_\cc{H}$ as $n\to \infty$, where $C\subset \supp H$ is some connected compact. 

Since $h_n := H(C_n) \in K_j$ and $H$ is continuous there exists $h\in K_j$ such that $h_n \to h$ as $n \to \infty$ and $H=h$ on $C$. Hence $C\subset E_h$.

By definition of regular value $\nabla H \ne 0$ $\Ha^1$-a.e. on $E_h$, hence $\Ha^1$-a.e. point of $E_h$ is an accumulation point of $E_{h_n} \equiv C_n$. Hence $\Ha^1(E_h \setminus C)=0$. Given that $E_h$ is a closed simple curve and $C\subset E_h$ is connected we conclude that $C=E_h$.

Since $f$ is continuous on $K_j$ we also have $\Ha^1(C) = \lim_{n\to \infty} \Ha^1(C_n)$.

Finally, since $H$ is continuous the set $H^{-1}(\cup_{j \in \NN} K_j)$ is $\sigma$-compact and in particular Borel.
\end{proof}


Suppose that a compactly supported Lipschitz function $H\colon \RR^2 \to \RR$ is decomposed as $H=\sum_{i\in \NN} H_i$
where $H_i \colon \RR^2 \to \RR$ are compactly supported and $|\nabla H_i| \Le^2 \perp |\nabla H_j| \Le^2$ when $i \ne j$.

Let $E^i_h := H_i^{-1}(h)$ and $E_h := H^{-1}(h)$ denote the level sets of $H_i$ and $H$ respectively.

\begin{lemma} \label{l:E_i and E_j are disjoint}
Suppose $i\ne j$.
Then for a.e. $h \in \RR$ the function $H_j$ is constant on $E^i_h$
and moreover $h_j := H_j(E^i_h)$ is not a regular value of $H_j$.
\end{lemma}
\begin{proof}
Since $|\nabla H_i| \Le^2 \perp |\nabla H_j| \Le^2$ by coarea formula we have
\begin{equation*}
0 = \int_{\RR^2} |\nabla H_i| \cdot |\nabla H_j| \, dx
= \int_\RR \int_{E^i_h} |\nabla H_j| \, d\Ha^1 \, dh
\end{equation*}
Therefore for a.e. $h\in \RR$ for $\Ha^1$-a.e. $x\in E^i_h$ there exists $\nabla H_j(x) = 0$. Let $\gamma \colon I \to \RR^2$ denote natural parametrization of $E^i_h$ (without loss of generality we can assume that $h$ is a regular value of $H_i$). Then
for $\Le$-a.e. $t\in I$ the function $t \mapsto H_j(\gamma(t))$ is differentiable (recall that $\Ha^1 \rest E^i_h = \gamma_\# \Le$) and
\begin{equation*}
(H_j(\gamma(t)))' = \nabla H_j(\gamma(t)) \cdot \gamma'(t) = 0
\end{equation*}
which means that $H_j$ is equal to some constant $h_j$ on $E^i_h$. Moreover, since $\nabla H_j = 0$ $\Ha^1$-a.e. on $E^i_h$ this constant $h_j$ cannot be a regular value of $H_j$.
\end{proof}

\begin{lemma} \label{l:H is is constant on E^i_h}
For any $i \in \NN$ and a.e. $h \in \RR$ the function $H$ is constant on $E^i_h$.
\end{lemma}
\begin{proof}
Observe that 
\[
\div(H \nabla^\perp H_i) = \nabla H \cdot \nabla^\perp H_i = \nabla H_i \cdot \nabla^\perp H_i=0
\]
a.e. in $\RR^2$. Hence by Theorems~\ref{th:div-disint} and~\ref{th:eq-along-curve} we obtain that for a.e. $h\in \RR$
\[
H' = 0 \quad \text{in} \quad \ss D'(E^i_h, \gamma)
\]
where $\gamma \colon I \to \RR^2$ denotes natural parametrization of a (regular) level set $E^i_h$.
\end{proof}


\noindent\textbf{Construction of $\ss C_i$ and $E_i$.} \label{Construction of E_i}
In view of lemmas~\ref{l:E_i and E_j are disjoint} and~\ref{l:H is is constant on E^i_h}
we can find negligible sets $N_i'\subset \RR$ such that applying 
lemma~\ref{l:sigma-closed-family} for $H_i$ and $N_i'$ we obtain $\sigma$-closed families $\ss C_i \subset \ss F(\supp H_i)$ and negligible sets $N_i$ with the following properties:
\begin{itemize}
\item for any $h \notin N_i$ the function $H$ is constant on $E^i_h$;
\item the sets $E_i=\cup_{C \in \ss C_i} C$ are Borel and pairwise disjoint;
\item $|\nabla H_i| \Le^2 = |\nabla H_i| \Le^2 \rest E_i$.
\end{itemize}
Then the connected components of almost all regular level sets of $H$ are the level sets of $H_i$:
\begin{lemma}\label{l:level sets of H are the level sets of H_i}
There exists a negligible set $N\subset \RR$ such that
\begin{itemize}
\item any $h\notin N$ is a regular value of $H$;
\item for any $C \in \ff{Conn^*}(E_h)$ there exists unique $i\in \NN$ such that $C \in \ss C_i$.
\end{itemize}
\end{lemma}

\begin{proof}
Consider the set $J$ of regular values $h$ of $H$ for which $\Ha^1(E_h \setminus \cup_i E_i)>0$. Since $\nabla H = 0$ a.e. on $\RR^2 \setminus \cup_i E_i$ by coarea formula we immediately obtain $\Le^1(J)=0$.

Let $h$ be a regular value of $H$ and fix $C\in \ff{Conn^*}(E_h)$. If $h \notin J$ then $C \cap E_i \ne \emptyset$ for some $i\in \NN$. Let $C_i$ denote the level set (w.r.t. $H_i$) of some $x \in C \cap E_i$. By lemma~\ref{l:H is is constant on E^i_h} $H$ is constant on $C_i$, hence $C_i \subset C$. If $C$ contains a point which does not belong to $C_i$ then $C$ has a triod, which is not possible since $h$ is a regular value. Hence $C=C_i$.
\end{proof}

\begin{proof}[Proof of Claim~\ref{measurablity-of-traces}]
\emph{Step 0.} If $\Omega \ne \RR^2$ then one can extend $H$ outside of $\Omega$ in such a way that the resulting function is Lipschitz and compactly supported (see proof of Theorem~\ref{th:ABC2}). Therefore in the rest of the proof we will assume that $\Omega = \RR^2$, omitting minor modifications needed in the general case.

\emph{Step 1.} By Theorem~\ref{th:decomposition} there exists a countable family $\{\hat H_i\}_{i\in \NN}$ of compactly supported Lipschitz functions $\hat H_i \colon \RR^2\to \RR$ such that $H = \sum_{i=1}^\infty \hat H_i$ and $|\nabla \hat H_i| \Le^2 \perp |\nabla \hat H_j| \Le^2$ whenever $i\ne j$. Let $\hat H_i^+(x) := \max(\hat H_i(x), 0)$ and $\hat H_i^-(x) := \min(\hat H_i(x),0)$.
Clearly $|\nabla \hat H_i^+| \Le^2 \perp |\nabla \hat H_i^-| \Le^2$. Thus we have decomposed the function $H$ into a sum $H=\sum_{i\in \NN} H_i$ of monotone Lipschitz functions with constant sign and satisfying $|\nabla H_i| \Le^2 \perp |\nabla H_j| \Le^2$, $i\ne j$.

\emph{Step 2.} Using the decomposition of $H$ from the previous step we construct the families $\ss C_i$ and the sets $E_i$ as above (see p.~\pageref{Construction of E_i}).
Let $N$ be given by lemma~\ref{l:level sets of H are the level sets of H_i}.

\emph{Step 3.} Let us fix $i\in \NN$. We are going to define the functions $\hat u^\pm$ on $E_i$.
Let us assume that $H_i \ge 0$ on $\RR^2$ (the argument will be the same when $H_i \le 0$ on $\RR^2$).

For any regular value $h$ of $H$ and any nontrivial connected component $C\in \ff{Conn^*}(E_h)$ let $\gamma_C\colon I_C \to \RR^2$ denote a natural parametrization of $C$ which agrees with $B$.

In view of \eqref{divub-on-curve} and Theorem~\ref{th:eq-along-curve} the function $u$ has traces $u_{C,B}^\pm$
along $(C,B)$. By definition
\begin{equation*}
u^+_{C,B}(x) = \lim_{\tau \to +0} \frac{1}{\tau} \int_{t_x}^{t_x+\tau} u(\gamma_C(\xi)) \, d\xi
\end{equation*}
for any $x\in C$, where $t_x$ is such that $\gamma_C(t_x)=x$.

For any $x\in E_i$ let us define
\begin{equation*}
\hat u^+_\tau(x) := \1_{\RR \setminus N}(H(x)) \frac{1}{\tau} \int_{E_i} \chi_\tau(x,y) u(y) \,d \Ha^1(y)
\end{equation*}
where $\chi_\tau$ is the characteristic function of the set
\begin{equation*}
S_\tau^i := \setof{(x,y)\in E_i \times E_i}{x = \gamma_C(t), \;\; y\in \gamma_C([t,t+\tau]) \;\; \text{for some } h\in \RR \setminus N}.
\end{equation*}

Since $(\gamma_C)_\# \Le = \Ha^1 \rest C$ we have $u^+_{C, B}(x) = \lim_{\tau \to +0} \hat u^+_\tau(x)$
provided that $x \in E_i \setminus H^{-1}(N)$.
On the other hand if $x \in E_i \cap H^{-1}(N)$ then $\hat u_\tau^+(x) = 0$.
Therefore for any $x\in E_i$ there exists
\begin{equation*}
\hat u^+(x) := \lim_{\tau \to 0} \hat u_\tau^+(x).
\end{equation*}
We are going to prove that this function is Borel. To do this it is sufficient to show that $S_\tau$ is Borel for any $\tau>0$. Our argument will be similar to the proof of Lemma~\ref{l:sigma-closed-family} in which we proved that $E_i$ is Borel.

By Lemma~\ref{l:sigma-closed-family} the family $\ss C_i$ (recall that $E_i = \cup_{C\in \ss C_i} C$) can be written as $\ss C_i = \cup_{j\in \NN} \ss C_{ij}$, where $\ss C_{ij}$ is closed in $\ss F(\supp H_i)$ family of closed simple curves such that the map $C \mapsto \Ha^1(C)$ is continuous on $\ss C_i$.

For any $j\in \NN$ let 
\begin{equation*}
S_\tau^{ij} := \setof{(x,y)\in E_i \times E_i}{x = \gamma_C(t), \; y\in \gamma_C([t,t+\tau]) \; \text{ for some } C \in \ss C_i}.
\end{equation*}
Since $H_i\ge 0$ on $\RR^2$ and $H_i$ is monotone, one can show that for any $C \in \ss C_i$
the parametrization $\gamma_C$ is \emph{negatively oriented}.
Then by Lemma~\ref{lem:rigidity} the set $S_\tau^{ij}$ is closed.
Therefore $S_\tau$ is a countable union of closed sets and hence Borel.
The construction of $\hat u^-$ is analogous to the construction of $\hat u^+$ above.
The functions $\hat u ^\pm$ are defined on $\cup_{i\in \NN} E_i$ and agree with the traces of $u$ along $(C,B)$ for any $C\in \ff{Conn^*}(E_h)$ for any $h\in \RR \setminus N$.
\end{proof}